\documentclass[11pt]{article}
\usepackage[a4paper,margin=1in]{geometry}
\usepackage[T1]{fontenc}
\usepackage{lmodern}
\usepackage{amsmath,amssymb,amsthm,mathtools,bm,mathrsfs,bbm}
\usepackage{microtype}
\usepackage{xcolor}
\usepackage{graphicx,booktabs,array,float,placeins,caption}
\usepackage{authblk}
\usepackage[round,authoryear]{natbib}
\usepackage[hidelinks]{hyperref}
\numberwithin{equation}{section}
\newtheorem{theorem}{Theorem}[section]
\newtheorem{lemma}[theorem]{Lemma}
\newtheorem{proposition}[theorem]{Proposition}
\newtheorem{corollary}[theorem]{Corollary}
\newtheorem{example}[theorem]{Example}
\theoremstyle{remark}

\newcommand{\Var}{\operatorname{Var}}

\newcommand{\tr}{\operatorname{tr}}
\newcommand{\diag}{\operatorname{diag}}
\newcommand{\dto}{\xrightarrow{d}}
\newcommand{\pto}{\xrightarrow{\mathbb{P}}}

\newcommand{\DAN}{\mathrm{DAN}}
\allowdisplaybreaks[2]
\title{When is Hotelling's $T^2$ asymptotically $F$-distributed?}
\author[1,2]{Yanpeng Li\thanks{Email: \texttt{20230256@hit.edu.cn}}}
\author[3]{Guangming Pan\thanks{Email: \texttt{gmpan@ntu.edu.sg}}}
\author[4]{Wang Zhou\thanks{Email: \texttt{wangzhou@nus.edu.sg}}}
\affil[1]{\small School of Mathematics, Harbin Institute of Technology, China}
\affil[2]{\small Faculty of Science, University of Macau, China}
\affil[3]{\small Division of Mathematical Sciences, Nanyang Technological University, Singapore}
\affil[4]{\small Department of Statistics and Data Science, National University of Singapore, Singapore}

\date{}
\begin{document}
\maketitle
\begin{abstract}
For Gaussian observations, a deterministic multiple of Hotelling's $T^2$ statistic has an exact $F$ distribution. This paper investigates when this approximation remains valid for an independent and identically distributed (i.i.d.) data matrix $\mathbf{X}\in \mathbb{R}^{p\times n}$ with a fixed, possibly heavy-tailed entry law $\xi$. Under symmetry and $\limsup_{n\rightarrow \infty} p/n<1$, we prove that approximation by the $F$ distribution in Kolmogorov distance is equivalent to belonging to the domain of attraction of the normal law. Both conditions are equivalent to relative concentration of the diagonal entries of the null-space projection $\mathbf{P}=\mathbf{I}_n-\mathbf{X}^{\top}(\mathbf{X}\mathbf{X}^{\top})^{-1}\mathbf{X}$. The central quantity $R=\mathbf{1}^{\top}\mathbf{P}\mathbf{1}$ measures the
geometric alignment between the averaging direction and the random row space of $\mathbf{X}$. For symmetric finite-variance entries, the sufficient approximation extends to every sequence $1\le p<n$, including fixed nullity $n-p$. As a consequence of the $F$ approximation, Hotelling's $T^2$ statistic exhibits different asymptotic distributions in different regimes, including the normal, chi-square, and inverse-chi-square limits depending on $p$ and $n$. Without symmetry, we establish a central limit theorem under the weak-$5/2$ tail condition $\lim_{x\to\infty}x^{5/2}\mathbb{P}(|\xi|>x)=0$ when $p/n\rightarrow \gamma\in (0,1)$, and obtain the $F$ approximation as a consequence. Centered standardized Pareto laws show that the exponent $5/2$ is sharp for this universal absolute-tail sufficient condition, although balanced-tail nonsymmetric laws can retain the approximation at the boundary, which reveals that the asymmetry of the tails matters for the $F$ approximation.
The main technical ingredients are a novel representation of the quadratic form associated with the averaging direction $\mathbf{1}^{\top}\mathbf{P}\mathbf{1}$, a conditional sign comparison under symmetry, and, in the absence of symmetry, a projection recursion identity and a row replacement comparison with a symmetrized law.
\end{abstract}


\section{Introduction}\label{sec:introduction}

Hotelling's $T^2$ statistic is a fundamental tool in multivariate statistical inference and has been widely used for hypothesis testing and confidence region construction; see, for example, \citet{Hotelling1931} and \citet{Anderson2003}. Let $\mathbf{x}_1,\ldots,\mathbf{x}_n\in\mathbb{R}^p$ be independent observations from $N_p(\mathbf{0},\boldsymbol{\Sigma})$, where $\boldsymbol{\Sigma}$ is positive definite and $1\le p<n$. Define
\begin{equation}
 \bar{\mathbf{x}}=\frac1n\sum_{j=1}^n\mathbf{x}_j,\qquad
 \mathbf{S}=\frac1n\sum_{j=1}^n(\mathbf{x}_j-\bar{\mathbf{x}})(\mathbf{x}_j-\bar{\mathbf{x}})^\top,
\end{equation}
and let
\begin{equation}
 T_n^2=n\bar{\mathbf{x}}^\top\mathbf{S}^{-1}\bar{\mathbf{x}}.
\end{equation}
Under these assumptions, the exact distribution of Hotelling's $T^2$ is given by
\begin{equation}\label{eq:intro-exact-F}
 \frac{n-p}{np}T_n^2\stackrel d=F_{p,n-p},\qquad
 F_{p,m}:=\frac{B_p/p}{A_m/m},
\end{equation}
where $A_m\sim\chi_m^2$ and $B_p\sim\chi_p^2$ are independent random variables. The factor in \eqref{eq:intro-exact-F} corresponds to the divisor $n$ in the definition of $\mathbf{S}$. In particular, if both $p$ and $n-p$ tend to infinity, the central limit theorem for the two chi-square variables gives
\begin{equation}
 \frac{(n-p)T_n^2/(np)-1}{\sqrt{2/p+2/(n-p)}}\dto N(0,1).
\end{equation}
Thus the classical Gaussian model supplies not only a normal approximation, but also an exact reference distribution for every $1\le p<n$.

For non-Gaussian observations, the validity of this reference depends on the distribution of the data and on the dimension. In the case $p=1$, \citet{GGM1997} proved that the Student $t$-statistic is asymptotically standard normal if and only if the underlying centered law belongs to the domain of attraction of the normal law, abbreviated as DAN. For fixed $p$, a related necessity result for the multivariate Student statistic based on symmetric random vectors was obtained by \citet{GineGotze2004}. In a high-dimensional setting, with $p/n\to\gamma\in(0,1)$, \citet{PanZhou2011} established a central limit theorem for Hotelling's $T^2$ statistic under a finite fourth moment. The limiting variance does not depend on the fourth moment of the entries. This suggests that the moment assumption may be stronger than necessary, but it does not identify the distributional condition that governs the approximation.

Heavy tails make this question substantially different from a direct application of the finite-fourth-moment theory. For example, extreme eigenvalues of sample covariance matrices may be determined by a small number of large entries; see \citet{auffinger2009poisson}. On the other hand, self-normalization can preserve Gaussian fluctuations under weaker tail conditions, as shown for sample correlation matrices by \citet{HeinyParolya2024,Li2024clt,LiLiuXieZhou2026}. These results indicate that the relevant condition depends on the structure of the statistic and inspire us to ask: Does the central limit theorem for Hotelling's $T^2$ statistic remain valid under weaker tail assumptions?

The Gaussian case also suggests that a normal limit alone is not the most natural formulation of the problem. For fixed $p$, the Gaussian Hotelling statistic tends to $\chi_p^2$, whereas a fixed difference $n-p$ leads to an inverse-chi-square limit after a different normalization. Consequently, we study the approximation by the $F$ distribution itself. For a data matrix $\mathbf{X}$ with independent and identically distributed entries, define
\begin{equation}\label{eq:intro-error}
 \varepsilon_{n,p}
 :=\sup_{t\in\mathbb{R}}\left|
 \mathbb{P} \left(\frac{n-p}{np}T_n^2\le t\right)
       -\mathbb{P}(F_{p,n-p}\le t)\right|.
\end{equation}
The question is therefore: under what conditions on the fixed entry law does $\varepsilon_{n,p}\to0$ as the sample size and dimension vary?

With a novel representation of Hotelling's $T^2$ statistic, we relate it to the projection representation $\mathbf{1}^{\top}\mathbf{P}\mathbf{1}$, which connects tail conditions on the entries with both fixed-direction fluctuations and the coordinate geometry of random subspaces.
Under symmetry, we show that for every dimension sequence satisfying $\limsup_{n\rightarrow \infty} p/n<1$, membership in DAN is equivalent to $\varepsilon_{n,p}\to0$. Moreover, both are equivalent to relative concentration of the diagonal entries of the null-space projection of the data matrix. The result is therefore not merely a reduction of a moment assumption: it identifies a geometric property of a random subspace that characterizes normal attraction through Hotelling's statistic. Under a finite second moment, we also obtain the sufficient $F$ approximation for every $1\le p<n$, including the near-square regime with bounded $n-p$.

The role of symmetry becomes apparent when it is removed. For $p/n\rightarrow \gamma\in (0,1)$, we establish asymptotic normality with the Gaussian centering, and consequently the $F$ approximation, under the weak-$5/2$ tail condition
\begin{equation}
 \lim_{x\to\infty}x^{5/2}\mathbb{P}(|\xi|>x)=0.
\end{equation}
The statistic thus admits a weaker absolute-tail sufficient condition than the weak-third-tail condition appearing in related high-dimensional CLTs for sample correlation matrices of \citet{Li2024clt} and \citet{LiLiuXieZhou2026}. 
The exponent $5/2$ is critical for a sufficient condition that is to hold for arbitrary nonsymmetric laws. For a centered standardized Pareto law with index $2<\alpha<5/2$, the $F$ approximation fails because of the accumulation of off-diagonal mean bias, which is of order $n^{3-\alpha}$ and reaches the $\sqrt n$ fluctuation scale of the Gaussian reference at $\alpha=5/2$. Moreover, a law with a bounded asymmetric component and exactly symmetric heavy tails can still have the $F$ approximation below this threshold. The relevant distinction is the signed tail imbalance: an absolute-tail condition controls it uniformly but does not characterize it for each individual law.

The rest of the paper is organized as follows. Section~\ref{sec:main} gives the definitions, states the results, and outlines the proof strategy. Sections~\ref{sec:proofmainresults} and \ref{sec:nonsymmetry} present the proofs of the symmetric and nonsymmetric results, respectively. Appendix~\ref{app:lemma-proofs} contains the proofs of the auxiliary lemmas. The general auxiliary lemmas are collected in Appendix~\ref{app:auxiliary}.   Examples~\ref{thm:regular-boundary} and  \ref{rem:balanced-tails} are verified in Appendix~\ref{sec:verification}.

\textbf{Notation.} Throughout the paper, $C,c>0$ denote constants whose values may change from line to line. The vectors $\mathbf{1}$ and $\mathbf{0}$ have the dimension indicated by the surrounding expression, and $\mathbf{e}_j$ denotes a coordinate vector. For a matrix $\mathbf{A}$, $\|\mathbf{A}\|$ and $\|\mathbf{A}\|_{\mathrm{F}}$ denote its operator and Frobenius norms. For real symmetric matrices, $\mathbf{A}\preceq\mathbf{B}$ means that $\mathbf{B}-\mathbf{A}$ is positive semidefinite; the notation $\mathbf{A}\prec\mathbf{B}$ means that this difference is positive definite. We use the notation $a\wedge b=\min(a,b)$ and $a\vee b=\max(a,b)$ for $a,b\in \mathbb{R}$.
We use $\mathrm{O}$ and $o$ for deterministic orders and $\mathrm{O}_{\mathbb{P}}$ and $o_{\mathbb{P}}$ for orders in probability. For a deterministic sequence $a_n>0$, $U_n=o_{\mathbb P}(a_n)$ means that $U_n/a_n\to0$ in probability. We write $a_n\sim b_n$ if $a_n/b_n\rightarrow 1$ as $n\rightarrow \infty$.
An expectation restricted to an event $E$ is always written as $\mathbb{E}[U\mathbbm{1}_E]$. For the filtration generated by the first $i$ rows, we write $\mathbb{E}_i(U)=\mathbb{E}(U\mid\mathcal{F}_i)$ and $\Var_i(U)=\mathbb{E}_i[(U-\mathbb{E}_iU)^2]$. A subscript naming an independent random vector, such as $\mathbb{E}_{\mathbf{x}}$ or $\mathbb{E}_{\boldsymbol{\varepsilon}}$, indicates integration with respect to that vector while the remaining quantities are held fixed. For $q>0$, we write $\|X_n\|_{L^q}=(\mathbb E|X_n|^q)^{1/q}$. An $L^q$ bound refers to a bound on this quantity, and $X_n\to0$ in $L^q$ means that $\|X_n\|_{L^q}\to0$ as $n\to\infty$.

\section{Main results}\label{sec:main}

\subsection{Basic definitions}\label{sec:definitions}

Let $\mathbf{X}=(X_{ij})_{1\le i\le p,\,1\le j\le n}$ have independent entries with a fixed nondegenerate real law $\xi$, where $1\le p<n$. Its columns are denoted by $\mathbf{x}_j=(X_{1j},\ldots,X_{pj})^\top$, for $1\le j\le n$. Set
\begin{equation}\label{eq:model}
 \begin{gathered}
 m=n-p,\qquad\gamma_n=p/n,\qquad
 \mathbf{H}=\mathbf{I}_n-n^{-1}\mathbf{1}\mathbf{1}^\top,\\
 \bar{\mathbf{x}}=n^{-1}\mathbf{X}\mathbf{1},\qquad
 \mathbf{S}=n^{-1}\mathbf{X}\mathbf{H}\mathbf{X}^\top.
 \end{gathered}
\end{equation}
The Hotelling $T^2$ statistic is defined by
\begin{equation}
 T_n^2=n\bar{\mathbf{x}}^\top\mathbf{S}^{-1}\bar{\mathbf{x}},
\end{equation}
and we set $T_n^2=0$ when $\mathbf{S}$ is singular.  Invariance under nonsingular deterministic linear transformations extends the conclusions to observations of the form $\mathbf{A}\mathbf{x}_j$ with nonsingular $\mathbf{A}$.

Let $\mathbf{P}=\mathbf{P}(\mathbf{X})$ be the orthogonal projection onto $\ker\mathbf{X}$. On the full-row-rank event $\mathcal{F}=\{\operatorname{rank}\mathbf{X}=p\}$, it has the representation $\mathbf{P}=\mathbf{I}_n-\mathbf{X}^\top(\mathbf{X}\mathbf{X}^\top)^{-1}\mathbf{X}$. Define
\begin{equation}\label{eq:PQR}
 \mathbf{L}=\mathbf{I}_n-\mathbf{P},\qquad R=\mathbf{1}^\top\mathbf{P}\mathbf{1},
 \qquad Q=\mathbf{1}^\top\mathbf{L}\mathbf{1}=n-R,\qquad v=mp/n.
\end{equation}
For any real matrix $\mathbf{A}$ with $r$ rows and $N>r$ columns, let $\mathcal{F}(\mathbf{A})=\{\operatorname{rank}\mathbf{A}=r\}$ and let $\mathcal{G}(\mathbf{A})$ denote the event that adjoining $\mathbf{1}_N^\top$ gives rank $r+1$. For the original matrix, $\mathcal{F}=\mathcal{F}(\mathbf{X})$ and the augmented-full-rank event is defined by
\begin{equation}\label{eq:def_fullrank}
 \mathcal{G}=\left\{\operatorname{rank}\begin{pmatrix}\mathbf{1}^\top\\\mathbf{X}\end{pmatrix}=p+1\right\}.
\end{equation}
We write $\eta_n=\mathbb{P}(\mathcal{G}^c)$.
These events will be used repeatedly below. The event $\mathcal{F}$ specifies full row rank of $\mathbf{X}$, whereas $\mathcal{G}\subseteq\mathcal{F}$ is precisely the event that $\mathbf{S}$ is invertible. For a truncated matrix or a row or column submatrix, the matrix is included in the event notation. On $\mathcal{G}$, we have the identity (see Section \ref{sec:strategy} below for details)
\begin{equation}\label{eq:identity}
 1+T_n^2/n=n/R,
\end{equation}
which is a key ingredient in the $F$-distribution approximation. The quantity $R=\mathbf{1}^{\top}\mathbf{P}\mathbf{1}=\|\mathbf{P}\mathbf{1}\|^2=\operatorname{dist}^2\big(\mathbf{1},\operatorname{row}(\mathbf{X})\big)$ measures the squared length of the projection of the constant direction $\mathbf{1}$ onto the null space of $\mathbf{X}$. In the rotationally invariant Gaussian case, the
row space is uniformly distributed on the Grassmannian, so this quantity
reduces to a classical random-projection statistic. 

For the independent chi-square variables in \eqref{eq:intro-exact-F}, set $G_{p,m}=B_p/A_m=(p/m)F_{p,m}$. For Gaussian entries, it follows that
\begin{equation}
 R^{\rm G}\stackrel d=\frac{nA_m}{A_m+B_p},\qquad
 T_n^{2,\rm G}/n\stackrel d=G_{p,m}.
\end{equation}
For real random variables $U,V$, the Kolmogorov distance is defined as
\[
d_K(U,V)=\sup_t|\mathbb{P}(U\le t)-\mathbb{P}(V\le t)|.
\]
The error in \eqref{eq:intro-error} can equivalently be written as
\begin{equation}
 \varepsilon_{n,p}
 =d_K \left(\frac{m}{np}T_n^2,F_{p,m}\right)
 =d_K(T_n^2/n,G_{p,m}).
\end{equation}
Thus $F_{p,m}$ always denotes the usual $F$ distribution, and $G_{p,m}$ is used only for the unscaled chi-square ratio.

For a symmetric law, $\xi\in\DAN$ means that $a_n^{-1}\sum_{j=1}^n\xi_j\dto N(0,1)$ for some positive sequence $a_n$. Equivalently, the function $\ell(t)=\mathbb{E}[\xi^2\mathbbm{1}_{\{|\xi|\le t\}}]$ is slowly varying at infinity. On $\mathcal{F}$, one has $\operatorname{tr}\mathbf{P}=m$. Since the law and the event $\mathcal{F}$ are invariant under column permutations,
\[
 \mathbb{E}\big[P_{jj}\mathbbm{1}_{\mathcal{F}}\big]
 =\frac mn\mathbb{P}(\mathcal{F}),\qquad 1\le j\le n.
\]
We define
\begin{equation}\label{eq:delta}
 \delta_{n,p}
 =\frac{\sum_{j=1}^n(P_{jj}-m/n)^2}{p(1-p/n)}
 =\frac{\sum_{j=1}^n(L_{jj}-p/n)^2}{p(1-p/n)},
\end{equation}
and set $\delta_{n,p}=0$ on $\mathcal{F}^c$. This quantity measures the diagonal concentration relative to the rank scale. Since $\tr\mathbf{L}=p$ and $0\le L_{jj}\le1$ on $\mathcal{F}$,
\[
 0\le\sum_{j=1}^n(L_{jj}-p/n)^2=\sum_{j=1}^nL_{jj}^2-p^2/n\le p-p^2/n,
\]
so that $0\le\delta_{n,p}\le1$. Moreover, for Gaussian entries, one has
\begin{equation}\label{eq:Gaussian-moments}
\mathbb{E} R^{\rm G}=m,\qquad
\Var(R^{\rm G})=\frac{2mp}{n+2},\qquad
\mathbb{E}\delta_{n,p}^{\rm G}=\frac{2}{n+2}
\end{equation}
by orthogonal invariance and some elementary algebra.

\subsection{The symmetric case}\label{sec:main-symmetric}

Our main results are stated as follows.
\begin{theorem}\label{thm:DAN}
Suppose that the fixed entry law $\xi$ is nondegenerate and symmetric about zero. If
\begin{equation}\label{eq:aspect-away}
 \limsup_{n\to\infty}p/n<1,
\end{equation}
then the following are equivalent:
\begin{itemize}
	\item[(i)] $\xi\in\DAN$;
	\item[(ii)] $\mathbb{E}\delta_{n,p}\to0$;
	\item[(iii)] $\varepsilon_{n,p}\rightarrow0$.
\end{itemize}
\end{theorem}
The reference distribution $F_{p,n-p}$ depends on both $n$ and $p$, so the theorem characterizes an approximation rather than convergence to a single fixed law. Its geometric condition concerns the diagonals relative to the rank scale. As a complementary result, under a finite second moment, the sufficient approximation also extends to the near-square regime.
\begin{corollary}\label{thm:all-dimensions}
	Let the fixed entry law be nondegenerate, symmetric about zero, and have finite variance. Then, for every sequence $1\le p<n$,
	\begin{equation}
	\varepsilon_{n,p}\rightarrow0.
	\end{equation}
\end{corollary}

The limiting forms of the $F$ reference give the following corollary for symmetric entries.
\begin{corollary}\label{cor:normal}
Suppose that either the hypotheses of Theorem~\ref{thm:DAN} hold with $\xi\in\DAN$, or the hypotheses of Corollary~\ref{thm:all-dimensions} hold. If $p\to\infty$ and $m=n-p\to\infty$, then
\begin{align}
 \frac{R-m}{\sqrt{2mp/n}}&\dto N(0,1),\label{eq:R-CLT}\\
 \sqrt{\frac{m^3}{2pn}}
 \left(\frac{T_n^2}{n}-\frac pm\right)&\dto N(0,1),\label{eq:T-CLT}\\
 \sqrt{\frac{nm}{2p}}
 \big(\log(1+T_n^2/n)+\log(m/n)\big)&\dto N(0,1).
 \label{eq:all-dimension-log}
\end{align}
Under symmetry and \eqref{eq:aspect-away}, if $p\to\infty$, either of \eqref{eq:R-CLT} and \eqref{eq:T-CLT} conversely implies $\xi\in\DAN$. For $p=k$ fixed, Theorem \ref{thm:DAN} gives $T_n^2\dto\chi_k^2$. For $m=k$ fixed, Corollary~\ref{thm:all-dimensions} gives
\begin{equation}
 R\dto\chi_k^2,\qquad T_n^2/n^2\dto1/\chi_k^2.
\end{equation}
\end{corollary}

For $p=1$, the usual signed Student statistic $t_{n-1}$ satisfies $T_n^2=nt_{n-1}^2/(n-1)$ under our convention. Under symmetry, convergence of its square to $\chi_1^2$ therefore recovers the normal limit in \citet{GGM1997}.

\subsection{The nonsymmetric case}\label{sec:main-nonsymmetric}

In growing dimensions, asymmetry may accumulate in the constant direction of the projection. Indeed, column exchangeability gives the exact identity
\begin{equation}
 \mathbb{E}R
 =\mathbb{E}\tr\mathbf{P}+n(n-1)\mathbb{E}P_{12}
 =m+\mathbb{E}\big(p-\operatorname{rank}\mathbf{X}\big)
      +n(n-1)\mathbb{E}P_{12}.
\end{equation}
The rank correction vanishes when $\mathbb{P}(\mathcal{F})=1$ and is exponentially small under \eqref{eq:aspect-away}; see Lemma \ref{lem:rank} for details. Without symmetry, the accumulated off-diagonal mean $n(n-1)\mathbb{E}P_{12}$ may affect the centering even when the rank correction is negligible. We next establish asymptotic normality without symmetry. The $F$ approximation then follows by comparing the statistic with the normal limit of its Gaussian reference.
\begin{theorem}\label{thm:weak-five-halves}
Suppose the fixed entry law satisfies $\mathbb{E}\xi=0$, $\mathbb{E}\xi^2=1$, and
\begin{equation}\label{eq:weak-five-halves}
 \lim_{x\to\infty}x^{5/2}\mathbb{P}(|\xi|>x)=0.
\end{equation}
If $p/n\to\gamma\in(0,1)$, then
\begin{align}
 \sqrt n\big(\log(1+T_n^2/n)+\log(1-\gamma_n)\big)
 &\dto N \left(0,\frac{2\gamma}{1-\gamma}\right),
 \label{eq:log-CLT}\\
 \sqrt n\left(\frac{T_n^2}{n}-\frac{\gamma_n}{1-\gamma_n}\right)
 &\dto N \left(0,\frac{2\gamma}{(1-\gamma)^3}\right).
\end{align}
Equivalently, the standardized limits \eqref{eq:R-CLT}--\eqref{eq:all-dimension-log} hold in this dimension range. Moreover,
\begin{equation}\label{eq:weak-five-halves-mean}
 \mathbb{E}R=m+o(\sqrt n),\qquad \varepsilon_{n,p}\to0.
\end{equation}
\end{theorem}

In particular, for a centered, variance-one law satisfying $\mathbb{P}(|\xi|>x)\sim c x^{-\alpha}\mathcal{L}(x)$ with $c>0$ and a positive slowly varying function $\mathcal{L}(x)$, then the conclusions hold for $\alpha>5/2$, and also for $\alpha=5/2$ when $\mathcal{L}(x)\to0$. The following example shows that a positive tail limit at the boundary already permits failure of the $F$ approximation for a fixed nonsymmetric finite-variance law.

\begin{example}[Failure of $F$ approximation below the boundary]\label{thm:regular-boundary}
For $2<\alpha\le5/2$, let $V_\alpha$ have density
$\alpha v^{-\alpha-1}\mathbbm{1}_{\{v\ge1\}}$, and define
\begin{equation}\label{eq:Pareto}
 \xi_\alpha=\frac{V_\alpha-\mu_\alpha}{\sigma_\alpha},\qquad
 \mu_\alpha=\frac\alpha{\alpha-1},\qquad
 \sigma_\alpha^2=\frac\alpha{(\alpha-1)^2(\alpha-2)}.
\end{equation}
Then $\xi_\alpha$ is centered, has variance one, is non-atomic and
nonsymmetric, and
\[
 x^\alpha\mathbb{P}(|\xi_\alpha|>x)\rightarrow \sigma_\alpha^{-\alpha}
\]
as $x\rightarrow \infty$.
For every $p/n\to\gamma\in(0,1)$, the corresponding approximation
error satisfies
\begin{equation}\label{eq:regular-boundary-separation}
 \begin{cases}
  \varepsilon_{n,p}\rightarrow 1,&2<\alpha<5/2,\\[2pt]
  \displaystyle\liminf_{n\to\infty}\varepsilon_{n,p}>0,&\alpha=5/2.
 \end{cases}
\end{equation}
In the first case, we have $(R-m)/\sqrt n\pto-\infty$.
\end{example}

In view of Theorem \ref{thm:weak-five-halves} and Example \ref{thm:regular-boundary}, the vanishing weak-$5/2$ tail condition is sharp as a sufficient condition for $F$ approximation without symmetry. However, this does not make it necessary for each nonsymmetric law, as the $F$ approximation can still be valid for a nonsymmetric law with symmetric tails below this boundary. The following example shows that the asymmetry on bounded support itself does not affect the $F$ approximation, while the asymmetry in the tails matters.  
\begin{example}[$F$ approximation with symmetric tails below the boundary]\label{rem:balanced-tails}
Fix $12/5<\alpha\le5/2$. Let $S_\alpha$ be symmetric with
$\mathbb{P}(|S_\alpha|>x)=x^{-\alpha}$ for $x\ge1$.
Let $A$ equal $1$ with probability $2/3$ and $-2$ with probability
$1/3$, let $U$ be independent and uniform on $[-1/4,1/4]$, and set
$\eta=A+U$. Define $Z$ by the mixture law
$\frac12\mathcal{L}(S_\alpha)+\frac12\mathcal{L}(\eta)$, and set
\begin{equation}\label{lt:eq:balanced-mixture}
 \xi=Z/\sigma,\qquad
 \sigma^2=\mathbb{E}Z^2
 =\frac{\alpha}{2(\alpha-2)}+\frac{97}{96}.
\end{equation}
It follows that $\xi$ is centered, has variance one, is non-atomic and nonsymmetric. Its asymmetric component has bounded support, while its two heavy tails agree exactly: for all sufficiently large $x$,
\begin{equation}\label{lt:eq:balanced-mixture-tails}
 \mathbb{P}(\xi>x)=\mathbb{P}(\xi<-x)
 =\frac14\sigma^{-\alpha}x^{-\alpha},\qquad
 \mathbb{P}(|\xi|>x)=\frac12\sigma^{-\alpha}x^{-\alpha}.
\end{equation}
Nevertheless, for every $p/n\to\gamma\in(0,1)$, we have
\begin{equation}\label{lt:eq:balanced-mixture-F}
 \frac{R-m}{\sqrt{2mp/n}}\dto N(0,1),\qquad
 \mathbb{E}R=m+o(\sqrt n),\qquad
 \varepsilon_{n,p}\rightarrow 0.
\end{equation}
Hence the condition in
Theorem~\ref{thm:weak-five-halves} is not necessary for an individual
nonsymmetric law. 
\end{example}

\subsection{Proof strategy}\label{sec:strategy}
The proof relies on a geometric representation of Hotelling's statistic with conditional comparisons of projection quadratic forms. The basic tools are the Sherman--Morrison and Woodbury identities, concentration and lower-tail estimates for quadratic forms \citep{Tikhomirov2016,DabagiaFernandez2025}, and Lindeberg replacement \citep{lindeberg1922neue}. 
The main technical ingredients are a comparison that preserves the diagonal cancellation under sign symmetrization (Lemma~\ref{sr:lem:row-mean}), recursive projection-moment bounds and a truncation estimate (Lemmas~\ref{lem:projection-moments} and \ref{sat:lem:truncation}), and a positive two-column identity for the directional bias (Lemma~\ref{lp:lem:conditional-bias}). These ingredients connect the $F$ approximation to diagonal concentration and DAN under symmetry, establish the Gaussian-centered central limit theorem under the weak-$5/2$ condition without symmetry, and explain the criticality of the exponent $5/2$.

\subsubsection{A geometric representation}

Our starting point is the identity \eqref{eq:identity}, which expresses Hotelling's statistic through the projection of the constant direction onto the null space of the data matrix. Let $\mathbf{e}=\mathbf{1}/\sqrt n$, $\mathbf{a}=\mathbf{X}\mathbf{e}$, and $\mathbf{C}=\mathbf{X}\mathbf{X}^\top$. On $\mathcal{G}$, both $\mathbf{C}$ and $\mathbf{C}-\mathbf{a}\mathbf{a}^\top=n\mathbf{S}$ are positive definite. Thus $U=\mathbf{a}^\top\mathbf{C}^{-1}\mathbf{a}<1$, and the Sherman--Morrison formula gives
\[
\frac{T_n^2}{n}
=\mathbf{a}^\top(\mathbf{C}-\mathbf{a}\mathbf{a}^\top)^{-1}\mathbf{a}
=\frac{U}{1-U}
=\frac{1}{\mathbf{e}^\top\mathbf{P}\mathbf{e}}-1
=\frac nR-1,
\]
where $\mathbf{e}^\top\mathbf{P}\mathbf{e}=1-U=R/n$ by \eqref{eq:PQR}. Moreover, on $\mathcal{F}$, the condition $R>0$ is equivalent to $\mathcal{G}$.

This representation reduces the $F$ approximation to a comparison between the directional projection statistic $R=\mathbf{1}^\top\mathbf{P}\mathbf{1}$ and its Gaussian counterpart $R^{\rm G}=\mathbf{1}^\top\mathbf{P}^{\rm G}\mathbf{1}$. The remaining arguments study how the entry law affects this fixed-direction geometry.

\subsubsection{The symmetric cases}
Under symmetry, multiplying each column by an independent Rademacher sign preserves the data law. Consequently,
\begin{equation}\label{eq:strategy-sign}
R\stackrel d=\boldsymbol{\varepsilon}^\top\mathbf{P}\boldsymbol{\varepsilon},
\qquad
\Var_{\boldsymbol{\varepsilon}}
\bigl(\boldsymbol{\varepsilon}^\top\mathbf{P}\boldsymbol{\varepsilon}
\mid\mathbf{P}\bigr)
=2\frac{mp}{n}(1-\delta_{n,p}),
\end{equation}
where the distributional identity is unconditional, and the variance formula holds on $\mathcal{F}$. This converts the fixed direction into independent signs while retaining the random projection. 
Lemma~\ref{lem:sign} provides both the conditional distributional comparison and a fourth-moment bound. The comparison yields the approximation when the relative diagonal concentration error tends to zero. For the converse, the fourth-moment bound gives uniform integrability, so convergence to the Gaussian reference forces convergence of the second moment and hence diagonal concentration. Proposition~\ref{thm:comparison} makes this reduction precise, including the cases of fixed rank and fixed nullity.

We then relate diagonal concentration to DAN. Truncation at the natural normal-attraction scale produces a centered variance-one triangular array with fourth moment $o(n)$. When $p/n\to0$, a Frobenius estimate controls the row-space projection. When $p/n\to\gamma\in(0,1)$, Lemma~\ref{lem:quadratic-LLN} and a fixed regularization parameter yield a scalar equation for the average diagonal of the regularized projection. Since only $o(p)$ rows are affected by truncation in expectation, comparison through the common row spaces transfers the diagonal estimate to the original matrix. Conversely, a law outside DAN admits a subsequence along which sufficiently many rows are dominated by a single coordinate. This contradicts relative diagonal concentration; the factor $n/p$ keeps the argument effective even in fixed dimension.

The near-square case requires finer control of the projection geometry. When $m=n-p$ is small, a squared-diagonal error of order $o(n)$ may be too large relative to $m$. Under finite variance, regularization and the lower-tail estimate in Lemma~\ref{lem:positive-form} instead give
$\max_{1\le j\le n}P_{jj}\to0$.
For fixed nullity, a separate row-exposure argument establishes the required rank property for atomic laws. These estimates are combined with the same sign comparison in Section~\ref{sec:proof-near-square}.

\subsubsection{The nonsymmetric cases}

For nonsymmetric laws, in view of the symmetric cases in Theorem \ref{thm:DAN}, we compare the projection statistic $R$ with its counterpart under entrywise sign symmetrization, conditioning at each row replacement on the common rows. Under the weak tail condition \eqref{eq:weak-five-halves}, we begin with a truncation, smoothing, and standardization argument in Lemma~\ref{lem:five-halves-truncation}. These operations do not affect the asymptotic behavior of \(R\) at the CLT scale; see Lemma~\ref{sat:lem:truncation} and Section~\ref{sec:prf:weak52}.
For further comparison,
motivated by \citet{LiLiuXieZhou2026}, we use the exact row update
\[
R_{i+1}
=R_i-
\frac{(\mathbf{1}^\top\mathbf{P}_i\mathbf{z}_{i+1})^2}
{\mathbf{z}_{i+1}^\top\mathbf{P}_i\mathbf{z}_{i+1}},
\]
where $\mathbf{P}_i$ is determined by the first $i$ rows and $\mathbf{z}_{i+1}$ is independent of $\mathbf{P}_i$. Thus, the comparison boils down to controlling the decrement by adding a new row. 
 
For a projection $\mathbf{J}$ of rank $r\asymp n$, set $\mathbf{u}=\mathbf{J}\mathbf{1}/\sqrt r$ and write the row decrement as $\Delta=rA/V$, where $A=(\mathbf{u}^\top\mathbf{z})^2$ and $V=\mathbf{z}^\top\mathbf{J}\mathbf{z}$. The key step is to split $V=D+O$ into its diagonal and off-diagonal parts and expand $A/V$ around $F=1+D$. The leading diagonal contribution is unchanged under entrywise sign symmetrization $\mathbf{z}\mapsto\mathbf{z}^\circ$, since it depends only on the squared entries, whose expectation therefore cancels exactly.
To estimate the remaining terms, we adapt the Laplace-integral technique used in higher-moment calculations for self-normalized sums \citep{albrecher2007asymptotic}:
\[
F^{-q}=\frac1{\Gamma(q)}
\int_0^\infty t^{q-1}e^{-tF}\,\mathrm{d}t,
\qquad q>0.
\]
This representation retains the diagonal denominator and, by independence, converts mixed expectations into products of rescaled moments. The residual asymmetry is captured by the signed transform
$g_n(t)=\mathbb{E}(\zeta_n e^{-t\zeta_n^2})$.
The uniform estimate $g_n(t)=o(t^{3/4})$ supplied by Lemma~\ref{lem:five-halves-truncation} controls the principal signed term. The same representation estimates the denominator corrections and remainder, giving the conditional mean comparison in Lemma~\ref{sr:lem:row-mean}.

The intermediate matrices have independent entries whose laws may differ across rows. Lemma~\ref{lem:projection-moments} uses a leave-one-column identity and a recursive fourth-moment estimate to control the coordinates of $\mathbf{u}$ uniformly over the intermediate row-law choices,
\[
\mathbb{E}\sum_{j=1}^n  u_j^4\le C/n,
\qquad
\mathbb{E}A^2
=3\mathbb{E}\|\mathbf{u}\|^4+O(k_n/n).
\]
Thus the fourth-moment error enters through $k_n/n\to0$. Projection lower-tail estimates control the denominator, while a subsequential Gaussian-mixture approximation and fourth-moment convergence establish uniform integrability of $A^2$. These ingredients yield $\Delta-A\to0$ in $L^2$ and uniform integrability of $\Delta^2$; see Lemma~\ref{sym:lem:increments}.

We then apply a second-order Lindeberg replacement to smooth test functions of the normalized statistic. The conditional mean comparison gives a decrement-mean error $o(n^{-1/2})$. The conditional second moments of $A$ match exactly under the two laws, and the $L^2$ approximation gives a second-moment error $o(1)$ after averaging over the common rows. At the $\sqrt n$ scale, the first two Taylor terms therefore contribute $o(n^{-1})$ per replacement. Uniform integrability of the squared decrements controls the Taylor remainders, so summing over $O(n)$ rows gives a vanishing total error. The Gaussian-centered limit consequently transfers from the symmetric reference to the comparison array, which further transfers the limit to the original fixed law, and \eqref{eq:identity} yields the Hotelling limits and the $F$ approximation.

\subsubsection{Criticality of the $5/2$ boundary}

The counterexample in Example~\ref{thm:regular-boundary} is governed by the off-diagonal mean of the full projection. For the conditionally truncated Pareto matrix with $d\asymp n$ rows, column exchangeability and the fixed trace give
\[
(n-d)-\mathbb{E}R=n(n-1)\mathbb{E}L_{12}.
\]
The key device is a two-column Woodbury identity: the conditional off-diagonal mean is expressed as a sum of nonnegative squares. The contribution of entries near $\sqrt n$ then gives
$\mathbb{E}L_{12}\ge c n^{1-\alpha}$, producing a directional bias of at least $c n^{3-\alpha}$; see Lemma~\ref{lp:lem:conditional-bias}. Moreover, Lemma~\ref{lem:projection-moments} bounds the variance of the retained statistic by $O(n)$. 

For $2<\alpha<5/2$, the bias dominates the $\sqrt n$ fluctuation scale of the Gaussian reference. Removing exceptional rows preserves a one-sided projection ordering, and conditional Chebyshev estimates yield distance tending to one. At $\alpha=5/2$, the bias is at least a fixed multiple of $\sqrt n$. A bounded monotone truncating function converts this bias and the variance bound into a persistent expectation gap, which gives a strictly positive limiting lower bound for the Kolmogorov distance. We refer to Appendix~\ref{sec:verification} for more details.

Thus $5/2$ is sharp for a universal sufficient condition formulated through a vanishing power bound on the absolute tail. However, it is not necessary for every nonsymmetric law. In Example~\ref{rem:balanced-tails}, the symmetric heavy-tail component contributes zero to the signed Laplace transform, while the bounded asymmetric component contributes only $O(t)$. The conditional comparison therefore remains effective at the absolute-tail boundary, and the $F$ approximation still holds.

\section{Proofs of the symmetric cases}\label{sec:proofmainresults}
This section presents the proofs of the symmetric cases, including Theorem \ref{thm:DAN}, Corollary \ref{thm:all-dimensions}, and Corollary \ref{cor:normal}. We begin with an estimate for the probabilities of full-rank events, which is used frequently in the sequel; the proof is given in Appendix \ref{app:lem:rank}.
\begin{lemma}\label{lem:rank}
	Recall the definitions of $\mathcal{F}$ and $\mathcal{G}$ in Section~\ref{sec:definitions}. Let $	q:=\sup_{x\in\mathbb{R}}\mathbb{P}(\xi=x)<1$ be the maximal atom of a nondegenerate entry law $\xi$. Then
	\begin{equation}
	\eta_n=\mathbb{P}(\mathcal{G}^c)\le q^{n-p}/(1-q),
	\end{equation}
	with value zero if $q=0$. Moreover, $\mathbb{P}(\mathcal{F}^c)\le\eta_n$.
\end{lemma}

On the event $\mathcal{G}$, in view of \eqref{eq:PQR} and the identity \eqref{eq:identity}, it suffices to study the quadratic form  $R=\mathbf{1}^{\top}\mathbf{P}\mathbf{1}$. Let $\boldsymbol{\varepsilon}=(\varepsilon_1,\ldots,\varepsilon_n)^\top$ have independent Rademacher coordinates, independently of $\mathbf{X}$, and set
$ \mathbf{D}_{\boldsymbol{\varepsilon}}=\diag(\varepsilon_1,\ldots,\varepsilon_n)$.
If the entries are symmetric, $\mathbf{X}\mathbf{D}_{\boldsymbol{\varepsilon}}\stackrel d=\mathbf{X}$, and the null-space projection satisfies
\begin{equation*}
\mathbf{P}(\mathbf{X}\mathbf{D}_{\boldsymbol{\varepsilon}})=\mathbf{D}_{\boldsymbol{\varepsilon}} \mathbf{P}(\mathbf{X})\mathbf{D}_{\boldsymbol{\varepsilon}},
\end{equation*}
by $\mathbf{D}_{\boldsymbol{\varepsilon}} \mathbf{D}_{\boldsymbol{\varepsilon}}=\mathbf{I}_n$. Consequently, $R\stackrel d=R^{\rm s}=\boldsymbol{\varepsilon}^\top \mathbf{P}\boldsymbol{\varepsilon}$ under symmetry.
Thus, the analysis of the quadratic form $\mathbf{1}^{\top}\mathbf{P}\mathbf{1}$ boils down to a Rademacher quadratic form $\boldsymbol{\varepsilon}^\top \mathbf{P}\boldsymbol{\varepsilon}$, whose centered version is $\boldsymbol{\varepsilon}^\top (\mathbf{P}-\operatorname{diag}\mathbf{P})\boldsymbol{\varepsilon}$ since $\boldsymbol{\varepsilon}^\top (\operatorname{diag}\mathbf{P})\boldsymbol{\varepsilon}=m$ on the event $\mathcal{G}$. The following lemma gives its distributional approximation, and the fourth-moment bound plays an essential role in the necessity argument. The proof, deferred to Appendix \ref{app:prf_sign}, relies on a Lindeberg replacement argument and a smooth comparison for distribution functions.
\begin{lemma}\label{lem:sign}
	Let $\mathbf{D}\in\mathbb{R}^{n\times n}$ be a deterministic real symmetric matrix with zero diagonal, let $s^2=2\tr \mathbf{D}^2>0$, and let $\boldsymbol{\varepsilon}$ have independent Rademacher coordinates. Then
	\begin{align}
	d_K(\boldsymbol{\varepsilon}^\top \mathbf{D}\boldsymbol{\varepsilon}/s,N(0,1))
	&\le C(\|\mathbf{D}\|/s)^{1/4},\label{eq:sign-bound}\\
	\mathbb{E}(\boldsymbol{\varepsilon}^\top \mathbf{D}\boldsymbol{\varepsilon})^4
	&\le60(\tr \mathbf{D}^2)^2.\label{eq:sign-fourth}
	\end{align}
	If $\mathbf{U}\in\mathbb{R}^{n\times k}$ has orthonormal columns, $k$ is fixed, and $b=\max_{1\le j\le n}\|\mathbf{U}^\top \mathbf{e}_j\|^2$, then
	\begin{equation}\label{eq:sign-fixed}
	d_K(\|\mathbf{U}^\top\boldsymbol{\varepsilon}\|^2,\chi_k^2)\le C_k b^{1/8}.
	\end{equation}
\end{lemma}

With Lemma \ref{lem:sign} and the symmetry condition, the following proposition illustrates the $F$ approximation error bound $\varepsilon_{n,p}$ in terms of the relative projection-diagonal concentration error $\mathbb{E}\delta_{n,p}$, which is a key ingredient in showing sufficiency and necessity in Theorem \ref{thm:DAN}. Moreover, in the near-square case $v\asymp 1$ in Corollary \ref{thm:all-dimensions}, the $F$ approximation error is bounded in terms of the maximum diagonal entry $\max_{1\le j\le n}P_{jj}$. The proof is deferred to Appendix \ref{app:prf_proposition}.

\begin{proposition}\label{thm:comparison}
	Recalling $v$ in \eqref{eq:PQR}, for a symmetric entry law and $v\ge1$,
	\begin{equation}\label{eq:comparison}
	\varepsilon_{n,p}
	\le C\big(v^{-1/8}+\mathbb{E}\delta_{n,p}+n^{-1}+\eta_n\big).
	\end{equation}
	For fixed $p=k$, the right-hand side can be replaced by
\begin{equation}\label{eq:fixed-dimension-bound}
	C_k\big(\mathbb{E}[\mathbbm{1}_{\mathcal{F}}(\max_{1\le j\le n}L_{jj})^{1/8}]+n^{-1/4}+\eta_n\big).
\end{equation}
	For fixed $m=k$, it can instead be replaced by
	\begin{equation}\label{eq:fixed-nullity-bound}
	C_k\big(\mathbb{E}[\mathbbm{1}_{\mathcal{F}}(\max_{1\le j\le n}P_{jj})^{1/8}]+n^{-1/4}+\eta_n\big).
	\end{equation}
	Consequently, under symmetry and for any sequence $1\le p<n$,
	\begin{equation}\label{eq:all-dimension-geometry}
	\varepsilon_{n,p}\to0
	\quad\text{is equivalent to}\quad
	\eta_n\to0\ \hbox{and}\ \mathbb{E}\delta_{n,p}\to0.
	\end{equation}
\end{proposition}

\subsection{Proof of Theorem \ref{thm:DAN}}
For a symmetric nondegenerate law, \eqref{eq:aspect-away} and Lemma~\ref{lem:rank} give $\eta_n\to0$. The equivalence \eqref{eq:all-dimension-geometry} in Proposition~\ref{thm:comparison} therefore makes $\varepsilon_{n,p}\to0$ equivalent to $\mathbb{E}\delta_{n,p}\to0$. It remains to show that this diagonal concentration condition $\mathbb{E}\delta_{n,p}\to0$ is equivalent to $\xi \in \mathrm{DAN}$.
\subsubsection{Sufficiency of DAN}
Let $\ell(a)=\mathbb{E}(\xi^2\mathbbm{1}_{\{|\xi|\le a\}})$. For a symmetric nondegenerate law, normal attraction is equivalent to slow variation of $\ell$; see \citet{Bingham1989}. Select $a_n\to\infty$ with $n\ell(a_n)/a_n^2\to1$. The associated tail criterion gives
\begin{equation}\label{eq:DAN-truncation}
n\mathbb{P}(|\xi|>a_n)\to0,
\qquad
\frac{\mathbb{E}(\xi^4\mathbbm{1}_{\{|\xi|\le a_n\}})}{n\ell(a_n)^2}\to0.
\end{equation}
The second assertion follows directly from slow variation. Indeed, for $0<\epsilon<1$,
\[
\frac{\mathbb{E}(\xi^4\mathbbm{1}_{\{|\xi|\le a\}})}{a^2\ell(a)}
\le\epsilon^2+\frac{\ell(a)-\ell(\epsilon a)}{\ell(a)}.
\]
First let $a\to\infty$ with $\epsilon$ fixed. The second term tends to zero by the slow variation of $\ell(a)$. Letting $\epsilon\rightarrow 0$ then gives the second assertion in \eqref{eq:DAN-truncation}.
Let $\mathbf{X}^\circ$ be the matrix from $\mathbf{X}$ with the entry law $\xi\mathbbm{1}_{\{|\xi|\le a_n\}}/\sqrt{\ell(a_n)}$. By symmetry and \eqref{eq:DAN-truncation}, the entries of $\mathbf{X}^\circ$ are centered, symmetric, and have variance one, with fourth moment $\mu_{4,n}$ satisfying $\mu_{4,n}/n\to0$. Denote its row-space and null-space projections by $\mathbf{L}^\circ$ and $\mathbf{P}^\circ=\mathbf{I}_n-\mathbf{L}^\circ$, respectively. These projections are also defined on rank failure. The usual inverse formula for $\mathbf{L}^\circ$ holds on $\mathcal{F}(\mathbf{X}^\circ)$.

Suppose first that $p/n\to0$ and let $\mathbf{C}_n=\mathbf{X}^\circ(\mathbf{X}^\circ)^\top/n$. Entrywise expansion gives
\begin{equation}
\frac1p\mathbb{E}\tr(\mathbf{C}_n-\mathbf{I}_p)^2
=\frac{\mu_{4,n}+p-2}{n}\to0.
\end{equation}
Consider the spectral decomposition of \(\mathbf{C}_n\) as
\[
\mathbf{C}_n=\mathbf{U}\boldsymbol{\Lambda} \mathbf{U}^\top
=\sum_{i=1}^p\lambda_i \mathbf{u}_i \mathbf{u}_i^\top ,
\]
where \(\lambda_1,\ldots,\lambda_p\) are the eigenvalues of \(\mathbf{C}_n\). For each strictly positive eigenvalue, define
\[
\mathbf{v}_i=\frac{(\mathbf{X}^\circ)^\top \mathbf{u}_i}{\sqrt{n\lambda_i}},\qquad \lambda_i>0.
\]
Then
\[
\mathbf{v}_i^\top \mathbf{v}_j
=\frac{\mathbf{u}_i^\top \mathbf{X}^\circ(\mathbf{X}^\circ)^\top \mathbf{u}_j}{n\sqrt{\lambda_i\lambda_j}}
=\frac{\mathbf{u}_i^\top \mathbf{C}_n\mathbf{u}_j}{\sqrt{\lambda_i\lambda_j}}
=\mathbbm{1}_{\{i=j\}}.
\]
Moreover, these vectors span the row space of $\mathbf{X}^\circ$, and hence we can represent
\[
\mathbf{L}^\circ=\sum_{\lambda_i>0}\mathbf{v}_i\mathbf{v}_i^\top=\sum_{\lambda_i\ge1/2}\mathbf{v}_i\mathbf{v}_i^\top+\sum_{0<\lambda_i<1/2}\mathbf{v}_i\mathbf{v}_i^\top=:\mathbf{L}_++\mathbf{L}_-.
\]
We first control the diagonal entries of \(\mathbf{L}_+\). For the \(j\)-th
diagonal element, we have
\[
(\mathbf{L}_+)_{jj}
=
\sum_{\lambda_i\ge1/2}
\frac{\big((\mathbf{X}^\circ)_{\cdot j}^{\top}\mathbf{u}_i\big)^2}
{n\lambda_i}\le
\frac2n
\sum_{\lambda_i\ge1/2}
\big((\mathbf{X}^\circ)_{\cdot j}^{\top}\mathbf{u}_i\big)^2\le 2\|\mathbf{x}_j^\circ\|^2/n,
\]
where we used \(\lambda_i\ge1/2\) and
\[
\sum_{i=1}^p
\big((\mathbf{X}^\circ)_{\cdot j}^{\top}\mathbf{u}_i\big)^2
=
\|(\mathbf{X}^\circ)_{\cdot j}\|^2
=\|\mathbf{x}_j^\circ\|^2
\]
by the orthonormality of \(\{\mathbf{u}_i\}_{i=1}^p\). Next, consider \(\mathbf{L}_-\). If \(0<\lambda_i<1/2\), then
\[
(1-\lambda_i)^2>\frac14 .
\]
Hence every eigenvalue contributing to \(\mathbf{L}_-\) contributes at least
\(1/4\) to $\operatorname{tr}(\mathbf{C}_n-\mathbf{I}_p)^2
=\sum_{i=1}^p(1-\lambda_i)^2$. Therefore,
\[
\frac14\operatorname{rank}(\mathbf{L}_-)
\le
\operatorname{tr}(\mathbf{C}_n-\mathbf{I}_p)^2.
\]
Since \(\mathbf{L}_-\) is an orthogonal projection onto the subspace spanned by the vectors \(\mathbf{v}_i\) with \(0<\lambda_i<1/2\), we have $\sum_{j=1}^n(\mathbf{L}_-)_{jj}^2\le\tr \mathbf{L}_-=\operatorname{rank}(\mathbf{L}_-)$, which implies
\begin{align}
\frac1p\mathbb{E}\sum_{j=1}^n(L^\circ_{jj})^2&\le
\frac2p\mathbb{E}\sum_{j=1}^n(\mathbf{L}_+)_{jj}^2
+
\frac2p\mathbb{E}\sum_{j=1}^n(\mathbf{L}_-)_{jj}^2 \nonumber\\
&\le\frac8{n^2p}\sum_{j=1}^n\mathbb{E}\|\mathbf{x}_j^\circ\|^4
+\frac8p\mathbb{E}\tr(\mathbf{C}_n-\mathbf{I}_p)^2\\
&\le C\frac{\mu_{4,n}+p}{n}\to0,\notag
\end{align}
where the last line follows from $\mathbb{E}\|\mathbf{x}_j^\circ\|^4=p\mu_{4,n}+p(p-1)$.

For a subsequence with $\gamma_n=p/n\to\gamma\in(0,1)$, fix $t>0$ and set
\[
a_n(t)=n^{-1}\tr(\mathbf{C}_n+t\mathbf{I}_p)^{-1},\qquad
\mathbf{K}_n(t)=\mathbf{I}_n-n^{-1}(\mathbf{X}^\circ)^\top(\mathbf{C}_n+t\mathbf{I}_p)^{-1}\mathbf{X}^\circ.
\]
Delete the $j$-th column $\mathbf{x}_j^\circ$ and let
\[
\mathbf{M}_j=(\mathbf{C}_n-n^{-1}\mathbf{x}_j^\circ(\mathbf{x}_j^\circ)^\top+t\mathbf{I}_p)^{-1},
\qquad q_j=n^{-1}(\mathbf{x}_j^\circ)^\top \mathbf{M}_j\mathbf{x}_j^\circ.
\]
By the Sherman--Morrison formula, and using $\mathbf{M}_j^2\preceq t^{-1}\mathbf{M}_j$, we obtain
$(\mathbf{K}_n(t))_{jj}=(1+q_j)^{-1}$ and
\[
0\le n^{-1}\tr \mathbf{M}_j-a_n(t)
=\frac{(\mathbf{x}_j^\circ)^\top \mathbf{M}_j^2\mathbf{x}_j^\circ}{n^2(1+q_j)}\le \frac{(\mathbf{x}_j^\circ)^\top \mathbf{M}_j\mathbf{x}_j^\circ}
{tn^2(1+q_j)}=
\frac{q_j}{nt(1+q_j)}
\le\frac1{nt}.
\]
Applying Lemma~\ref{lem:quadratic-LLN}, with $\mu_{4,n}/n\to0$, gives
\[
n^{-1}\sum_{j=1}^n\mathbb{E}|q_j-a_n(t)|\to0.
\]
Since the function $f(u)=(1+u)^{-1}$ is Lipschitz on $[0,\infty)$, we get
\begin{equation}\label{eq:k1ant}
\frac1n\sum_{j=1}^n|(\mathbf{K}_n(t))_{jj}-(1+a_n(t))^{-1}|\pto0.
\end{equation}
The trace identity $n^{-1}\tr \mathbf{K}_n(t)=1-\gamma_n+ta_n(t)$ therefore implies
\begin{equation}
1-\gamma_n+ta_n(t)-(1+a_n(t))^{-1}\pto0.
\end{equation}
To identify the limit, let $a_n^*(t)$ be the unique nonnegative zero of
\[
F_n(a)=1-\gamma_n+ta-(1+a)^{-1}.
\]
Indeed, $F_n(0)=-\gamma_n<0$, $F_n(a)\to\infty$ as $a\to\infty$, and $F_n'(a)=t+(1+a)^{-2}\ge t$. Consequently,
\[
|a_n(t)-a_n^*(t)|\le t^{-1}|F_n(a_n(t))|\pto0.
\]
Writing $k_n^*(t)=(1+a_n^*(t))^{-1}$ in $F_n(a_n^*(t))=0$ gives
\[
(k_n^*(t))^2-(1-\gamma_n-t)k_n^*(t)-t=0.
\]
Its positive root converges to the corresponding root with $\gamma_n$ replaced by $\gamma$. Therefore,
\begin{equation}\label{eq:ridge-root}
(1+a_n(t))^{-1}\pto
k_\gamma(t):=\frac{1-\gamma-t+\sqrt{(1-\gamma-t)^2+4t}}2.
\end{equation}
Combining \eqref{eq:k1ant} with \eqref{eq:ridge-root} gives
\[
\frac1n\sum_{j=1}^n\big((\mathbf{K}_n(t))_{jj}-k_\gamma(t)\big)^2\pto0.
\]
Indeed, all the diagonal entries and $k_\gamma(t)$ lie in $[0,1]$, so the squared differences are bounded by the absolute differences. The same convergence holds in expectation by boundedness.

Notice that singular-value decomposition gives
$\mathbf{0}\preceq\mathbf{P}^\circ\preceq\mathbf{K}_n(t)\preceq\mathbf{I}_n$.
Let $\mathcal{F}^\circ=\{\operatorname{rank}\mathbf{X}^\circ=p\}$. The maximal atom of the truncated law is at most the original maximal atom plus $\mathbb{P}(|\xi|>a_n)$, and hence is bounded away from one for all sufficiently large $n$. Applying Lemma~\ref{lem:rank} gives $\mathbb{P}((\mathcal{F}^\circ)^c)\le Ce^{-cn}$. Consequently $n^{-1}\mathbb{E}\tr\mathbf{P}^\circ\to1-\gamma$, whereas the trace identity gives $n^{-1}\mathbb{E}\tr\mathbf{K}_n(t)\to k_\gamma(t)$. Since the eigenvalues of $\mathbf{K}_n(t)-\mathbf{P}^\circ$ belong to $[0,1]$, we get
\[
\frac1n\sum_{j=1}^n\big((\mathbf{K}_n(t))_{jj}-P^\circ_{jj}\big)^2
\le\frac1n\tr\big(\mathbf{K}_n(t)-\mathbf{P}^\circ\big).
\]
Adding and subtracting $(\mathbf{K}_n(t))_{jj}$ and $k_\gamma(t)$, and using $(a+b+c)^2\le3(a^2+b^2+c^2)$, now yields
\begin{align*}
\limsup_{n\to\infty}\frac1n\mathbb{E}
\sum_{j=1}^n\big(P^\circ_{jj}-(1-\gamma_n)\big)^2
\le3\big(k_\gamma(t)-(1-\gamma)\big)+3\big(k_\gamma(t)-(1-\gamma)\big)^2\rightarrow 0,
\end{align*}
since $k_\gamma(t)\to1-\gamma$ as $t\rightarrow 0$.

We finally transfer these conclusions for $\mathbf{P}^{\circ}$ to the original projection $\mathbf{P}$. Let $N_n$ denote the number of rows containing at least one entry larger than $a_n$ in absolute value. From \eqref{eq:DAN-truncation}, we get
\[
\mathbb{E} N_n\le pn\mathbb{P}(|\xi|>a_n)=o(p).
\]
Since row rescaling does not change a row space, it suffices to compare $\mathbf{X}/\sqrt{\ell(a_n)}$ with $\mathbf{X}^\circ$. Let $\mathcal{S}$ be the span of these common rows. Each row space is obtained from $\mathcal{S}$ by adjoining at most $N_n$ rows, so their projections agree on $\mathcal{S}$ and their difference has rank at most $2N_n$. Since the difference of two orthogonal projections has operator norm at most one, we have
\begin{equation}
\sum_{j=1}^n(L_{jj}-L^\circ_{jj})^2
\le\|\mathbf{L}-\mathbf{L}^\circ\|_F^2\le2N_n.
\end{equation}
For $p/n\to0$, applying $(a+b)^2\le2a^2+2b^2$ gives
\[
\frac1p\mathbb{E}\sum_{j=1}^nL_{jj}^2
\le\frac2p\mathbb{E}\sum_{j=1}^n(L^\circ_{jj})^2
+\frac4p\mathbb{E}N_n\to0.
\]
On $\mathcal{F}$, the numerator of $\delta_{n,p}$ equals $\sum_{j=1}^nL_{jj}^2-p^2/n$, which further implies $\mathbb{E}\delta_{n,p}\to0$. When $p/n\to\gamma\in(0,1)$, using $P_{jj}-P^\circ_{jj}=-(L_{jj}-L^\circ_{jj})$ gives
\[
\frac1n\mathbb{E}\sum_{j=1}^n\big(P_{jj}-(1-\gamma_n)\big)^2
\le\frac4n\mathbb{E}N_n
+\frac2n\mathbb{E}\sum_{j=1}^n\big(P^\circ_{jj}-(1-\gamma_n)\big)^2
\to0.
\]
Since $mp/n\asymp n$ in this case, the same conclusion follows for $\mathbb{E}\delta_{n,p}$. Finally, every subsequence satisfying \eqref{eq:aspect-away} has a further subsequence with $p/n\to\gamma\in[0,1)$, which completes the proof of sufficiency.
\subsubsection{Necessity of DAN}
Let $W=\xi^2$, $U(t)=\mathbb{E}(W\mathbbm{1}_{\{W\le t\}})$, and $q(t)=\mathbb{P}(W>t)$. The normal-attraction tail criterion is
\begin{equation}\label{eq:DAN-tail-criterion}
tq(t)/U(t)\rightarrow 0;
\end{equation}
see \citet{Bingham1989} for details.
Suppose \eqref{eq:DAN-tail-criterion} fails. There exist $c>0$ and $t_k\to\infty$ with $t_kq(t_k)\ge cU(t_k)$. Fix $\gamma_0<1$ with $p/n\le\gamma_0$. Choose $\epsilon>0$ such that $h=(1+\epsilon)^{-1}>\gamma_0$, and then choose $b>0$ so small that $b(1+1/(c\epsilon))<1/2$. Set $n_k=\lfloor b/q(t_k)\rfloor$, so $n_kq(t_k)\to b$. At $n=n_k$,
\begin{align*}
\mathbb{P}\left(\sum_{j=2}^nW_j>\epsilon t_k\right)
\le(n-1)q(t_k)+\frac{(n-1)U(t_k)}{\epsilon t_k}\le b\left(1+\frac1{c\epsilon}\right)<1/2,
\end{align*}
which further implies
\[
\mathbb{P}\left(W_1>t_k,\ \sum_{j=2}^nW_j\le\epsilon t_k\right)
\ge q(t_k)/2.
\]
Define
\[
\mathcal{A}=\{W_1>t_k,\ \sum_{j=2}^nW_j\le\epsilon t_k\}.
\]
On $\mathcal{A}$, the corresponding normalized row $\mathbf{y}$ satisfies $y_1^2\ge h$. Since that row belongs to the row space, $\mathbf{L}\succeq\mathbf{y}\mathbf{y}^\top$ and $L_{11}\ge h$. For $p$ independent rows, we get
\[
\mathbb{P}(L_{11}\ge h)\ge1-(1-q(t_k)/2)^p\ge c_b p q(t_k).
\]
The last inequality is uniform because $pq(t_k)$ is bounded. More explicitly, $1-(1-q/2)^p\ge1-e^{-pq/2}$, and $1-e^{-u}\ge c_bu$ on the bounded interval containing $pq(t_k)/2$. Exchangeability of columns then gives
\[
\frac1p\mathbb{E}\sum_{j=1}^n(L_{jj}-p/n)^2
=\frac np\mathbb{E}(L_{11}-p/n)^2
\ge c_b(h-\gamma_0)^2nq(t_k).
\]
Since $n=n_k$ and $n_kq(t_k)\to b$, the limit inferior of the right-hand side is strictly positive. Restricting the left-hand side to the original full-row-rank event $\mathcal{F}$ changes it by at most $(n/p)\mathbb{P}(\mathcal{F}^c)\le(n/p)\eta_n=o(1)$. On $\mathcal{F}$, one has
\[
\frac1p\sum_{j=1}^n(L_{jj}-p/n)^2=(1-p/n)\delta_{n,p}.
\]
Consequently $\mathbb{E}\delta_{n,p}$ cannot tend to zero, completing the proof of necessity.

\subsection{Proof of Corollary~\ref{thm:all-dimensions}}\label{sec:proof-near-square}

In view of Theorem~\ref{thm:DAN}, it suffices to consider the case \(p/n\to1\) with \(p<n\). Without loss of generality, assume that the entry law has variance one. We first show that
\begin{equation}\label{eq:near-square-maxdiag}
\max_{1\le j\le n}P_{jj}\pto0,
\qquad
\mathbb{P}(\operatorname{rank}\mathbf{X}=p)\rightarrow 1.
\end{equation}
For fixed \(t>0\), let
\[
\mathbf{C}_n=\mathbf{X}\mathbf{X}^\top/n,~~
a_n(t)=n^{-1}\tr(\mathbf{C}_n+t\mathbf{I}_p)^{-1},
~~
\mathbf{K}_n(t)=\mathbf{I}_n-n^{-1}\mathbf{X}^\top
(\mathbf{C}_n+t\mathbf{I}_p)^{-1}\mathbf{X}.
\]
For the \(j\)-th column \(\mathbf{x}_j\), set
\[
\mathbf{M}_j=
\big(\mathbf{C}_n-n^{-1}\mathbf{x}_j\mathbf{x}_j^\top+t\mathbf{I}_p\big)^{-1},
~~
q_j=n^{-1}\mathbf{x}_j^\top\mathbf{M}_j\mathbf{x}_j.
\]
Then \(\mathbf{M}_j\) is independent of \(\mathbf{x}_j\), and the Sherman--Morrison formula gives
\[
(\mathbf{K}_n(t))_{jj}=(1+q_j)^{-1}.
\]
Moreover, using \(\mathbf{M}_j^2\preceq t^{-1}\mathbf{M}_j\),
\[
0\le n^{-1}\tr\mathbf{M}_j-a_n(t)
\le \frac{q_j}{nt(1+q_j)}
\le\frac1{nt}.
\]
Applying Lemma~\ref{lem:quadratic-LLN} conditionally on \(\mathbf{M}_j\) therefore gives
\[
\frac1n\sum_{j=1}^n\mathbb{E}|q_j-a_n(t)|\to0,
\]
which further implies
\begin{equation}\label{eq:near-square-k-diagonal}
\frac1n\sum_{j=1}^n
\left|(\mathbf{K}_n(t))_{jj}-(1+a_n(t))^{-1}\right|
\pto0.
\end{equation}
On the other hand,
\[
n^{-1}\tr\mathbf{K}_n(t)=1-p/n+ta_n(t),
\]
so \eqref{eq:near-square-k-diagonal} yields
\[
1-p/n+ta_n(t)-(1+a_n(t))^{-1}\pto0.
\]
The left-hand side is strictly increasing as a function of \(a_n(t)\). Since \(p/n\to1\), its zero converges to
\begin{equation}
a(t)=\frac{\sqrt{1+4/t}-1}{2},
\qquad
a_n(t)\pto a(t).
\end{equation}
Notice that \(a(t)\to\infty\) as \(t\rightarrow 0\).
Since \(\|\mathbf{M}_j\|\le1/t\) and \(\tr\mathbf{M}_j\ge na_n(t)\), conditioning on all columns except \(j\), applying Lemma~\ref{lem:positive-form} to \(\mathbf{M}_j\), and then taking a union bound give
\begin{align*}
&\mathbb{P}\left(a_n(t)\ge a(t)/2,
\min_{1\le j\le n}n^{-1}\mathbf{x}_j^\top\mathbf{M}_j\mathbf{x}_j
<a_0a(t)/2\right)\\
&\qquad\le
\sum_{j=1}^n
\mathbb{P}\left(\tr\mathbf{M}_j\ge na(t)/2,
\mathbf{x}_j^\top\mathbf{M}_j\mathbf{x}_j<a_0\tr\mathbf{M}_j\right)
\le Cn e^{-cnta(t)/2}\to0.
\end{align*}
Consequently,
\[
\mathbb{P}\left(
\max_{1\le j\le n}(\mathbf{K}_n(t))_{jj}
>
\frac1{1+a_0a(t)/2}
\right)\to0.
\]
A spectral decomposition gives \(\mathbf{0}\preceq\mathbf{P}\preceq\mathbf{K}_n(t)\). Given \(\epsilon>0\), choosing \(t>0\) sufficiently small so that \((1+a_0a(t)/2)^{-1}<\epsilon\), and then letting \(n\to\infty\), gives
\[
\max_{1\le j\le n}P_{jj}\pto0.
\]

It remains to verify full row rank when the nullity remains bounded. If \(m=n-p\to\infty\), Lemma~\ref{lem:rank} gives \(\mathbb{P}(\mathcal{F}^c)\to0\). Thus it remains to consider \(m=k\) fixed. Let \(q<1\) be the maximal atom of the entry law and fix an integer \(K>k\). If \(\mathbf{B}_i\) denotes the matrix formed by the first \(i\) rows of \(\mathbf{X}\), exposing the first \(n-K\) rows gives
\[
\mathbb{P}\big(\operatorname{rank}\mathbf{B}_{n-K}<n-K\big)
\le \frac{q^{K+1}}{1-q}.
\]
For \(n-K\le i<n-k\), choose an \(\mathcal{F}_i\)-measurable unit vector \(\mathbf{v}_i\in\ker\mathbf{B}_i\), where \(\mathcal{F}_i=\sigma(\mathbf{B}_i)\). Since \(\mathbf{v}_i=\mathbf{P}_i\mathbf{v}_i\), the Cauchy--Schwarz inequality gives
\[
\max_{1\le j\le n}|(\mathbf{v}_i)_j|^2
\le
\max_{1\le j\le n}(\mathbf{P}_i)_{jj}\pto0,
\]
where the maximum-diagonal estimate above is applied to the matrix \(\mathbf{B}_i\).
The next row \(\mathbf{z}_{i+1}\) is independent of \(\mathcal{F}_i\). Conditional on \(\mathcal{F}_i\), for every \(\epsilon>0\),
\[
\sum_{j=1}^n(\mathbf{v}_i)_j^2
\mathbb{E}_i\left[
\xi^2\mathbbm{1}_{\{|(\mathbf{v}_i)_j\xi|>\epsilon\}}
\right]
\le
\mathbb{E}\left[
\xi^2\mathbbm{1}_{\{|\xi|>\epsilon/\|\mathbf{v}_i\|_\infty\}}
\right]\pto0
\]
by finite variance. Hence the conditional Lindeberg central limit theorem gives
\[
\mathbf{v}_i^\top\mathbf{z}_{i+1}\dto N(0,1).
\]
In particular, for every \(\epsilon>0\),
\[
\limsup_{n\to\infty}
\mathbb{P}(\mathbf{v}_i^\top\mathbf{z}_{i+1}=0)
\le
\mathbb{P}(|N(0,1)|\le\epsilon),
\]
and letting \(\epsilon\rightarrow 0\) shows that this probability tends to zero. Since the event that the new row belongs to the previous row space is contained in \(\{\mathbf{v}_i^\top\mathbf{z}_{i+1}=0\}\), and there are only \(K-k\) remaining steps, we get
\[
\limsup_{n\to\infty}\mathbb{P}(\operatorname{rank}\mathbf{X}<p)
\le\frac{q^{K+1}}{1-q}.
\]
Letting \(K\to\infty\) gives the second assertion in \eqref{eq:near-square-maxdiag}. Together with the case \(m\to\infty\), a subsequence argument covers every sequence with \(p/n\to1\).

It remains to control the augmented-full-rank event \(\mathcal{G}\). If \(m\to\infty\), Lemma~\ref{lem:rank} gives directly
\[
\eta_n=\mathbb{P}(\mathcal{G}^c)
\le\frac{q^m}{1-q}\to0.
\]
Suppose now that \(m=k\) is fixed. Under symmetry, applying Lemma~\ref{lem:sign} to an orthonormal basis of \(\ker\mathbf{X}\) on \(\mathcal{F}\) gives
\[
d_K(R,\chi_k^2)
\le
C_k\mathbb{E}\left[
\mathbbm{1}_{\mathcal{F}}
\left(\max_{1\le j\le n}P_{jj}\right)^{1/8}
\right]
+\mathbb{P}(\mathcal{F}^c)\to0.
\]
Thus \(R\dto\chi_k^2\). Since \(R\ge0\), for every \(\epsilon>0\),
\[
\limsup_{n\to\infty}\mathbb{P}(R=0)
\le
\limsup_{n\to\infty}\mathbb{P}(R\le\epsilon)
\le
\mathbb{P}(\chi_k^2\le\epsilon),
\]
and letting \(\epsilon\rightarrow 0\) yields \(\mathbb{P}(R=0)\to0\). On \(\mathcal{F}\), the condition \(R>0\) is equivalent to \(\mathcal{G}\), so
\[
\eta_n
\le
\mathbb{P}(\mathcal{F}^c)+\mathbb{P}(R=0)
\to0.
\]
Hence \(\eta_n\to0\) for every sequence with $p/n\rightarrow 1$.

Finally, on \(\mathcal{F}\),
\[
0\le\delta_{n,p}
=
\frac{\sum_{j=1}^nP_{jj}^2-m^2/n}{mp/n}
\le
\frac np\max_{1\le j\le n}P_{jj}.
\]
Since \(n/p\to1\), \(\max_{1\le j\le n}P_{jj}\pto0\), and \(0\le\delta_{n,p}\le1\), bounded convergence gives \(\mathbb{E}\delta_{n,p}\to0\). Together with \(\eta_n\to0\), Proposition~\ref{thm:comparison} yields $\varepsilon_{n,p}\to0$, which completes the proof.

\subsection{Proof of Corollary~\ref{cor:normal}}\label{sec:proof-normal}

For simplicity, let
\[
Z_n=\frac{R-m}{\sqrt{2mp/n}},\qquad
Y_n=\sqrt{\frac{m^3}{2pn}}\left(\frac{T_n^2}{n}-\frac pm\right),\qquad
L_n=\sqrt{\frac{nm}{2p}}\big(\log(1+T_n^2/n)+\log(m/n)\big).
\]
We first record the relations among these three statistics. Suppose that \(p,m\to\infty\) and \(\eta_n\to0\). If \(Z_n\) is tight, writing
\[
h_n=\frac{R-m}{m}
=\sqrt{\frac{2p}{nm}}Z_n=o_{\mathbb P}(1),
\]
the identity \(1+T_n^2/n=n/R\) gives, on \(\mathcal G\),
\[
Y_n
=\sqrt{\frac{nm}{2p}}\left(\frac1{1+h_n}-1\right)
=-Z_n+o_{\mathbb P}(1),
\]
where we used $|\frac1{1+h}-1+h|\le 2h^2$ for $|h|\le1/2$. Similarly, applying \(|\log(1+h)-h|\le Ch^2\) gives
\[
L_n=-Z_n+o_{\mathbb P}(1).
\]
Consequently,
\begin{align}
Y_n&=-Z_n+o_{\mathbb P}(1),\label{eq:inverse-transform-CLT}\\
L_n&=-Z_n+o_{\mathbb P}(1).\label{eq:log-transform-CLT}
\end{align}
The converse implications follow from the same transformations. Indeed, if \(Y_n\) is tight, set
\[
u_n=\frac{T_n^2/n-p/m}{1+p/m}.
\]
Then \(u_n=o_{\mathbb P}(1)\) and, on \(\mathcal G\), \(R=m/(1+u_n)\). Hence
\[
Z_n
=\sqrt{\frac{nm}{2p}}\left(\frac1{1+u_n}-1\right)
=-Y_n+o_{\mathbb P}(1).
\]
If \(L_n\) is tight, writing \(a_n=\sqrt{nm/(2p)}\to\infty\), the identity $R=m\exp(-L_n/a_n)$ gives
\[
Z_n=a_n\big(\exp(-L_n/a_n)-1\big)
=-L_n+o_{\mathbb P}(1).
\]
Thus tightness of any one of \(Z_n,Y_n,L_n\) implies tightness of the other two.

We now proceed to complete the proof of Corollary \ref{cor:normal}. Under the applicable symmetric result in Theorem \ref{thm:DAN} or Corollary \ref{thm:all-dimensions}, the \(F\) approximation holds. Since the reference distribution is strictly positive and \(T_n^2=0\) on \(\mathcal G^c\), it also follows that \(\eta_n\to0\). Using the monotone transformation in \eqref{eq:identity}, we obtain
\[
d_K(R,R^{\rm G})\to0.
\]
If \(p,m\to\infty\), then \(v=mp/n\to\infty\). For the Gaussian projection, the sign comparison together with \(\mathbb E\delta_{n,p}^{\rm G}=2/(n+2)\) gives
\[
\frac{R^{\rm G}-m}{\sqrt{2v}}\dto N(0,1).
\]
Therefore,
\[
\frac{R-m}{\sqrt{2mp/n}}\dto N(0,1),
\]
which implies \eqref{eq:R-CLT}. Applying \eqref{eq:inverse-transform-CLT} and \eqref{eq:log-transform-CLT} then gives \eqref{eq:T-CLT} and \eqref{eq:all-dimension-log}.

Under symmetry and \eqref{eq:aspect-away}, the converse follows from the same projection representation. If \eqref{eq:R-CLT} holds, the uniformly bounded fourth moments of the rank-corrected Rademacher quadratic form used in the proof of Proposition~\ref{thm:comparison} give uniform integrability. Hence the limiting variance is one, which implies
\[
\mathbb E\delta_{n,p}\to0.
\]
Theorem~\ref{thm:DAN} then yields \(\xi\in\DAN\). If \eqref{eq:T-CLT} holds, the reverse transformation above first gives \eqref{eq:R-CLT}, and the same argument applies.

Finally, if \(p=k\) is fixed, the Gaussian row-space statistic satisfies $n-R^{\rm G}\dto\chi_k^2$. The \(F\) approximation therefore gives \(Q=n-R\dto\chi_k^2\). Since
\[
T_n^2-Q=\frac{Q^2}{n-Q}=o_{\mathbb P}(1)
\]
on \(\mathcal G\), we obtain $T_n^2\dto\chi_k^2$.
If \(m=k\) is fixed, the null-space comparison gives $R\dto\chi_k^2$. Using the identity
\[
\frac{T_n^2}{n^2}=\frac1R-\frac1n
\]
on \(\mathcal G\), together with the fact that \(\chi_k^2\) has no atom at zero and $\eta_n\rightarrow 0$, the continuous mapping theorem yields
\[
\frac{T_n^2}{n^2}\dto\frac1{\chi_k^2}.
\]
This completes the proof.

\section{Proofs of the nonsymmetric case}\label{sec:nonsymmetry}

This section contains the proof of Theorem~\ref{thm:weak-five-halves}.
Without symmetry, the reduction to a Rademacher quadratic form does not hold. However, the identity $1+T_n^2/n=n/R$ still holds on $\mathcal G$, which suggests that the asymptotic behavior of $T_n^2/n$ relies on the quadratic form $R=\mathbf{1}^{\top}\mathbf{P}\mathbf{1}$. Inspired by \cite{LiLiuXieZhou2026}, the projection matrices satisfy a recursion formula due to their intrinsic structure. Specifically, for a matrix $\mathbf{B}$, let
$\mathbf{J}=\mathbf{P}(\mathbf{B})=\mathbf{I}_n-\mathbf{B}^\top(\mathbf{B}\mathbf{B}^\top)^{-1}\mathbf{B}$ and $\mathbf{t}=\mathbf{J}\mathbf{1}$.
Adding an independent row $\mathbf{z}^{\top}$ forms a new matrix $(\mathbf{B};\mathbf{z}^{\top})$, which gives the exact update recursion
\begin{equation}\label{eq:five-halves-projection-recursion}
\mathbf{P}(\mathbf{B};\mathbf{z}^{\top})
=\mathbf{J}-\frac{\mathbf{J}\mathbf{z}\mathbf{z}^{\top}\mathbf{J}}
{\mathbf{z}^{\top}\mathbf{J}\mathbf{z}},
\qquad
R(\mathbf{B};\mathbf{z}^{\top})
=R(\mathbf{B})-
\frac{(\mathbf{t}^{\top}\mathbf{z})^2}
{\mathbf{z}^{\top}\mathbf{J}\mathbf{z}}.
\end{equation}
For completeness, the quotient above is set to zero if the denominator $\mathbf{z}^{\top}\mathbf{J}\mathbf{z}$ vanishes, which has exponentially small probability by Lemma \ref{lem:lower}. 
The exact update recursion in \eqref{eq:five-halves-projection-recursion} is a key ingredient in the proof. It controls both the error due to truncation and standardization and the error arising from the Lindeberg replacement in Section~\ref{sec:lindeberg}.
Specifically, we compare
its row decrements with those obtained from the entrywise sign-symmetrized
law, whose projection statistic satisfies the symmetric reference CLT established in Theorem \ref{thm:DAN}. Their diagonal terms agree exactly, whereas the signed off-diagonal terms are negligible under \eqref{eq:weak-five-halves}.
The comparison is first proved for a non-atomic triangular array $\zeta_n$ and then transferred to the original fixed law $\xi$ under \eqref{eq:weak-five-halves}.

\subsection{Truncation is negligible}
We begin with some estimates for truncated, smoothed, and standardized random variables, which are used frequently in the sequel. The smoothing argument is used to make the law non-atomic, which does not affect the asymptotic behavior of $R$ at the CLT scale by choosing an appropriate smoothing level. The proof is deferred to Appendix \ref{app:prf_truncation}.
\begin{lemma}\label{lem:five-halves-truncation}
Suppose $\mathbb{E}\xi=0$, $\mathbb{E}\xi^2=1$, and
\eqref{eq:weak-five-halves} holds. Let $U$ be independent of $\xi$
and uniform on $[-1,1]$, and define
\begin{gather}
 \tau_n=n^{2/3},\qquad \epsilon_n=n^{-3},\qquad
 \widehat\xi_n=\xi\mathbbm{1}_{\{|\xi|\le\tau_n\}},\qquad
 q_n=\mathbb{P}(|\xi|>\tau_n),\qquad
 \mu_n=\mathbb{E}\widehat\xi_n,
 \label{lt:eq:fixed-cutoff}\\
 \widetilde\xi_n=\widehat\xi_n+\epsilon_nU,\qquad
 s_n^2=\Var(\widehat\xi_n)+\epsilon_n^2/3,\qquad
 \zeta_n=(\widetilde\xi_n-\mu_n)/s_n.
 \label{lt:eq:fixed-array-law}
\end{gather}
Then $s_n\to1$, $\zeta_n$ is non-atomic with mean zero and variance one, and
\begin{gather}
 q_n=o(n^{-5/3}),\qquad |\mu_n|=o(n^{-1}),\qquad
 \mathbb{E}|\widehat\xi_n|^3=o(n^{1/3}),\qquad
 \mathbb{E}\widehat\xi_n^4=o(n),
 \label{lt:eq:fixed-cutoff-moments}\\
 l_n:=\mathbb{E}|\zeta_n|^3=o(n^{1/3}),\qquad
 k_n:=\mathbb{E}\zeta_n^4=o(n).
 \label{eq:five-halves-truncated-moments}
\end{gather}
For some fixed $n_0$, we have
\begin{equation}
 \lim_{x\to\infty}\sup_{n\ge n_0}
 x^{5/2}\mathbb{P}(|\zeta_n|>x)=0.
 \label{lt:eq:fixed-uniform-tail}
\end{equation}
Moreover, there is a bounded
nondecreasing function $\omega:(0,1]\to[0,\infty)$ with
$\omega(t)\to0$ as $t\rightarrow 0$ such that
\begin{equation}\label{sr:eq:signed-laplace-bound}
\left|\mathbb{E}\{\zeta_n e^{-s\zeta_n^2}\}\right|
\le\omega(t)s^{3/4},\qquad 0<s\le t\le1,
\end{equation}
uniformly in $n$.
\end{lemma}

Let $\mathbf{X}$ have i.i.d. entries with the law $\xi$ in Theorem~\ref{thm:weak-five-halves}. Couple $\mathbf{X}$, $\widehat{\mathbf{X}}$, $\widetilde{\mathbf{X}}$ and $\mathbf{Z}$ entrywise through \eqref{lt:eq:fixed-array-law}, with respective entry laws $\xi$, $\widehat\xi_n$, $\widetilde\xi_n$ and $\zeta_n$. For each of these matrices $\mathbf{A}$, use the notation $\mathbf{P}(\mathbf{A})$ for its null-space projection and $R(\mathbf{A})=\mathbf{1}^\top\mathbf{P}(\mathbf{A})\mathbf{1}$, as in Section~\ref{sec:definitions}. On $\mathcal F(\mathbf{A})$,
\[
 \mathbf{P}(\mathbf{A})=\mathbf{I}_n-\mathbf{A}^\top
            (\mathbf{A}\mathbf{A}^\top)^{-1}\mathbf{A}.
\]
With the exact update recursion and Lemma \ref{lem:five-halves-truncation}, 
the following lemma shows that, under \eqref{eq:weak-five-halves}, truncation at $n^{2/3}$ does not affect the asymptotic behavior of $R(\mathbf{X})$ at the CLT scale in Theorem \ref{thm:weak-five-halves}; the proof is deferred to Appendix \ref{app:prf:diff_truncation}.
\begin{lemma}\label{sat:lem:truncation}	
Let $\xi$ be a fixed nondegenerate real law with
$\mathbb{E}\xi=0$ and $\mathbb{E}\xi^2=1$, and suppose that
$p\leq\gamma_0n$ for a fixed $\gamma_0<1$.  Let $\tau_n\to\infty$,
and define
\[
 \widehat\xi=\xi\mathbbm{1}_{\{|\xi|\leq\tau_n\}},\qquad
 q_n=\mathbb{P}(|\xi|>\tau_n),\qquad
 \mu_n=\mathbb{E}\widehat\xi.
\]
Assume that $nq_n\to0$.  Couple $\mathbf{X}$ and
$\widehat{\mathbf{X}}$ entrywise using $\xi$ and $\widehat\xi$.
Then
\begin{equation}\label{sat:eq:transfer}
 \mathbb{E}|R(\mathbf{X})-R(\widehat{\mathbf{X}})|
 \leq Cpnq_n(1+n\mu_n^2)+Cn^C e^{-cn}.
\end{equation}
In particular, under \eqref{eq:weak-five-halves} with $\tau_n=n^{2/3}$ in Lemma \ref{lem:five-halves-truncation}, we have $\mathbb{E}|R(\mathbf{X})-R(\widehat{\mathbf{X}})|=o(n^{1/3})$.
\end{lemma}

\subsection{Conditional comparison with the symmetrized row}\label{sec:lindeberg}
With Lemma \ref{sat:lem:truncation}, it suffices to consider the truncated and standardized version ${\mathbf{Z}}$ with i.i.d. entries having the law of $\zeta_n$ in Lemma \ref{lem:five-halves-truncation}.
Let $\varepsilon$ be an independent Rademacher variable and set
$\zeta_n^\circ=\varepsilon\zeta_n$. This symmetrization preserves
the distribution of the square and hence $k_n$ and $l_n$.
Let $\mathbf{Z}^\circ$ have independent entries with the law of $\zeta_n^\circ$, and
write $\mathbf{P}^\circ=\mathbf{P}(\mathbf{Z}^\circ)$. We first verify the symmetric reference limit $R(\mathbf{Z}^{\circ})$ and then control the difference of $R(\mathbf{Z})$ and $R(\mathbf{Z}^{\circ})$. 

First, we show the asymptotic normality of $R(\mathbf{Z}^{\circ})$ by Lemma \ref{lem:sign}. Notice that the column-sign invariance gives, for an independent Rademacher vector
$\boldsymbol{\varepsilon}$,
\[
R(\mathbf{Z}^\circ)-m
\overset d=\boldsymbol{\varepsilon}^\top\mathbf{D}^\circ
\boldsymbol{\varepsilon},
\]
where $\mathbf{D}^\circ=\mathbf{P}^\circ-\diag(\mathbf{P}^\circ)$. To apply Lemma \ref{lem:sign}, it suffices to estimate $\operatorname{tr}[(\mathbf{D}^\circ)^2]$.
Delete two columns $j\ne k$ of $\mathbf{Z}^{\circ}$ and let $\mathbf{M}$ be the inverse Gram matrix of the remaining columns. By \eqref{sym:eq:hybrid-smallest} in Lemma \ref{lem:lower}, we have 
\[
\mathbb{P}(\|\mathbf{M}\|\ge C/n)\le Ce^{-cn}.
\]
The columns are independent, centered and variance one, with fourth
moment $k_n=o(n)$ by Lemma \ref{lem:five-halves-truncation}. The two-column Woodbury identity and the conditional
variance formula for a quadratic form therefore give
\begin{equation*}
	\begin{split}
	\mathbb{E}(P^\circ_{jj}-P^\circ_{kk})^2\le &\mathbb{E}(P^\circ_{jj}-P^\circ_{kk})^2\mathbbm{1}_{\|\mathbf{M}\|\le C/n}+\mathbb{P}(\|\mathbf{M}\|\ge C/n)\\ \le&\mathbb{E}|\mathbf{z}_j^\top\mathbf{M}\mathbf{z}_j
	-\mathbf{z}_k^\top\mathbf{M}\mathbf{z}_k|^2\mathbbm{1}_{\|\mathbf{M}\|\le C/n}+Ce^{-cn}\\
	=&\mathbb{E}\big((4\operatorname{tr}\mathbf{M}^2
	+2(k_n-3)\sum_iM_{ii}^2)\mathbbm{1}_{\|\mathbf{M}\|\le C/n}\big)+Ce^{-cn}
	\le Ck_n/n+Ce^{-cn}.
	\end{split}
\end{equation*}
Column exchangeability and $\sum_{j=1}^n(P^\circ_{jj}-m/n)=0$ now yield
\begin{equation}\label{sym:eq:symmetric-diagonal}
 \mathbb{E}(P^\circ_{jj}-m/n)^2
 =\frac{n-1}{2n}\mathbb{E}(P^\circ_{jj}-P^\circ_{kk})^2
 \le Ck_n/n+Ce^{-cn},
\end{equation}
which further implies
\[
 \operatorname{tr}[(\mathbf{D}^\circ)^2]
 =m-\sum_{j=1}^n(P^\circ_{jj})^2
 =\frac{mp}{n}-\sum_{j=1}^n(P^\circ_{jj}-m/n)^2
 =\frac{mp}{n}+o_{\mathbb{P}}(n)
\]
by Markov's inequality.
Conditional on $\mathbf{P}^\circ$, Lemma~\ref{lem:sign} applies with
$\sigma_{n,\circ}^2=2\operatorname{tr}[(\mathbf{D}^\circ)^2]$ and $\|\mathbf{D}^\circ\|\le2$. We have
$\sigma_{n,\circ}^2/(2mp/n)\pto1$ and
$\|\mathbf{D}^\circ\|/\sigma_{n,\circ}\pto0$. Thus, Lemma~\ref{lem:sign} yields
\begin{equation}\label{sym:eq:reference-clt}
 \frac{R(\mathbf{Z}^\circ)-m}{\sqrt{2mp/n}}\dto N(0,1).
\end{equation}
The same sign representation gives
$\mathbb{E}R(\mathbf{Z}^\circ)=m$ exactly.
In what follows, we compare $R(\mathbf{Z})$ with this symmetric
reference $R(\mathbf{Z}^{\circ})$ by replacing one row at a time.
In view of \eqref{eq:five-halves-projection-recursion}, the following lemma controls the conditional mean change caused by a row replacement. The proof relies on a regularized expansion of the denominator with exact cancellation of the leading diagonal term under sign symmetrization and Laplace-transform estimates for the remaining terms. The details are given in Appendix~\ref{app:prf:diff_symmetry}.

\begin{lemma}[Conditional mean comparison]
\label{sr:lem:row-mean}
	Under the assumptions of Lemma~\ref{lem:five-halves-truncation}, let $\mathbf z=(\zeta_{n1},\ldots,\zeta_{nn})^\top$ have independent coordinates with the law $\zeta_n$, and let $k_n,l_n$ be as in \eqref{eq:five-halves-truncated-moments}. 
Let $\mathbf{J}$ be a deterministic orthogonal projection on
$\mathbb{R}^n$ of rank $r\ge\delta n$, where $\delta>0$ is fixed.
Let $\mathbf{u}$ satisfy $\mathbf{J}\mathbf{u}=\mathbf{u}$, and set
$a=\|\mathbf{u}\|^2\le a_0$ for a fixed $a_0$. For independent Rademacher
variables $\varepsilon_j$, define
\[
 \mathbf{z}^\circ=(\varepsilon_1z_1,\ldots,\varepsilon_nz_n)^\top,
 \qquad
 \theta(\mathbf{z})=
 \frac{(\mathbf{u}^\top\mathbf{z})^2}{\mathbf{z}^\top\mathbf{J}\mathbf{z}},
\]
with value zero if $\mathbf{z}^\top\mathbf{J}\mathbf{z}=0$. Then, uniformly over
$\mathbf{J},\mathbf{u}$,
\begin{equation}\label{sr:eq:mean-comparison}
 \left|\mathbb{E}\theta(\mathbf{z})
       -\mathbb{E}\theta(\mathbf{z}^\circ)\right|
 =o(a n^{-3/2}).
\end{equation}
More precisely, there is a deterministic $\rho_n\to0$, depending
only on the family of laws and the fixed constants, such that the
left-hand side is at most
\begin{equation}\label{sr:eq:mean-comparison-rate}
 Ca\left\{
 \frac{\rho_n}{n^{3/2}}+\frac1{n^2}
 +\frac{l_n^2}{n^{5/2}}+\frac{l_n^2}{n^3}
 +\frac{l_n^4}{n^4}+\frac{k_n}{n^3}
 +n^2e^{-cn}\right\}.
\end{equation}
The conclusion holds conditionally for random $\mathbf{J},\mathbf{u}$
independent of the row being compared.
\end{lemma}

Next, we consider the Lindeberg replacement argument row by row. Let $\mathbf{Z}^{(i)}$ use the law of $\zeta_n$ in its first $i$ rows
and the law of $\zeta_n^\circ$ in its remaining rows.  Thus
$\mathbf{Z}^{(0)}\overset d=\mathbf{Z}^\circ$ and
$\mathbf{Z}^{(p)}\overset d=\mathbf{Z}$.
For the comparison of $\mathbf{Z}^{(i-1)}$ and $\mathbf{Z}^{(i)}$,
condition on their $p-1$ common rows, whose projection is defined by $\mathbf{J}$. Note that the entries of each intermediate matrix $\mathbf{Z}^{(i)}$ are independent but need not be identically distributed. 
For $1\le i\le p$, we call the matrix formed by the $p-1$ rows shared by $\mathbf{Z}^{(i-1)}$ and $\mathbf{Z}^{(i)}$ a common-row matrix, and refer to the remaining row in either matrix as a fresh row. It follows that a fresh row consists of i.i.d. components with the law either of $\zeta_n$ or $\zeta_n^{\circ}$.
Moreover, one can verify that the common-row matrices satisfy case~(i) of Lemma~\ref{lem:projection-moments} in Appendix~\ref{app:projection-moments}. Its uniform fourth-moment estimate controls the increments of row replacement below. 

We first formulate the increment estimates for the common rows of a single row replacement.  Let $\mathbf{B}$ be any permitted
$(p-1)\times n$ matrix and set
\[
\mathbf{J}=\mathbf{P}(\mathbf{B}),\qquad r=n-p+1=m+1,\qquad
\mathbf{t}=\mathbf{J}\mathbf{1},\qquad
\mathbf{u}=\frac{\mathbf{t}}{\sqrt r},\qquad
a=\|\mathbf{u}\|^2=\frac{\mathbf{1}^\top\mathbf{J}\mathbf{1}}{r}.
\]
In particular, $0\leq a\leq n/r\leq C$ deterministically.  For an
independent row with either law, write
\begin{equation}\label{sym:eq:increment-definitions}
A=(\mathbf{u}^\top\mathbf{z})^2,\qquad
V=\mathbf{z}^\top\mathbf{J}\mathbf{z},\qquad
\Delta=rA/V.
\end{equation}
The ratio is zero if $V=0$. It follows that $0\leq\Delta\leq n$. By some elementary algebra, we get
\begin{equation}\label{sym:eq:increment-moments}
\mathbb{E}(A\mid\mathbf{J})=a,\qquad
\mathbb{E}(A^2\mid\mathbf{J})
=3a^2+(k_n-3)\sum_{j=1}^n u_j^4,
\qquad \mathbb{E}\sum_{j=1}^n u_j^4\leq C/n,
\end{equation}
where the last estimate follows from  \eqref{sym:eq:t-fourth} in
Lemma~\ref{lem:projection-moments} and the fact  $\mathbf{u}=\mathbf{t}/\sqrt r$ with
$r=m+1\asymp n$. 
All assertions and limits below are uniform over the choices of the
common-row laws and over the two possibilities for the replaced row. The proof is given in Appendix \ref{app:prf_uniformint}.

\begin{lemma}\label{sym:lem:increments}
	Under the settings of Lemma \ref{lem:five-halves-truncation}, the following estimates hold, with all suprema taken over the intermediate matrices and fresh-row law choices:
	\begin{align}
	&\sup\mathbb{E}\Delta^2\leq C,
	\label{sym:eq:Delta-second}\\
	&\lim_{K\to\infty}\limsup_{n\to\infty}
	\sup\mathbb{E}\bigl[A^2\mathbbm{1}_{\{A>K\}}\bigr]=0,
	\label{sym:eq:A-UI}\\
	&\sup\mathbb{E}(\Delta-A)^2\rightarrow 0,
	\label{sym:eq:Delta-A-L2}\\
	&\lim_{K\to\infty}\limsup_{n\to\infty}
	\sup\mathbb{E}\bigl[\Delta^2\mathbbm{1}_{\{\Delta>K\}}\bigr]=0.
	\label{sym:eq:Delta-UI}
	\end{align}
\end{lemma}

With Lemma \ref{sr:lem:row-mean}, we now proceed to show the asymptotic normality of $R(\mathbf{Z})$. Use the notation in \eqref{sym:eq:increment-definitions}, adding a superscript $\circ$ for the replacing symmetric row, i.e., $A^{\circ},V^{\circ}$ and $\Delta^{\circ}$. In view of \eqref{eq:five-halves-projection-recursion}, denote 
$R_-=\mathbf{1}^\top\mathbf{J}\mathbf{1}$ for simplicity. Then the two updates can be represented as 
\[
R_--\Delta~~\text{and}~~R_--\Delta^{\circ}.
\]
Equation~\eqref{sr:eq:mean-comparison} in Lemma~\ref{sr:lem:row-mean}, together with $r\asymp n$ and $a\le C$, gives
\begin{equation}\label{sym:eq:decrement-mean-difference}
\left|\mathbb{E}(\Delta\mid\mathbf{J})
-\mathbb{E}(\Delta^\circ\mid\mathbf{J})\right|
\leq C\rho_n n^{-1/2},
\end{equation}
which further implies
\[
|\mathbb{E} R(\mathbf{Z})-m|
\le Cp\rho_n n^{-1/2}=o(\sqrt n).
\]
Moreover, the conditional second moments of $A$ and $A^\circ$ are
identical by \eqref{sym:eq:increment-moments}, which together with  \eqref{sym:eq:Delta-A-L2} and the Cauchy--Schwarz
inequality further implies
\begin{align}\label{sym:eq:decrement-second-difference}
\sup_{1\leq i\leq p}\mathbb{E}\left|
\mathbb{E}(\Delta^2\mid\mathbf{J})
-\mathbb{E}((\Delta^\circ)^2\mid\mathbf{J})\right|\to0.
\end{align}
Fix a bounded $C^3(\mathbb{R})$ test function $f$ with bounded derivatives, and
set $x=(R_--m)/\sqrt{2v}$, where $v=mp/n$ is defined in \eqref{eq:PQR}. In what follows, we shall bound $\mathbb{E}f(x-\Delta/\sqrt{2v})-\mathbb{E}f(x-\Delta^{\circ}/\sqrt{2v})$.
By Taylor's formula, for $h\geq0$,
\[
f(x-h)=f(x)-hf'(x)+\tfrac12h^2f''(x)+\mathcal R_f(x,h),
\qquad |\mathcal R_f(x,h)|\leq\tfrac12h^2\omega_f(h),
\]
where $\omega_f(h):=\sup_{|y-z|\leq h}|f''(y)-f''(z)|$ is bounded
and tends to zero as $h\rightarrow 0$.  The coefficients $f'(x)$ and
$f''(x)$ are measurable with respect to the common rows.
For the first-order terms, summing
\eqref{sym:eq:decrement-mean-difference} over the $p$ replacements
gives an absolute error at most
\[
\frac{Cp\|f'\|_\infty\rho_n}{\sqrt{2v}\sqrt n}=o(1).
\]
For the second-order terms, \eqref{sym:eq:decrement-second-difference}
and $p/(2v)=O(1)$ give an error tending to zero.
Finally, \eqref{sym:eq:Delta-UI} implies
\[
\sup_{1\le i\le p}\mathbb{E}\left[\Delta^2\omega_f(\Delta/\sqrt{2v})
+(\Delta^\circ)^2\omega_f(\Delta^\circ/\sqrt{2v})\right]\to0.
\]
To verify this, restrict first to $\Delta\leq K$, where the
expectation is bounded by $C\omega_f(K/\sqrt{2v})$, and use uniform
integrability on $\Delta>K$; the symmetric term is the same. The sum of the Taylor remainders, each multiplied by $(2v)^{-1}$,
therefore tends to zero. It follows that
\begin{equation}\label{sym:eq:smooth-statistic-comparison}
\mathbb{E} f\left(\frac{R(\mathbf{Z})-m}{\sqrt{2v}}\right)
-\mathbb{E} f\left(\frac{R(\mathbf{Z}^\circ)-m}{\sqrt{2v}}\right)
\rightarrow  0,
\end{equation}
To deduce weak convergence from \eqref{sym:eq:smooth-statistic-comparison},
take $f(x)=\cos(tx)$ and $f(x)=\sin(tx)$ for each fixed
$t\in\mathbb{R}$. Both functions are bounded and have bounded
derivatives of every order. Thus the characteristic function of
$(R(\mathbf{Z})-m)/\sqrt{2v}$ converges to $e^{-t^2/2}$ by
\eqref{sym:eq:reference-clt}, which further implies
\begin{equation}\label{sym:eq:R-clt}
\frac{R(\mathbf{Z})-m}{\sqrt{2mp/n}}\dto N(0,1).
\end{equation}

\subsection{Completion of the proof of Theorem~\ref{thm:weak-five-halves}}\label{sec:prf:weak52}
Lemma~\ref{lem:five-halves-truncation} and the comparison above give \eqref{sym:eq:R-clt} for $\mathbf{Z}$. Since $q_n=o(n^{-5/3})$ and $|\mu_n|=o(n^{-1})$, Lemma~\ref{sat:lem:truncation} gives
\begin{equation}
 \mathbb{E}|R(\mathbf{X})-R(\widehat{\mathbf{X}})|
 \le Cn^2q_n(1+n\mu_n^2)+Cn^C e^{-cn}
 =o(n^{1/3})
 \label{lt:eq:fixed-truncation-error}
\end{equation}
by \eqref{sat:eq:transfer} in Lemma \ref{sat:lem:truncation}. Next, we shall control the error due to smoothing and centering. Notice the following elementary projection estimate. If a matrix
$\mathbf{A}_0$ has full row rank and
$\|\mathbf{E}\|\le s_p(\mathbf{A}_0)/2$, then
\begin{equation}
 \|\mathbf{P}(\mathbf{A}_0+\mathbf{E})-\mathbf{P}(\mathbf{A}_0)\|
 \le \frac{3\|\mathbf{E}\|}{s_p(\mathbf{A}_0)},\qquad
 |R(\mathbf{A}_0+\mathbf{E})-R(\mathbf{A}_0)|
 \le \frac{3n\|\mathbf{E}\|}{s_p(\mathbf{A}_0)}.
 \label{lt:eq:fixed-projection-perturbation}
\end{equation}
In fact, let $\mathbf{A}_1=\mathbf{A}_0+\mathbf{E}$ and
$\mathbf{L}_i=\mathbf{I}-\mathbf{P}(\mathbf{A}_i)$. It follows that
$s_p(\mathbf{A}_1)\ge s_p(\mathbf{A}_0)/2$ and
\[
 \|(\mathbf{I}-\mathbf{L}_1)\mathbf{L}_0\|
 =\|(\mathbf{I}-\mathbf{L}_1)(\mathbf{A}_0^\top-\mathbf{A}_1^\top)(\mathbf{A}_0\mathbf{A}_0^\top)^{-1}\mathbf{A}_0\|\le \|\mathbf{E}\|/s_p(\mathbf{A}_0),
\]
which further implies \eqref{lt:eq:fixed-projection-perturbation} by $\mathbf{L}_0-\mathbf{L}_1=(\mathbf{I}-\mathbf{L}_1)\mathbf{L}_0-(\mathbf{I}-\mathbf{L}_0)\mathbf{L}_1$.
Applying \eqref{lt:eq:external-i.i.d.-inverse} in Lemma~\ref{lem:lower} to $\widehat{\mathbf{X}}^\top$, we get
\[
\mathbb{P}(s_p(\widehat{\mathbf{X}})\le c\sqrt n)\le Ce^{-cn}
\]
The smoothing perturbation has norm at most
$\epsilon_n\sqrt{pn}$. Thus
\begin{equation}
 \mathbb{E}|R(\widetilde{\mathbf{X}})-R(\widehat{\mathbf{X}})|
 \le C\epsilon_n n^{3/2}+Cne^{-cn}=o(\sqrt n).
 \label{lt:eq:fixed-smoothing-error}
\end{equation}
Since scalar multiplication does not change a row space, we can write
\[
 R(\widetilde{\mathbf{X}})
 =R(\mathbf{Z}+d_n\mathbf{1}_p\mathbf{1}_n^\top),\qquad
 d_n=\mu_n/s_n=o(n^{-1}).
\]
Similarly, applying
\eqref{lt:eq:fixed-projection-perturbation} with perturbation norm
$|d_n|\sqrt{pn}$ gives
\begin{equation}
 \mathbb{E}|R(\widetilde{\mathbf{X}})-R(\mathbf{Z})|
 \le C|d_n|n^{3/2}+Cne^{-cn}=o(\sqrt n).
 \label{lt:eq:fixed-centering-error}
\end{equation}
Combining \eqref{lt:eq:fixed-truncation-error}, \eqref{lt:eq:fixed-smoothing-error}, and
\eqref{lt:eq:fixed-centering-error}, we obtain
\begin{equation}
 \mathbb{E}|R(\mathbf{X})-R(\mathbf{Z})|=o(\sqrt n).
 \label{lt:eq:fixed-L1-transfer}
\end{equation}

Finally, \eqref{eq:R-CLT} and $\mathbb{E} R=m+o(\sqrt n)$ follow from \eqref{sym:eq:R-clt} and \eqref{lt:eq:fixed-L1-transfer}. 
In particular, $R/m\to1$ in probability. In view of Lemma \ref{lem:rank}, on $\mathcal G$ with $\mathbb{P}(\mathcal{G}^c)\le Ce^{-cn}$,
\eqref{eq:identity} gives the exact relation
\begin{equation}
\sqrt{\frac{m^3}{2pn}}
\left(\frac{T_n^2}{n}-\frac pm\right)
=-\frac mR\,\frac{R-m}{\sqrt{2mp/n}},
\label{lt:eq:fixed-T-transform}
\end{equation}
which gives \eqref{eq:T-CLT}. Similarly,
\[
\log(1+T_n^2/n)+\log(m/n)=-\log(R/m),
\]
and Taylor expansion at one proves \eqref{eq:all-dimension-log}.
Since $p/m\to\gamma/(1-\gamma)$, its normalization also gives
\eqref{eq:log-CLT}.
Finally, let
\[
Z_n^T=\sqrt{\frac{m^3}{2pn}}\left(\frac{T_n^2}{n}-\frac pm\right),
\qquad
Z_n^{\mathrm G}=\sqrt{\frac{m^3}{2pn}}\left(G_{p,m}-\frac pm\right).
\]
Thus, we have $Z_n^T\dto N(0,1)$ and $Z_n^{\mathrm G}\dto N(0,1)$. Since both limiting distribution functions are continuous, these two convergences are uniform in the distribution-function argument. Affine invariance of Kolmogorov distance and the triangle inequality therefore give
\[
\varepsilon_{n,p}=d_K\big(Z_n^T,Z_n^{\mathrm G}\big)
\le d_K\big(Z_n^T,N(0,1)\big)+d_K\big(Z_n^{\mathrm G},N(0,1)\big)\to0,
\]
completing the proof of Theorem~\ref{thm:weak-five-halves}.

\appendix

\section{Auxiliary proofs}\label{app:lemma-proofs}

\subsection{Proof of Lemma~\ref{lem:rank}}\label{app:lem:rank}
\begin{proof}
Let $\mathcal S\subset\mathbb R^n$ be a fixed $d$-dimensional subspace, where $0\le d<n$, and let $\mathbf z=(z_1,\ldots,z_n)^\top$ have independent coordinates with law $\xi$. We first show that
\begin{equation}\label{eq:subspace-atom-bound}
\mathbb P(\mathbf z\in\mathcal S)\le q^{n-d}.
\end{equation}
For $d=0$, this follows directly from independence. For $d\ge1$, form a $d\times n$ matrix from a basis of $\mathcal S$. A nonsingular $d\times d$ minor determines a set $J$ of $d$ coordinate positions on which the coordinate projection is one-to-one. Consequently, for each $k\notin J$, there is a linear function $f_k$ such that $x_k=f_k((x_j)_{j\in J})$ for every $\mathbf x\in\mathcal S$. Conditioning on $(z_j)_{j\in J}$ and using independence of the remaining coordinates gives
\[
\mathbb P\big(\mathbf z\in\mathcal S\mid(z_j)_{j\in J}\big)
\le\prod_{k\notin J}\sup_{x\in\mathbb R}\mathbb P(z_k=x)
\le q^{n-d}.
\]
Taking expectations gives \eqref{eq:subspace-atom-bound}. The same estimate applies conditionally to a random subspace independent of $\mathbf z$.

Let $\mathbf r_1^\top,\ldots,\mathbf r_p^\top$ be the rows of $\mathbf X$, and set
\[
\mathcal S_{i-1}=\operatorname{span}\{\mathbf1^\top,\mathbf r_1^\top,\ldots,\mathbf r_{i-1}^\top\},
\qquad \mathscr H_{i-1}=\sigma(\mathbf r_1,\ldots,\mathbf r_{i-1}),
\quad 1\le i\le p.
\]
The events $E_i=\{\dim\mathcal S_{i-1}=i,\ \mathbf r_i^\top\in\mathcal S_{i-1}\}$ describe the first rank failure and are mutually disjoint, with $\mathcal G^c=\bigcup_{i=1}^pE_i$. Since $\mathbf r_i$ is independent of $\mathscr H_{i-1}$, applying \eqref{eq:subspace-atom-bound} conditionally gives
\[
\mathbb P(E_i)=\mathbb E\left[
\mathbbm1_{\{\dim\mathcal S_{i-1}=i\}}
\mathbb P(\mathbf r_i^\top\in\mathcal S_{i-1}\mid\mathscr H_{i-1})\right]
\le q^{n-i}.
\]
Therefore,
\[
\eta_n=\sum_{i=1}^p\mathbb P(E_i)
\le\sum_{i=1}^p q^{n-i}
=q^{n-p}\frac{1-q^p}{1-q}
\le\frac{q^{n-p}}{1-q}.
\]
When $q=0$, the same argument gives $\eta_n=0$. Finally, $\mathcal G\subseteq\mathcal F$ implies $\mathbb P(\mathcal F^c)\le\eta_n$.
\end{proof}

\subsection{Proof of Lemma~\ref{lem:sign}}\label{app:prf_sign}
\begin{proof}
	We first consider \eqref{eq:sign-bound}. Let $\mathbf{g}=(g_1,\ldots,g_n)^\top$ have independent standard normal coordinates, independently of $\boldsymbol{\varepsilon}$. For $0\le j\le n$, define the intermediate vectors
	\[
	\mathbf{x}^{(j)}=(g_1,\ldots,g_j,\varepsilon_{j+1},\ldots,\varepsilon_n)^\top,
	\qquad
	Q_j=\frac{(\mathbf{x}^{(j)})^\top\mathbf{D}\mathbf{x}^{(j)}}s.
	\]
	Thus $Q_0=\boldsymbol{\varepsilon}^\top\mathbf{D}\boldsymbol{\varepsilon}/s$ and $Q_n=\mathbf{g}^\top\mathbf{D}\mathbf{g}/s$. In the $j$-th replacement, all coordinates except the $j$-th one agree. Denoting their common values by $x_k^{[j]}$, set
	\[
	U_j=\sum_{\substack{1\le k,l\le n\\k\ne j,\ l\ne j}}
	D_{kl}x_k^{[j]}x_l^{[j]},
	\qquad
	V_j=\sum_{\substack{1\le k\le n\\k\ne j}}D_{jk}x_k^{[j]}.
	\]
	Since $D_{jj}=0$,
	\[
	Q_{j-1}=(U_j+2\varepsilon_jV_j)/s,
	\qquad Q_j=(U_j+2g_jV_j)/s.
	\]
	The pair $(U_j,V_j)$ is independent of both $\varepsilon_j$ and $g_j$.
	
	For $f\in C_b^3(\mathbb{R})$, the class of bounded three times continuously differentiable functions with bounded derivatives up to order three, Taylor's formula gives
	\[
	f(u+t)=f(u)+tf'(u)+\frac{t^2}{2}f''(u)
	+\frac{t^3}{2}\int_0^1(1-v)^2f^{(3)}(u+vt)\mathrm{d}v.
	\]
	The last term has absolute value at most $\|f^{(3)}\|_\infty|t|^3/6$. Conditional on the common coordinates, applying this formula at $u=U_j/s$ with $t=2\varepsilon_jV_j/s$ and $t=2g_jV_j/s$ makes the linear and quadratic terms cancel after expectation: $\mathbb{E}\varepsilon_j=\mathbb{E}g_j=0$ and $\mathbb{E}\varepsilon_j^2=\mathbb{E}g_j^2=1$. Consequently,
	\[
	\big|\mathbb{E}f(Q_{j-1})-\mathbb{E}f(Q_j)\big|
	\le C\|f^{(3)}\|_\infty s^{-3}\mathbb{E}|V_j|^3.
	\]
	All coordinates appearing in $V_j$ are independent, centered, have variance one, and have fourth moment at most three. Setting
	$d_j^2=\sum_{k=1}^nD_{jk}^2$, expansion of the fourth moment gives
	\begin{align*}
	\mathbb{E}V_j^4
	&=\sum_{\substack{1\le k\le n\\k\ne j}}
	D_{jk}^4\mathbb{E}(x_k^{[j]})^4
	+6\sum_{\substack{1\le k<l\le n\\k,l\ne j}}D_{jk}^2D_{jl}^2
	\le3d_j^4,
	\end{align*}
	which gives $\mathbb{E}|V_j|^3\le3^{3/4}d_j^3$. Since $d_j=\|\mathbf{D}\mathbf{e}_j\|\le\|\mathbf{D}\|$ and $\sum_{j=1}^nd_j^2=\tr\mathbf{D}^2=s^2/2$, it follows that
	\begin{equation}\label{eq:sign-smooth-replacement}
	\big|\mathbb{E}f(Q_0)-\mathbb{E}f(Q_n)\big|\le \sum_{j=1}^n\big|\mathbb{E}f(Q_{j-1})-\mathbb{E}f(Q_j)\big|
	\le C\|f^{(3)}\|_\infty s^{-3}\sum_{j=1}^n d_j^3
	\le C\|f^{(3)}\|_\infty\frac{\|\mathbf{D}\|}{s}.
	\end{equation}
	Let $\lambda_1,\ldots,\lambda_n$ be the eigenvalues of $\mathbf{D}$. Orthogonal invariance of the standard Gaussian vector and $\tr\mathbf{D}=0$ give
	\[
	Q_n\stackrel d=\frac1s\sum_{j=1}^n\lambda_j(g_j^2-1).
	\]
	The summands are independent and centered, and their variances sum to one. The Berry--Esseen inequality for independent summands consequently gives
	\begin{equation}\label{eq:Gaussian-quadratic-BE}
	d_K\big(Q_n,N(0,1)\big)
	\le C\frac{\sum_{j=1}^n|\lambda_j|^3}{s^3}
	\le C\frac{\|\mathbf{D}\|}{s},
	\end{equation}
	since $\mathbb{E}|g_1^2-1|^3<\infty$ and
	$\sum_{j=1}^n|\lambda_j|^3\le\|\mathbf{D}\|\sum_{j=1}^n\lambda_j^2$.

	To show \eqref{eq:sign-bound}, it suffices to transfer the smooth comparison to distribution functions. Fix a nonnegative $C^\infty$ function $\rho$ supported in $(0,1)$ with integral one, and put $\psi(t)=\int_t^\infty\rho(u)\mathrm{d}u$. Then $\psi$ is nonincreasing, equals one on $(-\infty,0]$, and equals zero on $[1,\infty)$. For $x\in\mathbb{R}$ and $h>0$, define
	\[
	f_{x,h}(t)=\psi\big((t-x)/h\big).
	\]
	It satisfies
	\[
	\mathbbm{1}_{(-\infty,x]}(t)\le f_{x,h}(t)
	\le\mathbbm{1}_{(-\infty,x+h]}(t),
	\qquad \|f_{x,h}^{(3)}\|_\infty\le Ch^{-3}.
	\]
	Setting $a=\|\mathbf{D}\|/s$, applying \eqref{eq:sign-smooth-replacement} to $f_{x,h}$ and $f_{x-h,h}$, and then using \eqref{eq:Gaussian-quadratic-BE}, gives
	\begin{align*}
	\mathbb{P}(Q_0\le x)
	&\le\mathbb{E}f_{x,h}(Q_0)
	\le\mathbb{E}f_{x,h}(Q_n)+Cah^{-3}
	\le\Phi(x+h)+Ca+Cah^{-3},\\
	\mathbb{P}(Q_0\le x)
	&\ge\mathbb{E}f_{x-h,h}(Q_0)
	\ge\mathbb{E}f_{x-h,h}(Q_n)-Cah^{-3}
	\ge\Phi(x-h)-Ca-Cah^{-3}.
	\end{align*}
	Since the standard normal density is bounded, the inequalities above imply
	\[
	d_K\big(Q_0,N(0,1)\big)\le C\big(ah^{-3}+a+h\big).
	\]
	Moreover $0<a\le1/\sqrt2$, because $s^2=2\|\mathbf{D}\|_F^2$. Choosing $h=a^{1/4}$  yields \eqref{eq:sign-bound}.
	
	For \eqref{eq:sign-fourth}, let $\mathbf{A}=(|D_{jk}|)_{1\le j,k\le n}$. Expanding the fourth power and taking absolute values term by term give
	\[
	\mathbb{E}(\boldsymbol{\varepsilon}^\top\mathbf{D}\boldsymbol{\varepsilon})^4
	\le\mathbb{E}(\mathbf{g}^\top\mathbf{A}\mathbf{g})^4.
	\]
	The matrix $\mathbf{A}$ is symmetric and satisfies $\operatorname{tr}\mathbf{A}=0$. If $\alpha_1,\ldots,\alpha_n$ are its eigenvalues, then the right-hand side equals the fourth moment of $\sum_{j=1}^n\alpha_j(g_j^2-1)$. Using $\mathbb{E}(g_j^2-1)^2=2$ and $\mathbb{E}(g_j^2-1)^4=60$ gives
	\begin{align*}
	\mathbb{E}(\mathbf{g}^\top\mathbf{A}\mathbf{g})^4
	=60\sum_{j=1}^n\alpha_j^4
	+24\sum_{1\le j<k\le n}\alpha_j^2\alpha_k^2
	=12(\tr\mathbf{A}^2)^2+48\tr\mathbf{A}^4
	\le60(\tr\mathbf{D}^2)^2,
	\end{align*}
	where we used $\tr\mathbf{A}^2=\tr\mathbf{D}^2$ and
	$\tr\mathbf{A}^4\le(\tr\mathbf{A}^2)^2$. This completes the proof of \eqref{eq:sign-fourth}.
	
	For the fixed-rank comparison \eqref{eq:sign-fixed}, define $\mathbf{u}_j=\mathbf{U}^\top\mathbf{e}_j\in\mathbb{R}^k$, for $1\le j\le n$. Then
	\[
	\sum_{j=1}^n\mathbf{u}_j\mathbf{u}_j^\top=\mathbf{I}_k,
	\qquad \sum_{j=1}^n\|\mathbf{u}_j\|^2=k,
	\qquad \max_{1\le j\le n}\|\mathbf{u}_j\|^2=b.
	\]
	In the $j$-th replacement, let
	$\mathbf{S}_j=\sum_{l=1}^{j-1}g_l\mathbf{u}_l+
	\sum_{l=j+1}^n\varepsilon_l\mathbf{u}_l$.
	Applying the same conditional Taylor expansion to $t\mapsto f(\mathbf{S}_j+t\mathbf{u}_j)$ cancels the linear and quadratic terms after expectation and bounds the replacement error by $C\|\mathrm{D}^3f\|_\infty\|\mathbf{u}_j\|^3$, where $\|\mathrm{D}^3f\|_\infty$ is the supremum of the operator norm of the third derivative.
	Using $\sum_{j=1}^n\|\mathbf{u}_j\|^3\le k\sqrt b$ therefore gives
	\begin{equation}\label{eq:fixed-rank-smooth}
	\left|\mathbb{E}f(\mathbf{U}^\top\boldsymbol{\varepsilon})
	-\mathbb{E}f(\mathbf{G}_k)\right|
	\le Ck\sqrt b\,\|\mathrm{D}^3f\|_\infty,
	\qquad \mathbf{G}_k\sim N_k(\mathbf{0},\mathbf{I}_k).
	\end{equation}
	To approximate the distribution functions of the norms, let $B_r$ be the closed Euclidean ball of radius $r$ in $\mathbb{R}^k$, with $B_r=\varnothing$ for $r<0$. Choose a nonnegative smooth probability density $\rho$ supported in the unit ball, let $\rho_h(\mathbf{x})=h^{-k}\rho(\mathbf{x}/h)$, and define $f_{r,h}=\mathbbm{1}_{B_r}*\rho_h$. Then
	\[
	\mathbbm{1}_{B_{r-h}}\le f_{r,h}\le\mathbbm{1}_{B_{r+h}},
	\qquad \|\mathrm{D}^3f_{r,h}\|_\infty\le \int_{\mathbb{R}^k}|\mathrm{D}^3\rho_h(u)|\mathrm{d}u\le h^{-3}\int_{\mathbb{R}^k}|\mathrm{D}^3\rho(u)|\mathrm{d}u\le C_kh^{-3}
	\]
	by Young's inequality.
	Let $F_S(r)=\mathbb{P}(\|\mathbf{U}^\top\boldsymbol{\varepsilon}\|\le r)$ and $F_G(r)=\mathbb{P}(\|\mathbf{G}_k\|\le r)$, both set to zero for $r<0$. Applying \eqref{eq:fixed-rank-smooth} to the upper approximation $f_{r+h,h}$ and the lower approximation $f_{r-h,h}$ gives, for $r\ge0$,
	\[
	F_G(r-2h)-C_k\sqrt b\,h^{-3}
	\le F_S(r)\le F_G(r+2h)+C_k\sqrt b\,h^{-3}.
	\]
	When $r<h$, the lower bound follows directly from $F_S(r)\ge0$. The density of $\|\mathbf{G}_k\|$ is proportional to $r^{k-1}e^{-r^2/2}$ on $[0,\infty)$ and is bounded for each fixed $k\ge1$. Thus
	\[
	\sup_{r\ge0}|F_S(r)-F_G(r)|
	\le C_k\big(\sqrt b\,h^{-3}+h\big).
	\]
	In view of the fact that $\|\mathbf{G}_k\|^2\sim\chi_k^2$, choosing $h=b^{1/8}$ and replacing $r$ by $\sqrt x$ gives \eqref{eq:sign-fixed}, finishing the proof of the lemma.
\end{proof}

\subsection{Proof of Proposition~\ref{thm:comparison}}\label{app:prf_proposition}
\begin{proof}
	On the event $\mathcal{F}$, set $\mathbf{D}=\mathbf{P}-\diag \mathbf{P}$. It follows that
	\[
	\boldsymbol{\varepsilon}^\top \mathbf{P}\boldsymbol{\varepsilon}=\boldsymbol{\varepsilon}^\top \mathbf{D}\boldsymbol{\varepsilon}+\operatorname{tr}\mathbf{P}=\boldsymbol{\varepsilon}^\top \mathbf{D}\boldsymbol{\varepsilon}+m.
	\]
	By independence of the signs, $\mathbb{E}_{\boldsymbol{\varepsilon}}(\boldsymbol{\varepsilon}^\top \mathbf{D}\boldsymbol{\varepsilon})=0$ and
	\begin{equation}
	\Var_{\boldsymbol{\varepsilon}}(\boldsymbol{\varepsilon}^\top \mathbf{D}\boldsymbol{\varepsilon})
	=2\sum_{\substack{1\le j,k\le n\\j\ne k}}P_{jk}^2
	=2(m-\sum_{j=1}^nP_{jj}^2)
	=2v(1-\delta_{n,p}),
	\end{equation}
	where we used $\tr \mathbf{P}^2=\tr \mathbf{P}=m$ and the definition of $\delta_{n,p}$ in \eqref{eq:delta}.

	Since $\operatorname{tr}\mathbf{D}^2=v(1-\delta_{n,p})$ and $\|\mathbf{D}\|\le2$, if $\delta_{n,p}\le1/2$ we have
	$s=\sqrt{2\operatorname{tr}\mathbf{D}^2}\ge\sqrt{v}$, which together with Lemma~\ref{lem:sign} implies
	\[
	d_K(\boldsymbol{\varepsilon}^\top \mathbf{D}\boldsymbol{\varepsilon}/s,N(0,1))
	\le C(\|\mathbf{D}\|/s)^{1/4}\le Cv^{-1/8}.
	\]
	For simplicity, set $\sigma=\sqrt{1-\delta_{n,p}}$. If \(\delta_{n,p}\le1/2\), we have
	\(1/\sqrt2\le\sigma\le1\), and can write
	\[
	\frac{R^{\rm s}-m}{\sqrt{2v}}
	=
	\sigma\,
	\frac{\boldsymbol{\varepsilon}^\top\mathbf D\boldsymbol{\varepsilon}}{s}.
	\]
	For every \(x\in\mathbb R\),
	\begin{align*}
	\left|
	\mathbb P_{\boldsymbol{\varepsilon}}\left(
	\frac{R^{\rm s}-m}{\sqrt{2v}}\le x
	\,\middle|\,\mathbf X
	\right)-\Phi(x)
	\right|
	\le
	\left|
	\mathbb P_{\boldsymbol{\varepsilon}}\left(
	\frac{\boldsymbol{\varepsilon}^\top\mathbf D\boldsymbol{\varepsilon}}{s}
	\le \frac{x}{\sigma}
	\,\middle|\,\mathbf X
	\right)
	-\Phi\left(\frac{x}{\sigma}\right)
	\right|
	+
	\left|
	\Phi\left(\frac{x}{\sigma}\right)-\Phi(x)
	\right|.
	\end{align*}
	The first term is bounded by \(Cv^{-1/8}\). For the second one, consider
	\(g_x(u)=\Phi(x/u)\), \(u\in[1/\sqrt2,1]\). Since
	\[
	|g_x'(u)|
	=
	\frac{|x|}{u^2}\phi(x/u)
	=
	\frac1u\left|\frac xu\right|\phi(x/u)
	\le C,
	\]
	uniformly in \(x\in\mathbb R\) and \(u\in[1/\sqrt2,1]\), the mean-value theorem gives
	\[
	\sup_{x\in\mathbb R}
	\left|
	\Phi\left(\frac{x}{\sigma}\right)-\Phi(x)
	\right|
	\le C(1-\sigma)= C\frac{\delta_{n,p}}
	{1+\sqrt{1-\delta_{n,p}}}
	\le C\delta_{n,p}.
	\]
	Consequently, on
	\(\mathcal F\cap\{\delta_{n,p}\le1/2\}\), we have
	\begin{equation}\label{eq:conditional-sign-comparison}
	\sup_{x\in\mathbb R}
	\left|
	\mathbb P_{\boldsymbol{\varepsilon}}\left(
	\frac{R^{\rm s}-m}{\sqrt{2v}}\le x
	\,\middle|\,\mathbf X
	\right)-\Phi(x)
	\right|
	\le
	C\big(v^{-1/8}+\delta_{n,p}\big).
	\end{equation}
	On the event
	\(\mathcal F\cap\{\delta_{n,p}>1/2\}\), we use the trivial bound by one. Since
	\[
	\mathbb P\big(
	\mathcal F\cap\{\delta_{n,p}>1/2\}
	\big)
	\le
	2\mathbb E\delta_{n,p}
	\]
	by Markov's inequality, averaging the conditional estimate
	\eqref{eq:conditional-sign-comparison} over \(\mathbf X\), and using
	\(\mathbb P(\mathcal F^c)\le\eta_n\), gives
	\begin{equation}\label{eq:standard-sign-comparison}
	d_K\left(\frac{R^{\rm s}-m}{\sqrt{2v}},N(0,1)\right)
	\le C\big(v^{-1/8}+\mathbb{E}\delta_{n,p}+\eta_n\big).
	\end{equation}
	For Gaussian entries, \eqref{eq:Gaussian-moments} gives $\mathbb{E}\delta_{n,p}^{\rm G}=2/(n+2)$. Applying \eqref{eq:standard-sign-comparison} to that projection and using the triangle inequality gives
	\begin{equation}
	d_K(R^{\rm s},R^{\rm G})
	\le C\big(v^{-1/8}+\mathbb{E}\delta_{n,p}+n^{-1}+\eta_n\big).
	\end{equation}
	For symmetric entries this is also the bound for $d_K(R,R^{\rm G})$ since $R\stackrel d=R^{\rm s}$.
	Notice that the map $r\mapsto n/r-1$ is strictly decreasing on $(0,n]$. The distribution functions after this transformation are complements of left limits before it, and thus the Kolmogorov distance is unchanged. Extending the map at zero and then using the convention $T_n^2=0$ on $\mathcal{G}^c$ changes the distribution by at most $\eta_n$. This establishes \eqref{eq:comparison} as desired.
	
	For \eqref{eq:fixed-dimension-bound}, if $p=k$ is fixed, choose an orthonormal basis $\mathbf{U}$ of the row space on $\mathcal{F}$. Then $\mathbf{L}=\mathbf{U}\mathbf{U}^\top$ and $\max_{1\le j\le n}\|\mathbf{U}^\top \mathbf{e}_j\|^2=\max_{1\le j\le n}L_{jj}$. It follows that $\|\mathbf{U}^\top \boldsymbol{\varepsilon}\|^2=\boldsymbol{\varepsilon}^{\top}\mathbf{L}\boldsymbol{\varepsilon}=n-R^{\rm s}$.
	Equation \eqref{eq:sign-fixed} in Lemma~\ref{lem:sign} gives
	\[
	d_K(n-R^{\rm s},\chi_k^2)
	\le C_k\mathbb{E}[\mathbbm{1}_{\mathcal{F}}(\max_{1\le j\le n}L_{jj})^{1/8}]+\eta_n.
	\]
	For the Gaussian projection, we can write $n-R^{\rm G}=B_k/D_n$ with $D_n=(A_{n-k}+B_k)/n$. Since $\mathbb{E}(D_n-1)^2=2/n$, the event $|D_n-1|\le n^{-1/4}$ has probability at least $1-2n^{-1/2}$. On $|D_n-1|\le n^{-1/4}$, we have
	\[
	\frac{B_k}{1+n^{-1/4}}\le n-R^{\rm G}
	\le\frac{B_k}{1-n^{-1/4}}.
	\]
	The density $f_k$ of $\chi_k^2$ satisfies $\sup_{x>0}xf_k(x)<\infty$. Integrating in $\log x$ bounds both resulting distribution-function changes by $C_kn^{-1/4}$. Therefore, a triangle inequality and the monotone transformation complete the fixed-dimensional assertion.
	
	For $m=k$ fixed, the same argument uses an orthonormal basis of the null space instead of the row space. Lemma~\ref{lem:sign} gives
	\[
	d_K(R^{\rm s},\chi_k^2)
	\le C_k\mathbb{E}[\mathbbm{1}_{\mathcal{F}}(\max_{1\le j\le n}P_{jj})^{1/8}]+\eta_n.
	\]
	The same chi-square ratio calculation, now with numerator $A_k$, yields
	\[
	d_K(R^{\rm G},\chi_k^2)\le C_kn^{-1/4}.
	\]
	The triangle inequality and the decreasing transformation give \eqref{eq:fixed-nullity-bound}.
	
	For \eqref{eq:all-dimension-geometry}, if $\eta_n\to0$ and $\mathbb{E}\delta_{n,p}\to0$, applying \eqref{eq:comparison} gives convergence whenever $\min(p,m)\to\infty$, since $v\ge\min(p,m)/2$. Along a subsequence with $p=k$ fixed, $\sum_{j=1}^nL_{jj}^2=v\delta_{n,p}+p^2/n$ on $\mathcal{F}$, so the fixed-rank comparison applies. Along one with $m=k$ fixed, use instead $\sum_{j=1}^nP_{jj}^2=v\delta_{n,p}+m^2/n$ and \eqref{eq:fixed-nullity-bound}. Every subsequence has a further subsequence of one of these types, giving sufficiency for arbitrary dimension sequences.
	Finally, for necessity, the singular-event convention implies
	\[
	\eta_n\le\mathbb{P}(T_n^2=0)\le \varepsilon_{n,p},
	\]
	since $G_{p,m}>0$ almost surely. Thus $\eta_n\to0$. Define
	\[
	R^\sharp=\mathbbm{1}_{\mathcal{F}}R+\mathbbm{1}_{\mathcal{F}^c}m,
	\qquad Z^\sharp=(R^\sharp-m)/\sqrt{2v}.
	\]
	The event $\mathcal{F}$ is invariant under column signs. Conditional sign replacement and Lemma~\ref{lem:sign} therefore give
	\[
	\mathbb{E}|Z^\sharp|^4\le C,
	\qquad \mathbb{E}(Z^\sharp)^2=\mathbb{P}(\mathcal{F})-\mathbb{E}\delta_{n,p}.
	\]
	For the fourth moment, we use either $\mathbf{D}=\mathbf{P}-\diag \mathbf{P}$ or $\mathbf{D}=-(\mathbf{L}-\diag \mathbf{L})$ on $\mathcal{F}$. Its squared Frobenius norm is at most $\min(p,m)$, while $v\ge\min(p,m)/2$. The same bound holds for $Z^{\rm G}=(R^{\rm G}-m)/\sqrt{2v}$, whose second moment is $n/(n+2)$. The common monotone transformation and the rank correction give $d_K(Z^\sharp,Z^{\rm G})\to0$. For fixed $K$, the difference of the expectations of $x^2\wedge K^2$ is at most $2K^2$ times this distance. Both tail errors are $\mathrm{O}(K^{-2})$ by the fourth-moment bounds. Letting $n\to\infty$ first and then $K\to\infty$ gives $\mathbb{E}\delta_{n,p}\to0$, completing the proof of the proposition.
\end{proof}

\subsection{Proof of Lemma~\ref{lem:five-halves-truncation}}\label{app:prf_truncation}
\begin{proof}
In view of \eqref{eq:weak-five-halves} and $\mathbb{E}\xi=0$, for $0<s<5/2$, tail integration gives
\[
 \mathbb{E}[|\xi|^s\mathbbm{1}_{\{|\xi|>t\}}]
 =t^s\mathbb{P}(|\xi|>t)
   +s\int_t^\infty x^{s-1}\mathbb{P}(|\xi|>x)\,\mathrm dx
 =o(t^{s-5/2}),
\]
and for $s>5/2$, truncating the analogous integral at $t$ gives $\mathbb{E}[|\xi|^s\mathbbm{1}_{\{|\xi|\le t\}}]=o(t^{s-5/2})$. Thus, \eqref{lt:eq:fixed-cutoff-moments} follows by setting $t=n^{2/3}$.

Since $\mathbb{E}\widehat\xi_n^2\to1$ by dominated convergence,
$\mu_n\to0$, and $\epsilon_n\to0$, the definition of $s_n^2$ gives
$s_n\to1$. The independent smoothing makes
$\zeta_n$ non-atomic, and centering and rescaling give mean zero
and variance one. Since $s_n$ stays bounded away from zero, we get \eqref{eq:five-halves-truncated-moments}.
Moreover, for fixed constants $c,C>0$ and
large $x$, we have, uniformly over $n\ge n_0$,
\[
 \mathbb{P}(|\zeta_n|>x)\le\mathbb{P}(|\xi|>cx-C),
\]
which gives \eqref{lt:eq:fixed-uniform-tail}.

Fix $\epsilon>0$. The uniform tail estimate \eqref{lt:eq:fixed-uniform-tail} gives an $M\ge1$ such
that $\mathbb{P}(|\zeta_n|>x)\le\epsilon x^{-5/2}$ for all $x\ge M$
and all $n$. For $T\ge M$, tail integration gives
\[
\mathbb{E}[|\zeta_n|^3\mathbbm{1}_{\{|\zeta_n|\le T\}}]
\le M\mathbb{E}\zeta_n^2
+M^3\mathbb{P}(|\zeta_n|>M)
+3\int_M^T x^2\mathbb{P}(|\zeta_n|>x)\,\mathrm dx
\le M+C\epsilon T^{1/2},
\]
and
\[
\mathbb{E}[|\zeta_n|\mathbbm{1}_{\{|\zeta_n|>T\}}]
=T\mathbb{P}(|\zeta_n|>T)
+\int_T^\infty\mathbb{P}(|\zeta_n|>x)\,\mathrm dx
\le C\epsilon T^{-3/2}.
\]
Since the laws are centered, with $T=s^{-1/2}\ge M$,
\begin{align*}
\left|\mathbb{E}\{\zeta_n e^{-s\zeta_n^2}\}\right|
&=\left|\mathbb{E}\{\zeta_n (1-e^{-s\zeta_n^2})\}\right|\le\mathbb{E}\{|\zeta_n|(1-e^{-s\zeta_n^2})\}\\
&\le s\mathbb{E}[|\zeta_n|^3\mathbbm{1}_{\{|\zeta_n|\le T\}}]
+\mathbb{E}[|\zeta_n|\mathbbm{1}_{\{|\zeta_n|>T\}}]\\
&\le Ms+C\epsilon s^{3/4}.
\end{align*}
For each fixed $\epsilon>0$, the corresponding $M$ is fixed.
Dividing by $s^{3/4}$ and letting $s\rightarrow 0$ therefore gives
\[
 \limsup_{s\rightarrow 0}\sup_n
 s^{-3/4}\left|\mathbb{E}\{\zeta_n e^{-s\zeta_n^2}\}\right|
 \le C\epsilon\rightarrow 0
\]
by letting $\epsilon\rightarrow 0$. In particular, one can choose
\[
 \omega(t)=\sup_n\sup_{0<s\le t}
 s^{-3/4}\left|\mathbb{E}\{\zeta_n e^{-s\zeta_n^2}\}\right|,
 \qquad 0<t\le1.
\]
Thus \eqref{sr:eq:signed-laplace-bound} follows.
\end{proof}

\subsection{Proof of Lemma~\ref{sat:lem:truncation}}\label{app:prf:diff_truncation}
\begin{proof}
Denote $q=q_n$ and $\mu=\mu_n$ for simplicity. To compare $\mathbf{X}$ and $\widehat{\mathbf{X}}$, replacing the rows one at a time, we condition at each replacement on the $p-1$ common rows and denote $\mathbf{B}$ as this matrix.
Let $\mathbf{J}=\mathbf{P}(\mathbf{B})$, $\mathbf{t}=\mathbf{J}\mathbf{1}$, and
$d_j=J_{jj}$.  Define $t_j^2/d_j=0$ when $d_j=0$; then $t_j=0$ as
well. Next, we shall show
\begin{equation}\label{sat:eq:directional-leverage}
 \max_{1\le j\le n}\mathbb{E}\frac{t_j^2}{d_j}
 \leq C(1+n\mu^2)+Cne^{-cn}.
\end{equation}
For $d=0$ the ratio equals one. Otherwise, remove the $j$-th column $\mathbf{b}_j$ and write
\[
 \mathbf{M}=(\mathbf{B}_{-j}\mathbf{B}_{-j}^\top)^{-1},\qquad
 \mathbf{u}=\mathbf{M}\mathbf{B}_{-j}\mathbf{1}_{n-1}.
\]
On the event in \eqref{sym:eq:hybrid-smallest} for
$\mathbf{B}_{-j}$, we have $\|\mathbf{M}\|\leq C/n$ and
\[
 \|\mathbf{u}\|^2
 \leq\|\mathbf{M}\|
       \mathbf{1}_{n-1}^\top\mathbf{B}_{-j}^\top
       \mathbf{M}\mathbf{B}_{-j}\mathbf{1}_{n-1}\leq C.
\]
For the independent deleted column $\mathbf{b}_j$, the exact leave-column identities are
\[
 d_j=\frac1{1+\mathbf{b}_j^\top\mathbf{M}\mathbf{b}_j},\qquad
 t_j=d_j(1-\mathbf{b}_j^\top\mathbf{u}).
\]
Conditionally on
$\mathbf{B}_{-j}$, we get
\[
 \mathbb{E}(\mathbf{b}_j^\top\mathbf{u})^2
 \leq \|\mathbf{u}\|^2+
       \|\mathbb{E}\mathbf{b}_j\|^2\|\mathbf{u}\|^2
 \leq C(1+n\mu^2).
\]
which together with $d_j\leq1$ and the deterministic projection inequality
\[
 t_j^2=|\mathbf{e}_j^\top\mathbf{J}\mathbf{1}|^2
 \leq J_{jj}\,\mathbf{1}^\top\mathbf{J}\mathbf{1}\leq nd_j
\]
gives \eqref{sat:eq:directional-leverage}.

Next, we consider the difference of changing one coordinate of one replacing row. Condition on its $p-1$ common rows. Their null-space projection
$\mathbf{J}$ has rank at least $n-p+1$. In view of \eqref{eq:five-halves-projection-recursion}, adding a row $\mathbf{z}^\top$ decreases the directional statistic by
\[
 \mathcal Q(\mathbf{z})
 =\frac{(\mathbf{t}^\top\mathbf{z})^2}
        {\mathbf{z}^\top\mathbf{J}\mathbf{z}},
 \qquad 0\leq\mathcal Q(\mathbf{z})\leq n,
\]
with value zero if $\mathbf{z}^\top\mathbf{J}\mathbf{z}=0$.  Fix coordinate $j$ and set
$x=z_j$.  If $d=J_{jj}=0$, changing $x$ has no effect to $\mathcal Q(\mathbf{z})$. Otherwise, define
\[
 \mathbf{e}=\frac{\mathbf{J}\mathbf{e}_j}{\sqrt d},\qquad
 \mathbf{K}=\mathbf{J}-\mathbf{e}\mathbf{e}^\top,\qquad
 b=\frac{t_j}{\sqrt d}.
\]
The matrix $\mathbf{K}$ is an orthogonal projection of rank at least
$n-p$, and $\mathbf{K}\mathbf{e}_j=0$.  Let $\mathbf{z}^{(j)}$ be
$\mathbf{z}$ with its $j$-th coordinate set to zero, and define
\[
 s=\mathbf{e}^\top\mathbf{z}^{(j)},\quad
 U=\mathbf{1}^\top\mathbf{K}\mathbf{z}^{(j)},\quad
 V_0=(\mathbf{z}^{(j)})^\top\mathbf{K}\mathbf{z}^{(j)}.
\]
By orthogonal decomposition in the direction $\mathbf{e}$, we can represent
\begin{equation}\label{sat:eq:coordinate-decrement}
 \mathcal Q(x)
 =\frac{\{b(\sqrt d\,x+s)+U\}^2}
        {(\sqrt d\,x+s)^2+V_0}.
\end{equation}
For $V_0>0$, the Cauchy--Schwarz inequality gives
\begin{equation}\label{sat:eq:coordinate-saturation}
 0\leq\mathcal Q(x)\leq b^2+\frac{U^2}{V_0}.
\end{equation}
When $V_0=0$, one also has $U=0$, and the same assertion holds if
$U^2/V_0$ is assigned the value zero. It follows that
\begin{equation}\label{sat:eq:one-coordinate-difference}
 |\mathcal Q(x)-\mathcal Q(x\mathbbm{1}_{\{|x|\leq\tau_n\}})|
 \leq\mathbbm{1}_{\{|x|>\tau_n\}}
       \left(b^2+\frac{U^2}{V_0}\right).
\end{equation}
In view of $\mathbf{K}\mathbf{e}_j=0$, applying the projection small-ball estimate to the restriction of $\mathbf{K}$ to $\mathbf{e}_j^\perp$
and the $n-1$ random coordinates, as in Lemma~\ref{lem:lower}, gives
\[
 \mathbb{P}(V_0<cn\mid\mathbf{J})\leq Ce^{-cn}.
\]
On the other hand, we have
\[
 \mathbb{E}(U^2\mid\mathbf{J})
 \leq\|\mathbf{K}\mathbf{1}\|^2+
       \|\mathbb{E}\mathbf{z}^{(j)}\|^2
       \|\mathbf{K}\mathbf{1}\|^2
 \leq n+n^2\mu^2.
\]
since $\|\mathbf{K}\mathbf{1}\|^2\leq n$.
Consequently,
\begin{equation}\label{sat:eq:background-decrement}
 \mathbb{E}\left(\frac{U^2}{V_0}\,\middle|\,\mathbf{J}\right)
 \leq C(1+n\mu^2)+Cne^{-cn},
\end{equation}
where we used the deterministic bound $U^2/V_0\leq\|\mathbf{K}\mathbf{1}\|^2\leq n$.
The event $|x|>\tau_n$ is independent of $\mathbf{J}$ and all other coordinates of the row. Taking expectations in
\eqref{sat:eq:one-coordinate-difference} and using
\eqref{sat:eq:directional-leverage} and
\eqref{sat:eq:background-decrement} bounds the mean change in $R$ for one coordinate replacement by
\[
 Cq(1+n\mu^2)+Cnq e^{-cn},
\]
which gives \eqref{sat:eq:transfer} since there are $pn$ replacements, completing the proof of this lemma.
\end{proof}

\subsection{Proof of Lemma~\ref{sr:lem:row-mean}}\label{app:prf:diff_symmetry}
\begin{proof}
Write $k=k_n$, $\ell=l_n$, $d_j=J_{jj}$, and
\[
 \mathbf{B}=\mathbf{J}-\diag(\mathbf{J}),\qquad
 D=\sum_{j=1}^n d_jz_j^2,\qquad F=1+D,\qquad
 O=\mathbf{z}^\top\mathbf{B}\mathbf{z},
\]
\[
 A=(\mathbf{u}^\top\mathbf{z})^2,\qquad
 V=\mathbf{z}^\top\mathbf{J}\mathbf{z}=D+O.
\]
The same notation with a superscript $\circ$ refers to the
symmetrized row. In the stated coupling, $D^\circ=D$ and
$F^\circ=F$, since $(z_j^\circ)^2=z_j^2$.
Also $0\le d_j\le1$, $\sum_{j=1}^n d_j=r$, $\|\mathbf{B}\|\le2$,
and $\tr\mathbf{B}^2\le n$. The projection property gives
\begin{equation}\label{sr:eq:deterministic-bounds}
 A\le aV,\qquad
 V\le\left(\sum_{j=1}^n \sqrt{d_j}|z_j|\right)^2\le nD\le nF.
\end{equation}
Consequently, $0\le\theta\le a$ and $A=0$ whenever $V=0$.
The exact identity
\begin{equation}\label{sr:eq:ratio-expansion}
 \theta=\frac AF-\frac{AO}{F^2}+\frac A{F^2}
       +\frac{A(O-1)^2}{F^2V}
\end{equation}
holds with its last term set to zero when $V=0$. It also holds for
the symmetrized row.

By Lemma~\ref{lem:lower}, we have
\[
 \mathbb{P}(V<cn)+\mathbb{P}(V^\circ<cn)\le Ce^{-cn}.
\]
There are fixed $b,q>0$ such that
$\mathbb{P}(|\zeta_n|\ge b)\ge q$, uniformly in $n$.
For $0\le s\le1$, define
\begin{equation}\label{sr:eq:phi-bounds}
 \phi_j(s)=\mathbb{E}e^{-s d_jz_j^2},~
 \Phi(s)=\prod_{j=1}^n\phi_j(s)\le e^{-cns}.
\end{equation}
It follows that $e^{-1}\le e^{-sd_j}\le \phi_j(s)\le e^{-csd_j}$ since 
$1-\phi_j(s)\ge q(1-e^{-s d_jb^2})\ge csd_j$ and
$\mathbb{E}z_j^2=1$. Exponential Markov's inequality, applied at
$s=1$, gives, for a sufficiently small fixed $c_0>0$,
\[
 \mathbb{P}(F<c_0n)
 \le e^{c_0n-1}\mathbb{E}e^{-D}
 =e^{c_0n-1}\Phi(1)\le Ce^{-cn}.
\]
These exponential estimates remain valid after deleting at most
four coordinates from $D$, because the remaining diagonal sum is
at least $r-4\ge\delta n-4$. We now estimate the terms of
\eqref{sr:eq:ratio-expansion} separately.

For the first term, write
\[
 A_{\rm diag}=\sum_{j=1}^n u_j^2z_j^2,\qquad
 A_{\rm off}=\sum_{j\ne k}^n u_ju_kz_jz_k.
\]
The diagonal contribution cancels exactly:
\begin{equation}\label{sr:eq:diagonal-exact-cancellation}
 \mathbb{E}\frac{A_{\rm diag}}F
 =\mathbb{E}\frac{A_{\rm diag}^\circ}{F^\circ}.
\end{equation}
Moreover, $\mathbb{E}(A_{\rm off}^\circ/F)=0$ by symmetry.
For $0\le s\le1$, define
\[
 m_j(s)=\frac{\mathbb{E}[z_je^{-sd_jz_j^2}]}{\phi_j(s)},\qquad
 \mathbf{m}(s)=(m_1(s),\ldots,m_n(s))^\top.
\]
Independence gives the exact factorization
\[
 \mathbb{E}[A_{\rm off}e^{-sD}]
 =\Phi(s)\left\{(\mathbf{u}^\top\mathbf{m}(s))^2
                 -\sum_{j=1}^n u_j^2m_j(s)^2\right\}.
\]
By \eqref{sr:eq:phi-bounds} and
\eqref{sr:eq:signed-laplace-bound}, applied with argument $sd_j$, one gets
\begin{equation}\label{sr:eq:damped-signed-means}
 |m_j(s)|\le C\omega(s)s^{3/4}d_j^{3/4},\qquad 0\le s\le1.
\end{equation}
Since $\|\mathbf{u}\|^2=a$ and $0\le d_j\le1$,
\[
 \left|(\mathbf{u}^\top\mathbf{m}(s))^2
              -\sum_{j=1}^n u_j^2m_j(s)^2\right|
 \le2a\|\mathbf{m}(s)\|^2
 \le Ca n\omega(s)^2s^{3/2}.
\]
Using $F^{-1}=\int_0^\infty e^{-sF}\,\mathrm{d}s$, we obtain
\begin{align}
 \left|\int_0^1e^{-s}\mathbb{E}[A_{\rm off}e^{-sD}]
                    \,\mathrm{d}s\right|
 &\le Ca n\int_0^1\omega(s)^2s^{3/2}e^{-cns}\,\mathrm{d}s\notag\\
 &=\frac{Ca}{n^{3/2}}
   \int_0^n\omega(h/n)^2h^{3/2}e^{-ch}\,\mathrm{d}h
 =o(a n^{-3/2}),
 \label{sr:eq:principal-odd-term}
\end{align}
where the last integral tends to zero by dominated convergence:
$\omega(h/n)\to0$ for each fixed $h$, $\omega$ is bounded,
and $h^{3/2}e^{-ch}$ is integrable. 

For $s\ge1$, independence gives, for every $j\ne k$,
\begin{align*}
 \left|\mathbb{E}[z_jz_ke^{-sF}]\right|
 =e^{-s}\left|
      \mathbb{E}[z_je^{-sd_jz_j^2}]
      \mathbb{E}[z_ke^{-sd_kz_k^2}]
      \prod_{h\ne j,k}^n\phi_h(s)\right|\le e^{-s}\prod_{h\ne j,k}^n\phi_h(1)
 \le Ce^{-s-cn},
\end{align*}
where we used $\mathbb{E}|z_j|\le1$ and
$\phi_h(s)\le\phi_h(1)\le e^{-cd_h}$. Then the Cauchy--Schwarz inequality gives
$\sum_{j\ne k}^n |u_j||u_k|\le(\sum_{j=1}^n |u_j|)^2\le na$, which further gives
\[
 \int_1^\infty\left|\mathbb{E}[A_{\rm off}e^{-sF}]\right|\,\mathrm{d}s\le \int_1^\infty Ca n e^{-s-cn}\,\mathrm{d}s\le Ca n e^{-cn}.
\]
Together with \eqref{sr:eq:diagonal-exact-cancellation}, this yields
\[
 \left|\mathbb{E}\frac AF
          -\mathbb{E}\frac{A^\circ}{F^\circ}\right|
 =o(a n^{-3/2})+O(a n e^{-cn}).
\]

For the second term in \eqref{sr:eq:ratio-expansion}, we shall show 
\begin{equation}\label{sr:eq:AO-precise}
 \left|\mathbb{E}\frac{AO}{F^2}\right|
 \le Ca\{n^{-2}+\ell^2n^{-5/2}
                  +\ell^2n^{-3}+\ell^4n^{-4}\}+Ca e^{-cn}.
\end{equation}
For $0\le s\le1$, define
\[
 \sigma_j^2(s)=
 \frac{\mathbb{E}[(z_j-m_j(s))^2e^{-sd_jz_j^2}]}{\phi_j(s)},
 \qquad
 \kappa_j(s)=
 \frac{\mathbb{E}[(z_j-m_j(s))^3e^{-sd_jz_j^2}]}{\phi_j(s)}.
\]
The elementary centering estimate and \eqref{sr:eq:phi-bounds} give
\[
 |m_j(s)|\le Csd_j\ell,\qquad
 \sigma_j^2(s)\le C,\qquad
 |\kappa_j(s)|\le C\ell,\qquad
 \|\mathbf{m}(s)\|\le Cs\ell\sqrt n.
\]
where we used 
$\mathbb{E}z_j=0$, $1-e^{-x}\le x$, and $|m_j(s)|\le C$. Set $\boldsymbol{\chi}=(\chi_1,\ldots,\chi_n)^{\top}$ with 
$\chi_j=z_j-m_j(s)$, and note that
$\mathbb{E}[\chi_je^{-sd_jz_j^2}]=0$.
Independence gives, for each tuple of nonnegative integers
$(q_1,\ldots,q_n)$,
\[
 \frac{\mathbb{E}[\prod_{j=1}^n  \chi_j^{q_j}e^{-sD}]}{\Phi(s)}
 =\prod_{j=1}^n 
   \frac{\mathbb{E}[\chi_j^{q_j}e^{-sd_jz_j^2}]}{\phi_j(s)}
\]
whenever the displayed moments are finite. Thus any factor occurring to its first power contributes zero. Set
\[
 \lambda=\mathbf{u}^\top\mathbf{m},\qquad
 \mathbf{b}=\mathbf{B}\mathbf{m},\qquad
 c_0=\mathbf{m}^\top\mathbf{B}\mathbf{m}.
\]
Expanding
$A=(\lambda+\mathbf{u}^\top\boldsymbol{\chi})^2$ and
$O=c_0+2\mathbf{b}^\top\boldsymbol{\chi}
       +\boldsymbol{\chi}^\top\mathbf{B}\boldsymbol{\chi}$,
and using $B_{jj}=0$, gives 
\begin{align}
 \frac{\mathbb{E}[AOe^{-sD}]}{\Phi(s)}
 =c_0\left(\lambda^2+\sum_{j=1}^n u_j^2\sigma_j^2\right)
   +4\lambda\sum_{j=1}^n u_jb_j\sigma_j^2+2\sum_{j=1}^n u_j^2b_j\kappa_j
   +2\sum_{j\ne k}^n u_ju_kB_{jk}\sigma_j^2\sigma_k^2.
 \label{sr:eq:tilted-AO-identity}
\end{align}
The first two terms are bounded by
$Ca(\|\mathbf{m}\|^2+\|\mathbf{m}\|^4)$.
Moreover, the third term is bounded by
$Ca\ell\|\mathbf{B}\mathbf{m}\|_\infty
 \le Ca s\ell^2\sqrt n$ and the last term satisfies
$2|\mathbf{u}^\top\diag(\sigma_j^2)
   \mathbf{B}\diag(\sigma_j^2)\mathbf{u}|\le C\|\mathbf{u}\|^2=Ca$. Therefore, we obtain
\[
 |\mathbb{E}[AOe^{-sD}]|
 \le Ca\Phi(s)
       \{1+s\ell^2\sqrt n+s^2\ell^2n+s^4\ell^4n^2\}.
\]
Integrating against $se^{-s}$ over $[0,1]$, and using
$\int_0^1s^je^{-cns}\,\mathrm{d}s\le C_jn^{-j-1}$,
gives the polynomial terms on the right-hand side of
\eqref{sr:eq:AO-precise}.

For $s\ge1$, expand $AO$ into monomials of total degree four,
each involving at most four coordinates. Firstly, note that the Laplace representation
$F^{-2}=\int_0^\infty se^{-sF}\,\mathrm{d}s$ is legitimate
here because $|O|\le(n+1)F$ and
$|AO|/F^2\le an(n+1)$ by \eqref{sr:eq:deterministic-bounds}.
The Laplace factors from
the other coordinates give $Ce^{-cn}$ by \eqref{sr:eq:phi-bounds}. If a monomial has
exponents $q_j$ summing to four, applying H\"older's inequality gives
\[
 \prod_{j:q_j>0}^n\mathbb{E}|z_j|^{q_j}
 \le\prod_{j:q_j>0}^n k^{q_j/4}=k.
\]
The sum of the absolute coefficients is bounded by
\[
 (\sum_{j=1}^n |u_j|)^2\sum_{j,k}^n|B_{jk}|
 \le na\,n\|\mathbf{B}\|_{\mathrm{F}}
 \le Ca n^{5/2}.
\]
It follows that the integral over $[1,\infty)$ is at most
$Ca n^{5/2}k e^{-cn}\int_1^\infty se^{-s}\,\mathrm{d}s \le Cae^{-cn}$ since $k=o(n)$, which gives  \eqref{sr:eq:AO-precise}.
The symmetrized law has the same $k$ and $\ell$ and satisfies the
same estimate. 

For the last two terms of \eqref{sr:eq:ratio-expansion},
$\mathbb{E}A=a$, the lower bound for $F$, and
\eqref{sr:eq:deterministic-bounds} give
\begin{equation}\label{sr:eq:additional-regularization-term}
 \mathbb{E}\frac A{F^2}\le Ca n^{-2}+Ca n e^{-cn}.
\end{equation}
Setting $\mu_3=\mathbb{E}\zeta_n^3$, a straightforward expansion yields 
\begin{align}
 \mathbb{E}(AO^2)
 =2a\tr\mathbf{B}^2+8\mathbf{u}^\top\mathbf{B}^2\mathbf{u}
   +4(k-3)\sum_{j=1}^n u_j^2(\mathbf{B}^2)_{jj}+4\mu_3^2\sum_{j\ne k}^n u_ju_kB_{jk}^2
 \le Ca(n+k),
 \label{sr:eq:AO2-identity}
\end{align}
where the last bound follows from $\mu_3^2\le k$, $\|\mathbf{B}\|\le2$,
and
\[
 \sum_{j\ne k}^n |u_ju_k|B_{jk}^2
 \le\frac12\sum_{j\ne k}^n (u_j^2+u_k^2)B_{jk}^2
 =\sum_{j=1}^n u_j^2(\mathbf{B}^2)_{jj}\le4a.
\]
On the event $\{F\ge cn,V\ge cn\}$, using
$(O-1)^2\le2O^2+2$ gives
\[
 \mathbb{E}\left[
 \frac{A(O-1)^2}{F^2V}
   \mathbbm{1}_{\{F\ge cn,V\ge cn\}}\right]
 \le Ca(n+k)n^{-3}.
\]
On the complementary event, $O-1=V-F$ and
\eqref{sr:eq:deterministic-bounds} imply
\[
 \frac{A(O-1)^2}{F^2V}
 =\frac AV\left(\frac VF-1\right)^2\le a(n+1)^2.
\]
Thus
\begin{equation}\label{sr:eq:ratio-remainder}
 \mathbb{E}\frac{A(O-1)^2}{F^2V}
 \le Ca(n+k)n^{-3}+Ca n^2e^{-cn}.
\end{equation}
The same bound holds for the symmetrized row.

Finally, combine \eqref{sr:eq:principal-odd-term},
\eqref{sr:eq:AO-precise},
\eqref{sr:eq:additional-regularization-term}, and
\eqref{sr:eq:ratio-remainder}. The dominated-convergence integral
in \eqref{sr:eq:principal-odd-term} supplies a deterministic
$\rho_n\to0$, uniformly over the permitted projections and
directions, giving \eqref{sr:eq:mean-comparison-rate}.
Since $l_n^2\le k_n=o(n)$, each remaining term is
$o(a n^{-3/2})$, which proves \eqref{sr:eq:mean-comparison}.
All estimates hold for every permitted deterministic
$\mathbf{J},\mathbf{u}$; conditioning proves the assertion when
they are random and independent of the row, completing the proof of this lemma.
\end{proof}

\subsection{Proof of Lemma~\ref{sym:lem:increments}}\label{app:prf_uniformint}
\begin{proof}
Set $l_n=\mathbb{E}|\zeta_n|^3$.  The Cauchy--Schwarz inequality gives
$l_n^2\leq k_n$, so $l_n/\sqrt n\to0$. By Lemma \ref{lem:lower}, we have
\begin{equation}\label{sym:eq:V-smallball}
 \mathbb{P}\{V<cr\mid\mathbf{J}\}\leq Ce^{-cn}.
\end{equation}
On $V\geq cr$, $\Delta\leq CA$, while on its complement
$\Delta\leq n$. Therefore, \eqref{sym:eq:Delta-second} follows from \eqref{sym:eq:increment-moments} and
$k_n/n\to0$. 

Next, we show \eqref{sym:eq:A-UI}. Choose any sequence of the permitted
intermediate matrices and either fresh-row law at each $n$, and
write $S_n=\mathbf{u}^\top\mathbf{z}=\sum_{j=1}^n u_jz_j$. By \eqref{sym:eq:increment-moments} and
$a\le C$, the Cauchy--Schwarz inequality gives
\[
 \mathbb{E}\sum_{j=1}^n |u_j|^3
 \le\mathbb{E}\left\{a^{1/2}
                    \left(\sum_{j=1}^n  u_j^4\right)^{1/2}\right\}
 \le Cn^{-1/2}.
\]
Let $\psi_n(s)$ be the characteristic function of the fresh-row
entry law, which is centered, has variance one, and has absolute
third moment $l_n$. Taylor's formula for $e^{ix}$ gives
\[
 |\psi_n(s)-e^{-s^2/2}|\le C l_n|s|^3,
 \qquad s\in\mathbb{R}.
\] 
Conditional on
$\mathbf{J}$, by the independence,
for every fixed $t\in\mathbb{R}$, we have
\[
 \left|\mathbb{E}(e^{itS_n}\mid\mathbf{J})-e^{-t^2a/2}\right|
 =\left|\prod_{j=1}^n  \psi_n(tu_j)
        -\prod_{j=1}^n   e^{-t^2u_j^2/2}\right|
 \le C|t|^3l_n\sum_{j=1}^n |u_j|^3.
\]
In view of $l_n/\sqrt n\to0$, taking expectations yields
\begin{equation}\label{sym:eq:conditional-mixture-comparison}
 \left|\mathbb{E}e^{itS_n}-\mathbb{E}e^{-t^2a/2}\right|
 \le C|t|^3l_n n^{-1/2}\rightarrow 0.
\end{equation}

Fix an arbitrary subsequence. Since $0\le a\le C$, tightness allows us to select a subsequence, along which $a\dto a_*$ with $0\le a_*\le C$ almost surely. By \eqref{sym:eq:conditional-mixture-comparison},
$\mathbb{E}e^{itS_n}\to\mathbb{E}e^{-t^2a_*/2}$ for every $t\in\mathbb{R}$. The continuity theorem \citep[Theorem~3.3.17(ii)]{Durrett2019} therefore gives $S_n\dto\sqrt{a_*}G$, where $G$ is standard normal and independent of $a_*$. The boundedness gives $\mathbb{E}a^2\to\mathbb{E}a_*^2$, so \eqref{sym:eq:increment-moments} yields
\[
\mathbb{E}S_n^4
=3\mathbb{E}a^2+O(k_n/n)
\rightarrow 3\mathbb{E}a_*^2
=\mathbb{E}(\sqrt{a_*}G)^4.
\]
Set $W_n=|S_n|^4$ and $W=|\sqrt{a_*}G|^4$. Then $W_n\dto W$ and $\mathbb{E}W_n\to\mathbb{E}W$. Using bounded continuous truncations and
$W_n\mathbbm{1}_{\{W_n>2M\}}\le2(W_n-M)_+$, we obtain
\[
\limsup_n\mathbb{E}\bigl[W_n\mathbbm{1}_{\{W_n>2M\}}\bigr]
\le2\mathbb{E}(W-M)_+
\rightarrow 0
\qquad(M\to\infty).
\]
Thus $A^2=S_n^4$ is uniformly integrable along the selected subsequence.

To this end, if \eqref{sym:eq:A-UI} failed uniformly over the permitted choices,
there would be an $\epsilon>0$, integers $n_j\to\infty$, cutoffs
$K_j\to\infty$, and choices of intermediate matrices and fresh-row
laws such that
\[
 \mathbb{E}\bigl[A^2\mathbbm{1}_{\{A>K_j\}}\bigr]\ge\epsilon
 \qquad\text{at }n=n_j.
\]
The preceding argument provides a subsequence on which
$A^2$ is uniformly integrable, contradicting this display because
$\{A>K_j\}=\{A^2>K_j^2\}$, which gives
\eqref{sym:eq:A-UI}.

Finally, conditionally on $\mathbf{J}$, we have
\[
 \mathbb{E}\bigl[(V-r)^2\mid\mathbf{J}\bigr]
 =2r+(k_n-3)\sum_{j=1}^n  J_{jj}^2\leq Cnk_n.
\]
Thus $V/r\to1$ in probability uniformly.  On $V\geq cr$, the
variables $(r/V-1)^2$ are bounded and tend to zero in probability.
Splitting $A^2$ at a fixed truncation at $K$ and applying
\eqref{sym:eq:A-UI} gives
\[
 \sup\mathbb{E}\left[A^2(r/V-1)^2
                         \mathbbm{1}_{\{V\geq cr\}}\right]\to0.
\]
On $V<cr$, the contribution of $\Delta^2$ is at most $Cn^2e^{-cn}$,
and the contribution of $A^2$ tends to zero by uniform integrability. Then, \eqref{sym:eq:Delta-A-L2} follows, which together with the uniform integrability
of $A^2$, convergence to zero in $L^2$ of $\Delta-A$
gives \eqref{sym:eq:Delta-UI}.
\end{proof}

\section{Technical lemmas}\label{app:auxiliary}
This appendix collects some technical lemmas which are used in the proofs.
\subsection{Concentration of symmetric quadratic forms}
\begin{lemma}\label{lem:quadratic-LLN}
	Let $\mathbf{z}$ have i.i.d. centered variance-one coordinates from a fixed law, in dimension at most $Cn$. If $\mathbf{D}$ is a symmetric matrix independent of $\mathbf{z}$ and $\|\mathbf{D}\|\le K$, then
	\begin{equation}\label{eq:quadratic-LLN}
	\mathbb{E}\left|\frac{\mathbf{z}^\top \mathbf{D}\mathbf{z}-\tr \mathbf{D}}{n}\right|\rightarrow0
	\end{equation}
	uniformly over such matrices $\mathbf{D}=\mathbf{D}_n$ with dimension at most $Cn$. For centered variance-one i.i.d. entries forming a triangular array with fourth moment $k_n$, one has
	\begin{equation}\label{eq:quadratic-L2}
	\mathbb{E}|\mathbf{z}^\top \mathbf{D}\mathbf{z}-\tr \mathbf{D}|^2\le C_K n(k_n+1).
	\end{equation}
	In particular, \eqref{eq:quadratic-LLN} holds if $k_n=o(n)$.
\end{lemma}
\begin{proof}
	Let $\mathbf{z}\in\mathbb{R}^d$, where $d\le Cn$. We may first condition on $\mathbf{D}$, which is independent of $\mathbf{z}$. All constants below are uniform over $\|\mathbf{D}\|\le K$. For fixed $M>0$, define
	\[
	z_j^{(M)}=z_j\mathbbm{1}_{\{|z_j|\le M\}}
	-\mathbb{E}\big[z_j\mathbbm{1}_{\{|z_j|\le M\}}\big],
	\qquad 1\le j\le d,
	\]
	and let $\vartheta_M=\mathbb{E}(z_j-z_j^{(M)})^2$. Since $z_j-z_j^{(M)}$ is the centered omitted tail,
	\[
	\vartheta_M\le\mathbb{E}\big[z_j^2\mathbbm{1}_{\{|z_j|>M\}}\big]\to0.
	\]
	Setting $v_M=\mathbb{E}(z_j^{(M)})^2\le1$, the Cauchy--Schwarz inequality gives
	\begin{align*}
	\frac1n\mathbb{E}\left|
	\mathbf{z}^\top\mathbf{D}\mathbf{z}
	-(\mathbf{z}^{(M)})^\top\mathbf{D}\mathbf{z}^{(M)}\right|
	\le\frac Kn
	\big(\mathbb{E}\|\mathbf{z}-\mathbf{z}^{(M)}\|^2\big)^{1/2}
	\left(\big(\mathbb{E}\|\mathbf{z}\|^2\big)^{1/2}
	+\big(\mathbb{E}\|\mathbf{z}^{(M)}\|^2\big)^{1/2}\right)
	\le C_K\sqrt{\vartheta_M}.
	\end{align*}
	Also $|1-v_M|\le\vartheta_M+2\sqrt{v_M\vartheta_M}\to0$. Conditional on $\mathbf{D}$, the truncated quadratic form has mean $v_M\tr\mathbf{D}$. Thus the normalized error caused by changing this mean to $\tr\mathbf{D}$ is at most
	\[
	\frac1n|1-v_M|\,|\tr\mathbf{D}|\le C_K|1-v_M|.
	\]
	Let $\mu_{4,M}=\mathbb{E}(z_j^{(M)})^4$. Separating diagonal terms and unordered off-diagonal pairs gives
	\[
	\Var\big((\mathbf{z}^{(M)})^\top\mathbf{D}\mathbf{z}^{(M)}\mid\mathbf{D}\big)
	=2v_M^2\tr\mathbf{D}^2
	+(\mu_{4,M}-3v_M^2)\sum_{j=1}^dD_{jj}^2
	\le C_{K,M}n.
	\]
	Applying the Cauchy--Schwarz inequality to this conditional variance gives
	\[
	\frac1n\mathbb{E}\left|
	(\mathbf{z}^{(M)})^\top\mathbf{D}\mathbf{z}^{(M)}-v_M\tr\mathbf{D}
	\right|\le C_{K,M}/\sqrt n.
	\]
	The three estimates above, with $n\to\infty$ first and $M\to\infty$ afterwards, give \eqref{eq:quadratic-LLN}. Since all bounds are uniform in $\mathbf{D}$, integrating over an independent random $\mathbf{D}$ yields the stated result. For centered variance-one triangular entries, the same conditional variance identity with fourth moment $k_n$ directly gives \eqref{eq:quadratic-L2}.
\end{proof}

\subsection{A lower bound for positive quadratic forms}

\begin{lemma}\label{lem:positive-form}
	Let $\mathbf{z}\in\mathbb{R}^n$ have independent coordinates from a fixed nondegenerate real law. There are $a,c,C>0$, depending only on that law, such that for every nonzero positive semidefinite matrix $\mathbf{M}$,
	\begin{equation}
	\mathbb{P}(\mathbf{z}^\top \mathbf{M}\mathbf{z}<a\tr \mathbf{M})
	\le C\exp\big(-c\tr \mathbf{M}/\|\mathbf{M}\|\big).
	\end{equation}
\end{lemma}

\begin{proof}
	Since the entry law is nondegenerate, we can choose a bounded interval on which the conditional distribution has positive variance. Hence each coordinate can be represented as
	\[
	z_j=\eta_jU_j+(1-\eta_j)V_j,
	\]
	where \(\eta_j\) are i.i.d. Bernoulli variables with parameter \(\varrho>0\), the \(U_j\) are i.i.d. bounded random variables with mean \(\mu\) and variance \(v_0>0\), and the three families are independent. If \(\varrho=1\), the variables \(V_j\) may be chosen arbitrarily.
	
	Set $T=\tr \mathbf{M}$, $L_0=\|\mathbf{M}\|$, and $T_\eta=\sum_{j=1}^n\eta_jM_{jj}$. For $s>0$, independence and $1-e^{-x}\ge x-x^2/2$ for $x\ge 0$ give
	\[
	\mathbb{E} e^{-sT_\eta}\le \exp\big(-\varrho\sum_{j=1}^{n}\big(1-e^{-sM_{jj}}\big)\big)
	\le\exp\big(-\varrho sT+\tfrac12\varrho s^2\sum_{j=1}^nM_{jj}^2\big).
	\]
	Since $\sum_{j=1}^nM_{jj}^2\le L_0T$, exponential Markov's inequality with $s=1/(2L_0)$ yields
	\begin{equation}\label{eq:retained-trace-new}
	\mathbb{P}(T_\eta<\varrho T/2)\le e^{-\varrho T/(8L_0)}.
	\end{equation}
	Let $\boldsymbol{\eta}=(\eta_1,\ldots,\eta_n)^\top$. Conditional on $(\eta_j,V_j)_{j=1}^n$, define
	\[
	\mathbf M^{1/2}\mathbf z=\mathbf B\mathbf w+\mathbf u,
	\qquad
	\mathbf B=\mathbf M^{1/2}\diag(\boldsymbol\eta),
	\qquad
	w_j=U_j-\mu.
	\]
	Then \(\mathbf u\) is fixed and the \(w_j\) are independent, centered, and bounded. Moreover,
	\[
	\|\mathbf B\|^2\le L_0,\qquad
	\tr(\mathbf B^\top\mathbf B)=T_\eta,\qquad
	\tr\big((\mathbf B^\top\mathbf B)^2\big)\le L_0T_\eta.
	\]
	The Hanson--Wright inequality \citep[Theorem~1.1]{RudelsonVershynin2013} therefore gives
	\begin{equation}\label{eq:Bw-lower}
	\mathbb P\left(
	\|\mathbf B\mathbf w\|^2<\frac{v_0T_\eta}{2}
	\,\middle|\,
	(\eta_j,V_j)_{j=1}^n
	\right)
	\le
	2e^{-cT_\eta/L_0}.
	\end{equation}
	
	It remains to control a possible cancellation between \(\mathbf B\mathbf w\) and \(\mathbf u\). On the event
	\[
	\|\mathbf B\mathbf w\|^2\ge v_0T_\eta/2,
	\qquad
	\|\mathbf B\mathbf w+\mathbf u\|^2<v_0T_\eta/4,
	\]
	we have
	\[
	2\mathbf u^\top\mathbf B\mathbf w
	<
	-\|\mathbf u\|^2-\frac{v_0T_\eta}{4}.
	\]
	Since the \(w_j\) are independent, centered, and bounded, Hoeffding's inequality gives, conditionally on \((\eta_j,V_j)_{j=1}^n\),
	\[
	\mathbb P\left(
	2\mathbf u^\top\mathbf B\mathbf w
	<
	-\|\mathbf u\|^2-\frac{v_0T_\eta}{4}
	\right)
	\le
	e^{-cT_\eta/L_0}.
	\]
	Indeed, \(\|\mathbf B^\top\mathbf u\|^2\le L_0\|\mathbf u\|^2\), and
	\[
	\frac{\big(\|\mathbf u\|^2+v_0T_\eta/4\big)^2}
	{\|\mathbf u\|^2}
	\ge v_0T_\eta
	\]
	when \(\mathbf u\ne0\); if \(\mathbf u=0\), the cancellation event is empty.
	Combining this estimate with \eqref{eq:Bw-lower}, on \(T_\eta\ge\varrho T/2\) we obtain
	\[
	\mathbb P\left(
	\mathbf z^\top\mathbf M\mathbf z
	<
	\frac{v_0\varrho}{8}T
	\,\middle|\,
	(\eta_j,V_j)_{j=1}^n
	\right)
	\le
	3e^{-cT/L_0}.
	\]
	Together with \eqref{eq:retained-trace-new}, this yields
	\[
	\mathbb P\big(
	\mathbf z^\top\mathbf M\mathbf z<a\tr\mathbf M
	\big)
	\le
	C\exp\left(
	-c\frac{\tr\mathbf M}{\|\mathbf M\|}
	\right),
	\]
	with \(a=v_0\varrho/8\), as desired.
\end{proof}

\subsection{Projection quadratic form and the smallest eigenvalue}

The estimates below are direct consequences of Theorem~1 and Corollary~6 in \citet{Tikhomirov2016} and Theorem 1 in \cite{DabagiaFernandez2025}; we omit their proofs.
\begin{lemma}\label{lem:lower}
	Suppose a family of real laws satisfies, for some fixed $a,b>0$,
	\begin{equation}\label{eq:anticoncentration}
	\sup_{u\in\mathbb{R}}\mathbb{P}(|\zeta-u|\le a)\le1-b.
	\end{equation}
	There are constants $c,C>0$, uniform over this family, with the following properties. If $\mathbf{z}$ has independent coordinates from the family and $\mathbf{J}$ is an independent orthogonal projection of rank $r$, then
	\begin{equation}
	\mathbb{P}(\mathbf{z}^\top \mathbf{J}\mathbf{z}<c r\mid \mathbf{J})\le C e^{-cr}.
	\end{equation}
	Fix $\gamma_0<1$ and suppose $d/N\le\gamma_0$. If $\mathbf{A}$ is an $N\times d$ matrix with i.i.d. entries
	satisfying the same concentration bound and $\mathbf{D}$ is any
	deterministic $N\times d$ matrix, then
	\begin{equation}
	\mathbb{P}\{\lambda_{\min}[(\mathbf{A}+\mathbf{D})^{\top}(\mathbf{A}+\mathbf{D})]<c N\}
	\le Ce^{-cN}.
	\label{lt:eq:external-i.i.d.-inverse}
	\end{equation}
	If the entries of $\mathbf{B}\in \mathbb{R}^{i\times N}$ are instead independent, centered,
	variance one and uniformly anti-concentrated as \eqref{eq:anticoncentration}, and $i/N\le\gamma_0<1$, then, for all sufficiently large $N$,
	\begin{equation}\label{sym:eq:hybrid-smallest}
	\mathbb{P}\big(\lambda_{\min}(\mathbf{B}\mathbf{B}^\top)<cN\big)\le e^{-cN}.
	\end{equation}
\end{lemma}

\subsection{Fourth moments and variance of the projected constant vector}
\label{app:projection-moments}

The following lemma gives uniform estimates for the row sums of the
null-space projection. It applies both to the intermediate matrices in the row replacement (i.e., Section \ref{sec:lindeberg}) and to the conditional truncations that retain a small nonzero mean (i.e., Appendix \ref{lp:subsec:bias} for Example \ref{thm:regular-boundary}).

\begin{lemma}\label{lem:projection-moments}
Fix $\gamma_0<1$ and $p\le\gamma_0n$, and put $m=n-p$. For
\[
 \mathcal{I}_n=\{(d,N):0\le d\le p,\quad m+d\le N\le n\},
\]
let $\mathbf{B}$ be a $d\times N$ matrix from either of the following
two classes.
\begin{enumerate}
\item Its entries are independent, and the coordinates in each row
have a non-atomic centered variance-one law. All these laws satisfy \eqref{eq:anticoncentration}
with the same constants and have fourth moments bounded by
$k_n=o(n)$. 
\item Its entries are i.i.d. with a non-atomic law $\zeta_n$ satisfying
\eqref{eq:anticoncentration} with fixed constants. Denote
\[
 \mu_n=\mathbb{E}\zeta_n,\qquad
 \omega_n^2=\Var(\zeta_n),\qquad
 k_n=\mathbb{E}(\zeta_n-\mu_n)^4,
\]
and assume, for fixed $c_0,C_0>0$, that
\begin{equation}\label{eq:projection-moment-assumptions}
 n\mu_n^2\le C_0,\qquad c_0\le\omega_n^2\le C_0,
 \qquad k_n=o(n).
\end{equation}
\end{enumerate}
Setting $\mathbf{t}=\mathbf{P}(\mathbf{B})\mathbf{1}_N$, we have
\begin{equation}\label{sym:eq:t-fourth}
 \sup_{(d,N)\in\mathcal{I}_n}
 \sup_{\text{row-law choices}}\max_{1\le j\le N}
 \mathbb{E}|t_j|^4\le C,
\end{equation}
and, for every such matrix,
\begin{equation}\label{lp:eq:noncentered-variance}
 \Var\{\mathbf{1}_N^\top\mathbf{P}(\mathbf{B})\mathbf{1}_N\}
 \le Cd.
\end{equation}
\end{lemma}

\begin{proof}
Let $T_{d,N}=\sup_{\text{row-law choices}}\max_{1\le j\le N}
\mathbb{E}|t_j|^4$. Notice that $\mathbf{P}(\mathbf{B})=\mathbf{I}_N$ for $d=0$, which gives $t_j=1$. All
dimensions in $\mathcal{I}_n$ satisfy $N\ge m\asymp n$. Moreover,
\[
 \frac{d}{N}\le\frac{p}{n}\le\gamma_0,
 \qquad
 \frac{d}{N-1}\le\frac{p}{n-1}\le\gamma_1<1
\]
for some fixed $\gamma_1$ and all sufficiently large $n$. Therefore
Lemma~\ref{lem:lower} applies to these matrices and their one-column
deletions. In the first class we use the centered independent-entry
estimate \eqref{sym:eq:hybrid-smallest}; in the second class we use
the i.i.d. estimate \eqref{lt:eq:external-i.i.d.-inverse}, applied
to the transpose of the matrix. In both cases, we have
\begin{equation}\label{eq:projection-deleted-inverse}
 \mathbb{P}\{\lambda_{\min}(\mathbf{B}_{-j}
                    \mathbf{B}_{-j}^\top)<cN\}
 \le Ce^{-cN}.
\end{equation}
For $d\ge1$, remove the $j$-th column $\mathbf{b}_j$ and set
\[
 \mathbf{M}=(\mathbf{B}_{-j}\mathbf{B}_{-j}^\top)^{-1},\qquad
 \mathbf{v}=\mathbf{M}\mathbf{B}_{-j}\mathbf{1}_{N-1}.
\]
On the event
$E=\{\lambda_{\min}(\mathbf{B}_{-j}\mathbf{B}_{-j}^\top)\ge cN\}$,
we have $\|\mathbf{M}\|\le C/N$ and
\[
 \|\mathbf{v}\|^2
 \le\|\mathbf{M}\|\,
 \mathbf{1}_{N-1}^\top\mathbf{B}_{-j}^\top
       \mathbf{M}\mathbf{B}_{-j}\mathbf{1}_{N-1}
 \le C.
\]
For the independent deleted column $\mathbf{b}_j$,
the Sherman--Morrison formula gives
\begin{equation}\label{eq:projection-leave-column}
 t_j=\frac{1-\mathbf{b}_j^\top\mathbf{v}}
             {1+\mathbf{b}_j^\top\mathbf{M}\mathbf{b}_j}.
\end{equation}

For the rest of the proof, set $\mu_n=0$ and $\omega_n^2=1$ in
the first class, which also satisfies \eqref{eq:projection-moment-assumptions}. The coordinates of $\mathbf{b}_j$ may then have
different laws, but they have a common mean and variance and their
centered fourth moments are bounded by $k_n$. For a deterministic
vector $\mathbf{v}$, independence and centering give
\begin{align*}
 \mathbb{E}\{(\mathbf{b}_j-\mu_n\mathbf{1}_d)^\top
                                      \mathbf{v}\}^4
 =3\omega_n^4\|\mathbf{v}\|^4
   +\sum_{i=1}^d\big\{
      \mathbb{E}(b_{ij}-\mu_n)^4-3\omega_n^4\big\}v_i^4\le C\|\mathbf{v}\|^4+Ck_n\sum_{i=1}^d v_i^4.
\end{align*}
Consequently, in view of \eqref{eq:projection-moment-assumptions} and $|x+y|^4\le8|x|^4+8|y|^4$,
\begin{align}\label{eq:fourthmomentbv}
 \mathbb{E}(\mathbf{b}_j^\top\mathbf{v})^4
 \le C\|\mathbf{v}\|^4+Ck_n\sum_{i=1}^d v_i^4
       +8\mu_n^4(\mathbf{1}_d^\top\mathbf{v})^4\le C\|\mathbf{v}\|^4+Ck_n\sum_{i=1}^d v_i^4,
\end{align}
since $\mu_n^4d^2\le(n\mu_n^2)^2\le C$. Using $|t_j|^4\le N^2$ and
\eqref{eq:projection-deleted-inverse}, we obtain 
\begin{equation}\label{eq:projection-fourth-recursion-first}
 T_{d,N}\le C+Ck_n\sum_{i=1}^d
             \mathbb{E}[v_i^4\mathbbm{1}_E]+CN^2e^{-cN}.
\end{equation}

Remove the $i$-th row $\mathbf{z}^\top$ from $\mathbf{B}_{-j}$, let
$\mathbf{J}_0$ project onto the null space of the remaining rows. The Schur complement
identity gives
\[
 v_i=\frac{\mathbf{z}^\top\mathbf{J}_0\mathbf{1}_{N-1}}
           {\mathbf{z}^\top\mathbf{J}_0\mathbf{z}}.
\]
On $E$, we have
$\mathbf{z}^\top\mathbf{J}_0\mathbf{z}=1/M_{ii}\ge cN$.
Set $\mathbf{t}^-=\mathbf{J}_0\mathbf{1}_{N-1}$. The row
$\mathbf{z}$ is independent of $\mathbf{J}_0$, and the projection
property gives
\[
 \|\mathbf{t}^-\|^2
 =\mathbf{1}_{N-1}^\top\mathbf{t}^-\le N-1.
\]
A similar argument for \eqref{eq:fourthmomentbv}, conditioning on $\mathbf{J}_0$, yields
\begin{align*}
 \mathbb{E}(\mathbf{z}^\top\mathbf{t}^-)^4
 \le CN^2+Ck_n\sum_{j=1}^{N-1}\mathbb{E}|t_j^-|^4
       +8\mu_n^4(N-1)^4\le CN^2+Ck_n\sum_{j=1}^{N-1}\mathbb{E}|t_j^-|^4.
\end{align*}
The pair $(d-1,N-1)$ is still admissible and
the remaining rows belong to the same permitted class. Thus, we have
\[
 \mathbb{E}[v_i^4\mathbbm{1}_E]
 \le\frac C{N^2}+\frac{Ck_n}{N^3}T_{d-1,N-1}.
\]
Substitution into \eqref{eq:projection-fourth-recursion-first}
gives
\begin{equation}\label{eq:projection-fourth-recursion}
 T_{d,N}\le C+C\frac{k_n}{N}
       +C\left(\frac{k_n}{N}\right)^2T_{d-1,N-1}
       +CN^2e^{-cN}.
\end{equation}
Since $N\ge m\asymp n$ and $k_n=o(n)$, the coefficient of
$T_{d-1,N-1}$ is at most $1/2$ for all sufficiently large $n$,
uniformly over the permitted dimensions and row laws. Iterating to
$T_{0,N-d}=1$ gives \eqref{sym:eq:t-fourth} as desired.

For the variance estimation \eqref{lp:eq:noncentered-variance}, remove the $i$-th row from $\mathbf{B}$ and let
$\mathbf{J}$ be the null-space projection determined by its other
$d-1$ rows. By \eqref{eq:five-halves-projection-recursion},
adding an independent row $\mathbf{z}^\top$ decreases
the directional statistic $\mathbf{1}_N^\top\mathbf{P}(\mathbf{B})\mathbf{1}_N$ by
\[
 D(\mathbf{J},\mathbf{z})
 =\frac{(\mathbf{z}^\top\mathbf{J}\mathbf{1}_N)^2}
        {\mathbf{z}^\top\mathbf{J}\mathbf{z}},\qquad
 0\le D(\mathbf{J},\mathbf{z})\le N,
\]
with value zero when $\mathbf{z}^\top\mathbf{J}\mathbf{z}=0$. The rank of
$\mathbf{J}$ is at least $N-d+1\ge m+1\asymp n$. Thus
Lemma~\ref{lem:lower} gives
\[
 \mathbb{P}\{\mathbf{z}^\top\mathbf{J}\mathbf{z}<cN
                   \mid\mathbf{J}\}\le Ce^{-cN}.
\]
With a similar argument for \eqref{eq:fourthmomentbv} and the bound for $T_{d-1,N}$ above, we obtain
\begin{align*}
 \mathbb{E}D(\mathbf{J},\mathbf{z})^2
 \le\frac C{N^2}\big(N^2+k_n\sum_{j=1}^N
 \mathbb{E}|\{\mathbf{J}\mathbf{1}_N\}_j|^4\big)
       +CN^2e^{-cN}
 \le C+C\frac{k_nT_{d-1,N}}{N}+CN^2e^{-cN}\le C.
\end{align*}
Finally, let $R^{(i)}$ be obtained from
$R=\mathbf{1}_N^\top\mathbf{P}(\mathbf{B})\mathbf{1}_N$
by replacing row $i$ with an independent copy having the same law.
Conditioning on the other rows, the difference $R-R^{(i)}$ is the
difference of two such decrements. Therefore
$\mathbb{E}(R-R^{(i)})^2\le C$, uniformly in $i$. Applying the Efron--Stein inequality
to the $d$ independent rows gives
\[
 \Var(R)\le\frac12\sum_{i=1}^d
                   \mathbb{E}(R-R^{(i)})^2\le Cd,
\]
which implies \eqref{lp:eq:noncentered-variance} and completes the proof of this lemma.
\end{proof}

\section{Technical details of Examples~\ref{thm:regular-boundary} and \ref{rem:balanced-tails}}\label{sec:verification}

\subsection{Verification for Example \ref{thm:regular-boundary}}
\label{lp:subsec:bias}
Throughout this subsection, suppose that the law of $\xi$ satisfies \eqref{eq:Pareto} with $2<\alpha\le5/2$. 
It is bounded below and satisfies
\begin{equation}\label{lp:eq:pareto-tail}
 \mathbb{P}(\xi>x)\sim c_\xi x^{-\alpha},\qquad
 c_\xi=\sigma_\alpha^{-\alpha}>0.
\end{equation}
Choose $\beta$ such that
\begin{equation}\label{lp:eq:beta-window}
 \frac{\alpha-1}{\alpha}<\beta<\frac1{4-\alpha},
\end{equation}
and let the truncation level $\tau_n=n^\beta$.
It follows that $\beta>1/2$. Let $\zeta_n$ have the conditional law
\begin{equation}\label{lp:eq:conditional-law}
 \mathcal{L}(\zeta_n)
 =\mathcal{L}(\xi\mid |\xi|\le\tau_n).
\end{equation}
Denote $q_n=\mathbb{P}(|\xi|>\tau_n)$ and
$\mu_n=\mathbb{E}\zeta_n$, $\omega_n^2=\Var(\zeta_n)$ and $k_n=\mathbb{E}(\zeta_n-\mu_n)^4$, as in case~(ii) of Lemma~\ref{lem:projection-moments}. Since the negative part of $\xi$ is bounded, all removed entries are positive for large $n$. Similar to Lemma \ref{lem:five-halves-truncation}, tail
integration gives
\begin{equation}\label{lp:eq:conditional-moments}
 \begin{gathered}
 q_n\sim c_\xi\tau_n^{-\alpha},\qquad
 \mu_n=-\frac{\mathbb{E}[\xi\mathbbm{1}_{\{\xi>\tau_n\}}]}{1-q_n}
 \sim-\frac{\alpha c_\xi}{\alpha-1}\tau_n^{1-\alpha},\\
 \operatorname{Var}(\zeta_n)\rightarrow 1,\qquad
 \mathbb{E}(\zeta_n-\mu_n)^4=O(\tau_n^{4-\alpha})=o(n),
 \qquad n\mu_n^2\rightarrow 0.
 \end{gathered}
\end{equation}
where the last statement follows from $\beta>1/2$ and $\alpha>2$. It is straightforward that the laws in
\eqref{lp:eq:conditional-law} converge in total variation to the fixed law. They are non-atomic and obey a uniform small-ball condition \eqref{eq:anticoncentration} in Lemma \ref{lem:lower}.
With the notation above, the following lemma shows that there exists a directional bias of $R(\mathbf{Z})$, where $\mathbf{Z}$ is a $d\times n$ i.i.d. matrix with the entry law
\eqref{lp:eq:conditional-law}.

\begin{lemma}
\label{lp:lem:conditional-bias}
Let $\mathbf{Z}$ be a $d\times n$ i.i.d. matrix with the entry law
\eqref{lp:eq:conditional-law}, where
$\gamma_-n\le d\le\gamma_+n$ for fixed
$0<\gamma_-<\gamma_+<1$. Then, there exists a constant $c>0$ such that, uniformly over these dimensions,
\begin{equation}\label{lp:eq:conditional-bias}
 B_{n,d}:=n-d-\mathbb{E}R(\mathbf{Z})
 \ge c n^{3-\alpha}.
\end{equation}
\end{lemma}

\begin{proof}
Denote the columns of $\mathbf{Z}$ as
$\mathbf{z}_1,\ldots,\mathbf{z}_n$, and set
\[
 \mathbf{A}=\sum_{j=3}^n\mathbf{z}_j\mathbf{z}_j^\top,
 \qquad \mathbf{M}=\mathbf{A}^{-1}.
\]
The matrices have full row rank almost surely by Lemma \ref{lem:rank}. Let $\mathbf{L}=\mathbf{Z}^{\top}(\mathbf{Z}\mathbf{Z}^{\top})^{-1}\mathbf{Z}$. Conditional on
$\mathbf{M}$, let $\mathbf{x}=\mathbf{z}_1$, $\mathbf{y}=\mathbf{z}_2$,
$q_x=\mathbf{x}^\top\mathbf{M}\mathbf{x}$, $q_y=\mathbf{y}^\top\mathbf{M}\mathbf{y}$ and $q_{xy}=\mathbf{x}^\top\mathbf{M}\mathbf{y}$. The Cauchy--Schwarz inequality gives $q_{xy}^2\le q_xq_y$, which further implies $|L_{12}|\le1/2$. 
The Woodbury formula gives
\[
 L_{12}=\frac{q_{xy}}{(1+q_x)(1+q_y)-q_{xy}^2}
 =\sum_{l\ge0}\frac{q_{xy}^{2l+1}}{(1+q_x)^{l+1}(1+q_y)^{l+1}}.
\]
For
$h_l(\mathbf{x})=(\mathbf{M}^{1/2}\mathbf{x})^{\otimes(2l+1)}
/(1+\mathbf{x}^\top\mathbf{M}\mathbf{x})^{l+1}$,
the conditional expectation of the $l$-th summand is given by
$\|\mathbb{E}_{\mathbf{x}}h_l(\mathbf{x})\|^2$. In particular, we have
\begin{equation}\label{lp:eq:positive-identity}
 \mathbb{E}(L_{12}\mid\mathbf{M})
 \ge\mathbf{v}^\top\mathbf{M}\mathbf{v}\ge0,
 \qquad
 \mathbf{v}=:\mathbf{v}(\mathbf{M})
 =\mathbb{E}_{\mathbf{x}}\frac{\mathbf{x}}{1+\mathbf{x}^\top\mathbf{M}\mathbf{x}}.
\end{equation}
By the column exchangeability and $\operatorname{tr}\mathbf{L}=d$, we get
\begin{equation}\label{lp:eq:bias-trace}
 B_{n,d}=n(n-1)\mathbb{E} L_{12}.
\end{equation}
Set $\ell_n=n^{(3-\alpha)/2}$. For any fixed $C_1>0$, uniformly over
$0<t\le C_1/n$, we have
\begin{equation}\label{lp:eq:conditional-H}
 H_n(t):=\mathbb{E}\frac{\zeta_n}{1+t\zeta_n^2}
 \le-c\ell_n t.
\end{equation}
To see this, recalling \eqref{lp:eq:conditional-moments}, let $M_0$ bound the negative part of $\xi$ and notice the identity
\[
 H_n(t)=\mu_n-t\mathbb{E}\frac{\zeta_n^3}{1+t\zeta_n^2},
 \qquad \mu_n\le0.
\]
Therefore
\begin{align*}
 \mathbb{E}\frac{\zeta_n^3}{1+t\zeta_n^2}\ge\frac{n^{3/2}}{8(1+C_1)}
       \mathbb{P}\{\sqrt n/2\le\zeta_n\le\sqrt n\}-M_0^3\ge c\ell_n
\end{align*}
for large $n$, since $\alpha<3$ and $\tau_n/\sqrt n\to\infty$, which further gives \eqref{lp:eq:conditional-H}.

Now fix a positive definite $\mathbf{M}$ with $\|\mathbf{M}\|\le C_1/n$.
For coordinate $i$, write
\[
 1+\mathbf{x}^\top\mathbf{M}\mathbf{x}
 =D_i+a_ix_i^2+2r_ix_i,
 \quad a_i=M_{ii},\quad
 D_i=1+\mathbf{x}_{-i}^\top\mathbf{M}_{-i,-i}\mathbf{x}_{-i},
 \quad r_i=\sum_{j\ne i}M_{ij}x_j.
\]
Since $x_i$ and $r_i$
are independent, $\mathbb{E}x_i^2\le C$, and
$\sqrt n|\mu_n|\le C$, we have
\begin{align*}
 \left|v_i(\mathbf{M})-\mathbb{E}\frac{x_i}{D_i+a_ix_i^2}\right|
 \le 2\mathbb{E}|r_i|\,\mathbb{E}x_i^2\le C\left\{\left(\sum_{j\ne i}M_{ij}^2\right)^{1/2}
       +|\mu_n|\left|\sum_{j\ne i}M_{ij}\right|\right\}
 \le C/n
\end{align*}
and
$\mathbb{E}D_i=1+\omega_n^2\operatorname{tr}\mathbf{M}_{-i,-i}
+\mu_n^2\mathbf{1}^\top\mathbf{M}_{-i,-i}\mathbf{1}\le C$.
Using \eqref{lp:eq:conditional-H} and Jensen's inequality gives
\begin{equation}\label{lp:eq:conditional-v}
 v_i(\mathbf{M})
 \le-c\ell_n M_{ii}\mathbb{E}D_i^{-2}+C/n
 \le-c\ell_n M_{ii}+C/n.
\end{equation}
Next, choose sufficiently large fixed constants $C_1,C_2$ and consider the event
\[
 \mathcal{H}_n=\left\{
 \|\mathbf{M}\|\le C_1/n,\quad
 \#\{i:A_{ii}\le2n\}\ge d/2,\quad
 \mathbf{1}_d^\top\mathbf{A}\mathbf{1}_d\le C_2dn
 \right\}.
\]
By Lemma \ref{lem:lower}, we get $\mathbb{P}(\|\mathbf{M}\|\le C_1/n)\ge 1-Ce^{-cn}$. 
By \eqref{lp:eq:conditional-moments},
$\mathbb{E}A_{ii}=(n-2)\mathbb{E}\zeta_n^2\sim n$ and
$\operatorname{Var}(A_{ii})\le Cn(k_n+1)$, which together with Chebyshev's inequality gives 
\[
\mathbb{P}(\#\{i:A_{ii}\le2n\}< d/2)=O(k_n/n)=o(1).
\]
Moreover, notice that \[
\mathbb{E}(\mathbf{1}_d^\top\mathbf{A}\mathbf{1}_d)=(n-2)(d\omega_n^2+d^2\mu_n^2)\le Cdn,
\]
which implies $\mathbb{P}(\mathbf{1}_d^\top\mathbf{A}\mathbf{1}_d\ge C_2dn)\le 1/4$ by setting $C_2$ large enough. Thus $\mathbb{P}(\mathcal{H}_n)\ge1/2$ for
large $n$. On $\mathcal{H}_n$, we have $M_{ii}\ge A_{ii}^{-1}$ and
$\sum_iM_{ii}\ge d/(4n)$. Summing
\eqref{lp:eq:conditional-v}, and using $\ell_n\to\infty$, gives
$-\mathbf{1}_d^\top\mathbf{v}\ge c d\ell_n/n$. Thus, applying the Cauchy--Schwarz inequality yields
\[
 \mathbf{v}^\top\mathbf{M}\mathbf{v}
 \ge\frac{(\mathbf{1}_d^\top\mathbf{v})^2}
          {\mathbf{1}_d^\top\mathbf{A}\mathbf{1}_d}
 \ge c\frac{d\ell_n^2}{n^3}
\]
for some constant $c>0$. 
In view of 
\eqref{lp:eq:positive-identity} and \eqref{lp:eq:bias-trace}, we get $B_{n,d}\ge c(d/n)\ell_n^2\ge c n^{3-\alpha}$ as desired.
\end{proof}

Now, we proceed to show \eqref{eq:regular-boundary-separation} in Example \ref{thm:regular-boundary}. Recalling \eqref{lp:eq:pareto-tail} and \eqref{lp:eq:beta-window}. We call a row of $\mathbf{X}$ good if
all its $n$ entries have absolute value at most $\tau_n$; a bad row if not. Let $K_n$
be the number of bad rows and let $\mathbf{X}_{\rm good}$ contain all
good rows, in their original order. Recalling \eqref{lp:eq:conditional-moments}, a union bound gives
\begin{equation}\label{lp:eq:deleted-row-count}
\mathbb{E}K_n\le pnq_n
=O(n^{2-\alpha\beta})=o(n^{3-\alpha})
\end{equation}
since $\beta>(\alpha-1)/\alpha$.  
Condition on the vector of good/bad row indicators. The retained rows
are independent, and all their entries have the conditional law
\eqref{lp:eq:conditional-law}. Define the event
\[
\mathcal{K}=\{p-K_n\ge\gamma n/2\},
\]
Since $2<\alpha<3$, \eqref{lp:eq:deleted-row-count} gives $\mathbb{P}(\mathcal{K})\to1$. On $\mathcal{K}$, Lemmas~\ref{lem:projection-moments} and \ref{lp:lem:conditional-bias} apply uniformly and give
\begin{equation}\label{lp:eq:retained-mean-variance}
\begin{gathered}
\mathbb{E}[R(\mathbf{X}_{\rm good})\mid K_n]
\le n-p+K_n-c n^{3-\alpha},~~
\operatorname{Var}[R(\mathbf{X}_{\rm good})\mid K_n]\le Cn.
\end{gathered}
\end{equation}
Since deleting rows enlarges the null space, we get
\begin{equation}\label{lp:eq:row-deletion-order}
 R(\mathbf{X})\le R(\mathbf{X}_{\rm good}).
\end{equation}
Suppose first that $2<\alpha<5/2$. Applying conditional Chebyshev's inequality to
\eqref{lp:eq:retained-mean-variance} gives
\[
 \mathbb{P}\left\{
 R(\mathbf{X}_{\rm good})-(n-p+K_n)>-\frac c2 n^{3-\alpha}
 \mid K_n\right\}
 \le C n^{2\alpha-5}=o(1)
\]
on $\mathcal{K}$. By \eqref{lp:eq:deleted-row-count},
$K_n=o_{\mathbb{P}}(n^{3-\alpha})$, which together with
\eqref{lp:eq:row-deletion-order} gives
\[
 \mathbb{P}\{R(\mathbf{X})-m\le-\frac c4 n^{3-\alpha}\}\rightarrow 1.
\]
Consequently, $(R-m)/\sqrt n\pto-\infty$.
The Gaussian reference satisfies
\begin{equation}\label{lp:eq:Gaussian-reference}
 \frac{R^{\rm G}-m}{\sqrt n}
 \xrightarrow{d}N(0,2\gamma(1-\gamma)),
\end{equation}
which further implies $\varepsilon_{n,p}=d_{K}(R,R^{\rm G})\rightarrow 1$.

At the critical point $\alpha=5/2$, define
\[
 \mathcal{R}_n=\frac{R(\mathbf{X})-m}{\sqrt n},\qquad
 \mathcal{R}_n^{\rm good}=\frac{R(\mathbf{X}_{\rm good})-(m+K_n)}{\sqrt n}.
\]
Then $\mathcal{R}_n\le \mathcal{R}_n^{\rm good}+K_n/\sqrt n$ and
$\mathbb{E}K_n=o(\sqrt n)$. By \eqref{lp:eq:retained-mean-variance}, on $\mathcal{K}$ there are fixed constants $c,C>0$ such that
\begin{equation}\label{lp:eq:critical-conditional-bounds}
 \mathbb{E}[\mathcal{R}_n^{\rm good}\mid K_n]=-b_{n,K_n},\qquad
 b_{n,K_n}\ge c>0,\qquad
 \operatorname{Var}(\mathcal{R}_n^{\rm good}\mid K_n)\le C.
\end{equation}
Let $h_T(x)=\max(-T,\min(x,T))$ and choose the fixed number $T$ so
that
\begin{equation}\label{lp:eq:clip-choice}
 T\ge\max\{4C/c,\ 2\sqrt C+c\}.
\end{equation}
For any real random variable $Y$ with $\mathbb{E}Y=-b$, $b\ge c$,
and $\operatorname{Var}(Y)\le C$, we have
\begin{equation}\label{lp:eq:clip-mean}
 \mathbb{E}h_T(Y)\le-c/2.
\end{equation}
In fact, if $b\le T/2$, upper clipping decreases the variable, and hence
\[
 \mathbb{E}h_T(Y)
 \le-b+\mathbb{E}(-T-Y)_+
 \le-b+\frac C{T-b}
 \le-c+\frac{2C}{T}\le-c/2,
\]
where we used the elementary fact that $(v-a)_+\le v^2/a$ for $a>0$ with application to
$v=-(Y+b)$ and $a=T-b$. If $b>T/2$, the Lipschitz property gives
\[
 \mathbb{E}h_T(Y)
 \le h_T(-b)+\mathbb{E}|Y+b|
 \le-T/2+\sqrt C\le-c/2,
\]
which gives \eqref{lp:eq:clip-mean} uniformly in $b\ge c$. 

Apply \eqref{lp:eq:clip-mean} conditionally on $K_n$ on $\mathcal{K}$, and use $|h_T|\le T$ on its complement. Since $\mathbb{P}(\mathcal{K}^c)\to0$, we obtain $\mathbb{E}h_T(\mathcal{R}_n^{\rm good})\le-c/2+o(1)$. Monotonicity and the Lipschitz property of $h_T$ now give
\begin{equation}\label{lp:eq:full-clip-mean}
 \mathbb{E}h_T(\mathcal{R}_n)
 \le\mathbb{E}h_T(\mathcal{R}_n^{\rm good})+\mathbb{E}K_n/\sqrt n
 \le-c/2+o(1).
\end{equation}
By \eqref{lp:eq:Gaussian-reference} and symmetry of its limit, we have
\[
 \mathbb{E}h_T\left(\frac{R^{\rm G}-m}{\sqrt n}\right)\rightarrow 0.
\]
Moreover, since $h_T$ is increasing with total variation $2T$, we have 
\[
|\mathbb{E}h_T\left(\frac{R^{\rm G}-m}{\sqrt n}\right)-\mathbb{E}h_T(\mathcal{R}_n)|\le 2T d_K(R(\mathbf{X}),R^{\rm G})=2T\varepsilon_{n,p},
\]
which follows directly by integrating the
difference of the distribution functions over $[-T,T]$ for $h_T$. 
Consequently, we have 
\[
 \liminf_{n\to\infty}\varepsilon_{n,p}\ge\frac{c}{4T}>0,
\]
completing the proof of 
\eqref{eq:regular-boundary-separation}.  \qed

\subsection{Verification for Example \ref{rem:balanced-tails}}
Recalling the definition of $\xi$ in Example \ref{rem:balanced-tails}, define the signed Laplace transform $g(t)=\mathbb{E}[\xi e^{-t\xi^2}]$ for $t\ge 0$. Since the symmetric component's
contribution to the signed Laplace transform vanishes exactly and $\mathbb{E}\eta=0$, Taylor's expansion gives
\begin{equation}\label{lt:eq:balanced-mixture-g}
 g(t)=\frac1{2\sigma}\mathbb{E}[\eta e^{-t\eta^2/\sigma^2}]
 =\frac{t}{\sigma^3}+O(t^2).
\end{equation}
Next, for $12/5<\alpha\le 5/2$, choose
\begin{equation}
 \frac{3}{2\alpha}<\kappa<\frac1{4-\alpha},\qquad
 \tau_n=n^{\kappa},\qquad \epsilon_n=n^{-3}.
 \label{lt:eq:balanced-mixture-cutoff}
\end{equation}
Define
\[
 \widehat\xi_n=\xi\mathbbm{1}_{\{|\xi|\le\tau_n\}}.
\]
For all sufficiently large $n$, $\widehat\xi_n$ is a mixture of a symmetric truncated heavy-tail component and the unchanged bounded centered component. In particular, one has
\begin{equation}
 \mu_n=\mathbb{E}\widehat\xi_n=0,\qquad
 q_n=\mathbb{P}(|\xi|>\tau_n)=O(n^{-\kappa\alpha}),\qquad
 \mathbb{E}\widehat\xi_n^4=O(1+n^{\kappa(4-\alpha)})=o(n).
 \label{lt:eq:balanced-mixture-truncated-moments}
\end{equation}
Let $U_0\sim\operatorname{Unif}[-1,1]$ be independent and define
\[
 s_n^2=\mathbb{E}\widehat\xi_n^2+\epsilon_n^2/3,\qquad
 \zeta_n=\frac{\widehat\xi_n+\epsilon_nU_0}{s_n}.
\]
Then $s_n\to1$, the law of $\zeta_n$ is non-atomic, centered,
and variance one, and $k_n=\mathbb{E}\zeta_n^4=o(n)$.
Weak convergence to the nondegenerate law of $\xi$ gives the uniform small-ball condition \eqref{eq:anticoncentration} in Lemma \ref{lem:lower}.
Similarly, the signed Laplace transform of $\zeta_n$ depends only on the bounded part $H_n=(\eta/\sigma+\epsilon_nU_0)/s_n$, which is given by
\begin{equation}
\left|\mathbb{E}[\zeta_n e^{-t\zeta_n^2}]\right|
=\frac12\left|\mathbb{E}[H_n(e^{-tH_n^2}-1)]\right|
\le\frac t2\mathbb{E}|H_n|^3\le Ct
\label{lt:eq:balanced-mixture-uniform-laplace}
\end{equation}
uniformly
over sufficiently large $n$ and all $0<t\le1$.
In particular, the uniform signed-Laplace bound
\eqref{sr:eq:signed-laplace-bound} holds with
$\omega(t)=Ct^{1/4}$. The laws also satisfy $k_n=o(n)$ and the common concentration bound. Lemma~\ref{sr:lem:row-mean} therefore applies directly under its signed-transform assumptions since only the signed-transform consequence of its tail condition is required, and the row replacement in Section~\ref{sec:lindeberg} gives \eqref{sym:eq:R-clt} and $\mathbb{E}R(\mathbf{Z})=m+o(\sqrt n)$.  

It remains to transfer the conclusion to the fixed law $\xi$. By \eqref{lt:eq:balanced-mixture-cutoff},
$nq_n\to0$ and $n^2q_n=O(n^{2-\kappa\alpha})=o(\sqrt n)$.
Since $\mu_n=0$, \eqref{sat:eq:transfer} in Lemma \ref{sat:lem:truncation} yields
\[
 \mathbb{E}|R(\mathbf{X})-R(\widehat{\mathbf{X}})|
 \le Cn^2q_n+Cn^C e^{-cn}=o(\sqrt n).
\]
The bounded smoothing perturbation changes $R$ in $L^1$ by
at most $C\epsilon_n n^{3/2}+Cne^{-cn}=o(\sqrt n)$,
by Lemma \ref{lem:lower} and
\eqref{lt:eq:fixed-projection-perturbation}. Moreover, noting $\mu_n=0$ and rescaling by $s_n$ leaves the row space unchanged, we have
\[
 \mathbb{E}|R(\mathbf{X})-R(\mathbf{Z})|=o(\sqrt n).
\]
Therefore, the central limit theorem and mean conclusion for $\mathbf{Z}$ reduce to $\mathbf{X}$ and complete the proof of \eqref{lt:eq:balanced-mixture-F}.

\bibliographystyle{plainnat}
\bibliography{hotelling_F}

@book{Anderson2003,
  author = {Anderson, Theodore W.},
  title = {An Introduction to Multivariate Statistical Analysis},
  publisher = {Wiley},
  address = {Hoboken, NJ},
  year = {2003}
}

@book{Bingham1989,
  author = {Bingham, Nicholas H. and Goldie, Charles M. and Teugels, Jozef L.},
  title = {Regular Variation},
  publisher = {Cambridge University Press},
  address = {Cambridge},
  year = {1989}
}

@article{GineGotze2004,
  author = {Gin{\'e}, Evarist and G{\"o}tze, Friedrich},
  title = {On standard normal convergence of the multivariate {Student} {$t$}-statistic for symmetric random vectors},
  journal = {Electronic Communications in Probability},
  volume = {9},
  pages = {162--171},
  year = {2004},
}

@article{GGM1997,
  author = {Gin{\'e}, Evarist and G{\"o}tze, Friedrich and Mason, David M.},
  title = {When is the {Student} {$t$}-statistic asymptotically standard normal?},
  journal = {The Annals of Probability},
  volume = {25},
  number = {3},
  pages = {1514--1531},
  year = {1997},
}

@article{HeinyParolya2024,
  author = {Heiny, Johannes and Parolya, Nestor},
  title = {Log determinant of large correlation matrices under infinite fourth moment},
  journal = {Annales de l'Institut Henri Poincar\'{e} (B) Probability and Statistics},
  volume = {60},
  number = {2},
  pages = {1048--1076},
  year = {2024},
}

@article{Hotelling1931,
  author = {Hotelling, Harold},
  title = {The generalization of {Student}'s ratio},
  journal = {The Annals of Mathematical Statistics},
  volume = {2},
  number = {3},
  pages = {360--378},
  year = {1931}
}

@misc{LiLiuXieZhou2026,
  author = {Li, Yanpeng and Liu, Zhi and Xie, Jiahui and Zhou, Wang},
  title = {The logarithmic law of sample correlation matrices},
  year = {2026},
  note = {arXiv preprint arXiv:2603.19800},
  eprint = {2603.19800},
  archivePrefix = {arXiv},
  primaryClass = {math.PR}
}

@misc{Li2024clt,
  author = {Li, Yanpeng and Pan, Guangming and Xie, Jiahui and Zhou, Wang},
  title = {Necessary and sufficient condition for {CLT} of linear spectral statistics of sample correlation matrices},
  year = {2024},
  note = {arXiv preprint arXiv:2409.12536},
  eprint = {2409.12536},
  archivePrefix = {arXiv},
  primaryClass = {math.PR}
}

@article{PanZhou2011,
  author = {Pan, Guangming and Zhou, Wang},
  title = {Central limit theorem for {Hotelling}'s {$T^2$} statistic under large dimension},
  journal = {The Annals of Applied Probability},
  volume = {21},
  number = {5},
  pages = {1860--1910},
  year = {2011},
}

@article{RudelsonVershynin2013,
  author = {Rudelson, Mark and Vershynin, Roman},
  title = {{Hanson--Wright} inequality and sub-gaussian concentration},
  journal = {Electronic Communications in Probability},
  volume = {18},
  number = {82},
  pages = {1--9},
  year = {2013},
}

@article{Tikhomirov2016,
  author = {Tikhomirov, Konstantin E.},
  title = {The smallest singular value of random rectangular matrices with no moment assumptions on entries},
  journal = {Israel Journal of Mathematics},
  volume = {212},
  number = {1},
  pages = {289--314},
  year = {2016},
}

@article{auffinger2009poisson,
  author = {Auffinger, Antonio and Ben Arous, G{\'e}rard and P{\'e}ch{\'e}, Sandrine},
  title = {Poisson convergence for the largest eigenvalues of heavy tailed random matrices},
  journal = {Annales de l'Institut Henri Poincar\'{e} (B) Probability and Statistics},
  volume = {45},
  number = {3},
  pages = {589--610},
  year = {2009},
}

@misc{DabagiaFernandez2025,
  author        = {Dabagia, Max and Fernandez, Manuel},
  title         = {The smallest singular value of inhomogeneous random rectangular matrices},
  year          = {2025},
  note          = {arXiv preprint arXiv:2408.14389v2},
  eprint        = {2408.14389},
  archivePrefix = {arXiv},
  primaryClass  = {math.PR},
}

@book{Durrett2019,
  author    = {Durrett, Rick},
  title     = {Probability: Theory and Examples},
  edition   = {5},
  series    = {Cambridge Series in Statistical and Probabilistic Mathematics},
  volume    = {49},
  publisher = {Cambridge University Press},
  address   = {Cambridge},
  year      = {2019}
}

@article{lindeberg1922neue,
  title={Eine neue Herleitung des Exponentialgesetzes in der Wahrscheinlichkeitsrechnung},
  author={Lindeberg, Jarl Waldemar},
  journal={Mathematische Zeitschrift},
  volume={15},
  number={1},
  pages={211--225},
  year={1922},
  publisher={Springer}
}

@article{albrecher2007asymptotic,
  title={Asymptotic analysis of a measure of variation},
  author={Albrecher, Hansj{\"o}rg and Teugels, Jef},
  journal={Theory of Probability and Mathematical Statistics},
  volume={74},
  pages={1--10},
  year={2007}
}
\end{document}